\documentclass[article,12pt]{amsart}
\usepackage{amsfonts,url}
\usepackage{color}
\usepackage[margin=23mm]{geometry}
\usepackage{amssymb, amsmath, amsthm, amscd, accents, mathtools, mathrsfs}
\usepackage{graphicx,times,bbm,url}
\usepackage[T1]{fontenc}
\usepackage{hyperref}
\usepackage{comment}

\title[]{Nonlocal operators with Neumann conditions}

\author[]{Krzysztof Bogdan}
\address{Faculty of Pure and Applied Mathematics,
Wroc\l aw University of Science and Technology,
Wyb. Wyspia\'nskiego 27, 50-370 Wroc\l aw, Poland}
\email{krzysztof.bogdan@pwr.edu.pl}

\author[]{Damian Fafu{\l}a}
\address{Faculty of Pure and Applied Mathematics,
Wroc\l aw University of Science and Technology,
Wyb. Wyspia\'nskiego 27, 50-370 Wroc\l aw, Poland}
\email{damian.fafula@pwr.edu.pl}

\author[]{Pawe{\l} Sztonyk}
\address{Faculty of Pure and Applied Mathematics,
Wroc\l aw University of Science and Technology,
Wyb. Wyspia\'nskiego 27, 50-370 Wroc\l aw, Poland}
\email{pawel.sztonyk@pwr.edu.pl}
    	
\date{\today}

\thanks{The study has been supported by  the Opus grant 2017/27/B/ST1/01339 of the National Science Center (Poland).} 



\newtheorem{lemma}{Lemma}
\numberwithin{lemma}{section}
\newtheorem{theorem}[lemma]{Theorem}
\newtheorem{corollary}[lemma]{Corollary}
\newtheorem{proposition}[lemma]{Proposition}
\theoremstyle{definition}
\newtheorem{remark}[lemma]{Remark}

\DeclareSymbolFont{bbsymbol}{U}{bbold}{m}{n}
\DeclareMathSymbol{\ind}{\mathbin}{bbsymbol}{'061}

\newcommand*\mE{\mathbb{E}}

\newcommand*\mP{\mathbb{P}}
\newcommand{\Rd}{\mathbb{R}^d}
\renewcommand{\P}{\widehat{P}}
\newcommand{\dx}{\mathrm{d}x}
\newcommand{\dy}{\mathrm{d}y}
\newcommand{\dz}{{\rm d}z}
\newcommand{\ds}{{\rm d}s}
\newcommand{\dv}{{\rm d}v}
\newcommand{\E}{\mathcal{E}}
\newcommand{\F}{\mathcal{F}}
\newcommand{\D}{\mathcal{D}}
\newcommand{\R}{\mathbb{R}}

\DeclareMathOperator{\supp}{supp}
\newcommand{\norm}[1]{\left\lVert#1\right\rVert}
\DeclareMathOperator{\dist}{dist}

\newcommand{\dt}{{\rm d}t}
\newcommand{\dw}{{\rm d}w}
\newcommand{\da}{{\rm d}a}
\newcommand{\db}{{\rm d}b}
\newcommand{\dc}{{\rm d}c}
\newcommand{\de}{{\rm d}e}
\newcommand{\du}{{\rm d}u}
\newcommand{\dr}{{\rm d}r}

\numberwithin{equation}{section}

\usepackage{tikz}
\usepackage{float}

\usepackage{fancyhdr}

\usepackage{xpatch}
\makeatletter   
\xpatchcmd{\@tocline}
{\hfil\hbox to\@pnumwidth{\@tocpagenum{#7}}\par}
{\ifnum#1<0\hfill\else\dotfill\fi\hbox to\@pnumwidth{\@tocpagenum{#7}}\par}
{}{}
\makeatother    

\makeatletter
\def\l@section{\@tocline{2}{0pt}{1pc}{6pc}{}}
 \def\l@subsection{\@tocline{2}{0pt}{36pt}{6pc}{}}
\def\l@subsubsection{\@tocline{3}{0pt}{8pc}{8pc}{}}
 \makeatother

\usepackage{graphicx}
\usepackage{scalerel}

\begin{document}

\begin{abstract}
    We construct a strong Markov process corresponding to the Dirichlet form of Servadei and Valdinoci and use the process to solve the corresponding Neumann boundary problem for the fractional Laplacian and the half-line.
\end{abstract}

\maketitle

\counterwithout{equation}{section}
\numberwithin{equation}{section}
\numberwithin{figure}{section}

\tableofcontents

\section{Introduction and Preliminaries}

\subsection{Motivation}

Many nonlocal operators considered in PDEs arise as infinitesimal generators of transition semigroups of jump Markov processes, even L\'evy processes. 
In particular, the fractional Laplacian $(-\Delta)^{\alpha/2}$ is 
the infinitesimal generator of the isotropic
$\alpha$-stable L\'evy process; see Sato \cite[Section 31]{MR1739520}, Kwa\'{s}nicki \cite{MR3613319}, Nezza et al. \cite{DINEZZA2012521}, or Silvestre \cite{Silvestre2007} for more information.

Although they form a narrow class of Markov processes, L\'evy processes are rather representative and play an important role in modeling real-world phenomena: For their applications to financial markets, see Barndorff--Nielsen et al. \cite{barndorff2001levy}, Cont and Tankov \cite{MR2042661}, and Schoutens \cite{schoutens2003levy}. 
For applications to genetics, see, e.g., Blomberg et al. \cite{gen1}, Gjessing et al. \cite{gen3} and Landis et al. \cite{gen2}. For applications to fluid mechanics, solid state physics, and polymer chemistry, see, e.g.,  Barndorff--Nielsen et al. \cite{barndorff2001levy}.

In specific problems, it is important to consider \textit{boundary conditions}. For instance, the Dirichlet problem for $(-\Delta)^{\alpha/2}$ and open set $D\subset\Rd$ is 
\begin{align}\label{eq:Dirichlet_problem}
    \begin{cases}
        (-\Delta)^{\alpha/2} u = f, &\textrm{ in } D, \\
        \phantom{(-\Delta)^{\alpha/2}}u = f, &\textrm{ in } 
    \Rd\setminus \overline{D}.
    \end{cases}
\end{align}
The problem has been studied in various settings  for many years, see, e.g.,  the surveys by Ros--Oton \cite{MR3447732} and Bucur, Valdinoci \cite{bucur2016nonlocal} and the papers by Felsinger et al. \cite{MR3318251}, Rutkowski \cite{AR2018}, and Servadei, Valdinoci \cite{MR2879266, MR3002745}. 
Here we focus on the \textit{variational setting}, 
using quadratic forms and Sobolev-type function spaces.
To this end, we recall the Servadei--Valdinoci form 
\begin{align}\label{e.SVf}
    \E_D(u,v) := \frac{1}{2} \iint_{\Rd\times\Rd \setminus D^c\times D^c} (u(x)-u(y))(v(x)-v(y))\nu(x,y)\,\dx\,\dy,
\end{align}
where $u,v\in L^2(\Rd, \dx)$, and $\nu$ is the integral kernel of $(-\Delta)^{\alpha/2}$; see, e.g.,  \cite{MR3651008}. The form \eqref{e.SVf} first appeared in Servadei and Valdinoci \cite{MR3002745}, see also \cite{MR4088505, MR3318251, MR3900821, MR2879266, MR3002745, Voight2017}. Its main virtue for   the Dirichlet problem  is that (finiteness of) the form imposes minimal smoothness assumptions on the values of $f$ on $D^c$ \cite{MR4088505}. As we shall see below, the form is also  natural in the variational setting of the \textit{Neumann problem}  
\begin{align}\label{eq:Neumann_problem}
        \begin{cases}
            \hfill (-\Delta)^{\alpha/2} u &= ~f, \quad \mathrm{in } ~D, \\
            \hfill \mathcal{N}_{\alpha/2} u&= ~f, \quad \,\mathrm{ in } ~\Rd\setminus {D},
        \end{cases}
\end{align}
introduced by Dipierro et al. \cite{MR3651008}. The Neumann problem is the main subject of this paper. Note that  \eqref{eq:Neumann_problem} is a set of two nonlocal equations because 
the \emph{nonlocal normal derivative}
$\mathcal{N}_{\alpha/2}$ is defined as 
\begin{align}
\label{eq:Nonlocal_derivative}
    \mathcal{N}_{\alpha/2} u(x) := \int_D (u(x)-u(y))\nu(x,y)\,\dy, \qquad x\in \Rd\setminus \overline{D}.
\end{align}
Note that \eqref{eq:Neumann_problem} appeared in Dipierro et al. \cite{MR3651008}; 
see Grube and Hensiek \cite{GH2024} for a recent discussion.
The authors of \cite{MR3651008} introduce the following tentative description of a Markov process corresponding to the Neumann problem \eqref{eq:Neumann_problem}. Consider a particle moving randomly inside \( D \) according to the $\alpha$-stable L\'evy motion. When the particle attempts to exit \( D \) to $x\in D^c$, it immediately returns to \( D \) according to $\mathcal{N}_\alpha$. To wit, the distribution of the return position is proportional to \( \nu(x,y)\dy \) for $y\in D$, so in fact, it is $\nu(x,\dy)/\nu(x,D)$. However suggestive, this description has a problem: The operator
$\mathcal{N}_\alpha$ has a \textit{finite} integral kernel so in reality the particle should spend \textit{positive} time at $x$, before returning to $D$. To resolve this difficulty,  Vondra{\v c}ek \cite{MR4245573} considers the form \eqref{e.SVf} on a suitable \textit{weighted} $L^2$ space on $\Rd$. Upon exit, the particle in the model of Vonra\v{c}ek stays at $x\in D^c$ for an exponential time with mean one. This may be considered as a \textit{change of clock} in the model of Servadei--Valdinoci. Put differently, \cite{MR4245573} replaces $\mathcal{N}_\alpha f(x)$ by $\mathcal{N}_\alpha f(x)/\nu(x,D)$. 

\begin{figure}
~\vfill
\centering 
\begin{minipage}{0.49\textwidth}
\begin{tikzpicture}[scale=1]
\pgftext{\vstretch{1}{\includegraphics[scale=0.57]{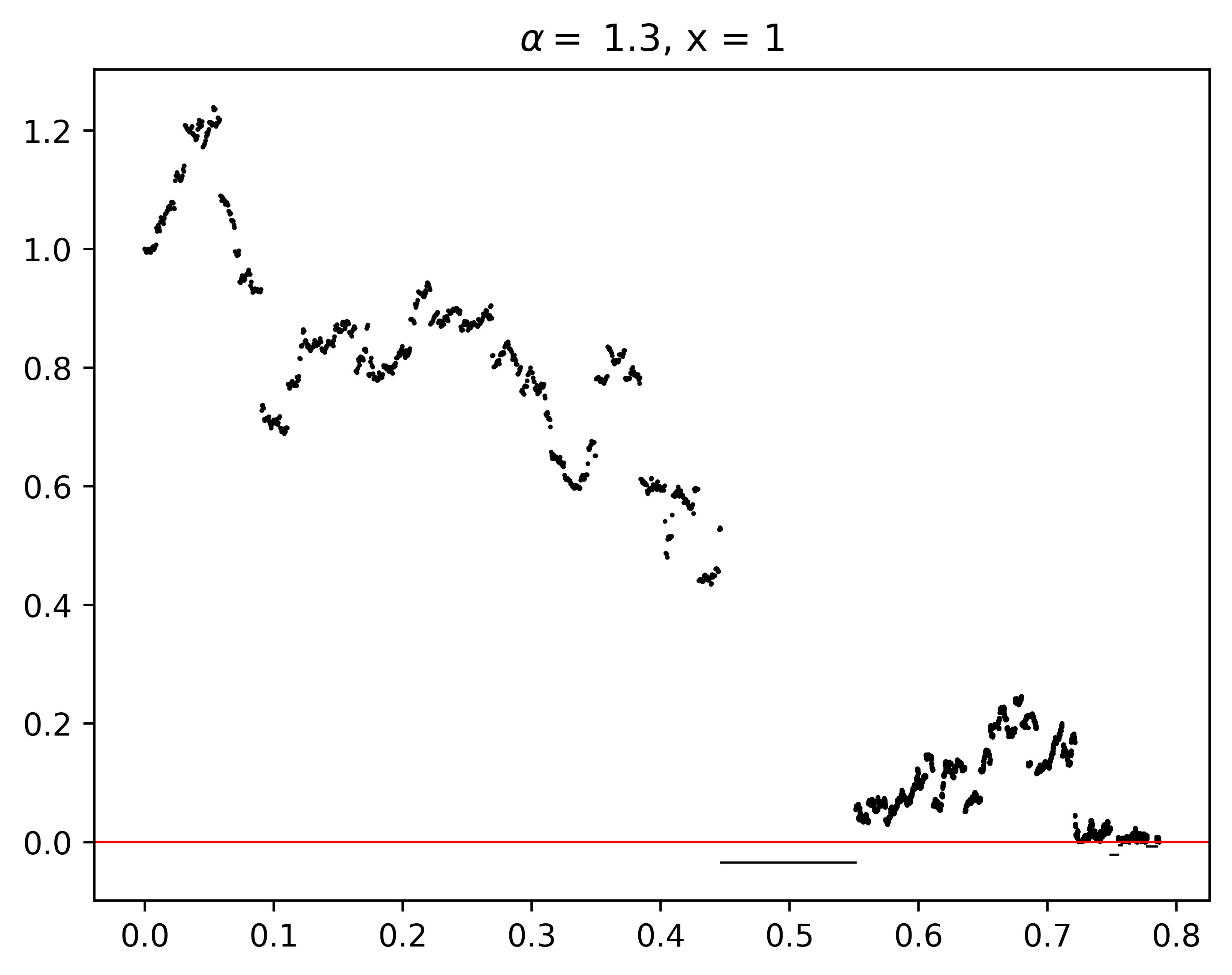}}} at (0pt,0pt);
\draw[draw=blue, thick] (2.85,-2.5) rectangle ++(0.9,0.6);
\end{tikzpicture}
\end{minipage}
\begin{minipage}{0.49\textwidth}
\begin{tikzpicture}[scale=1]
\pgftext{\vstretch{1.2}{\includegraphics[scale=0.57]{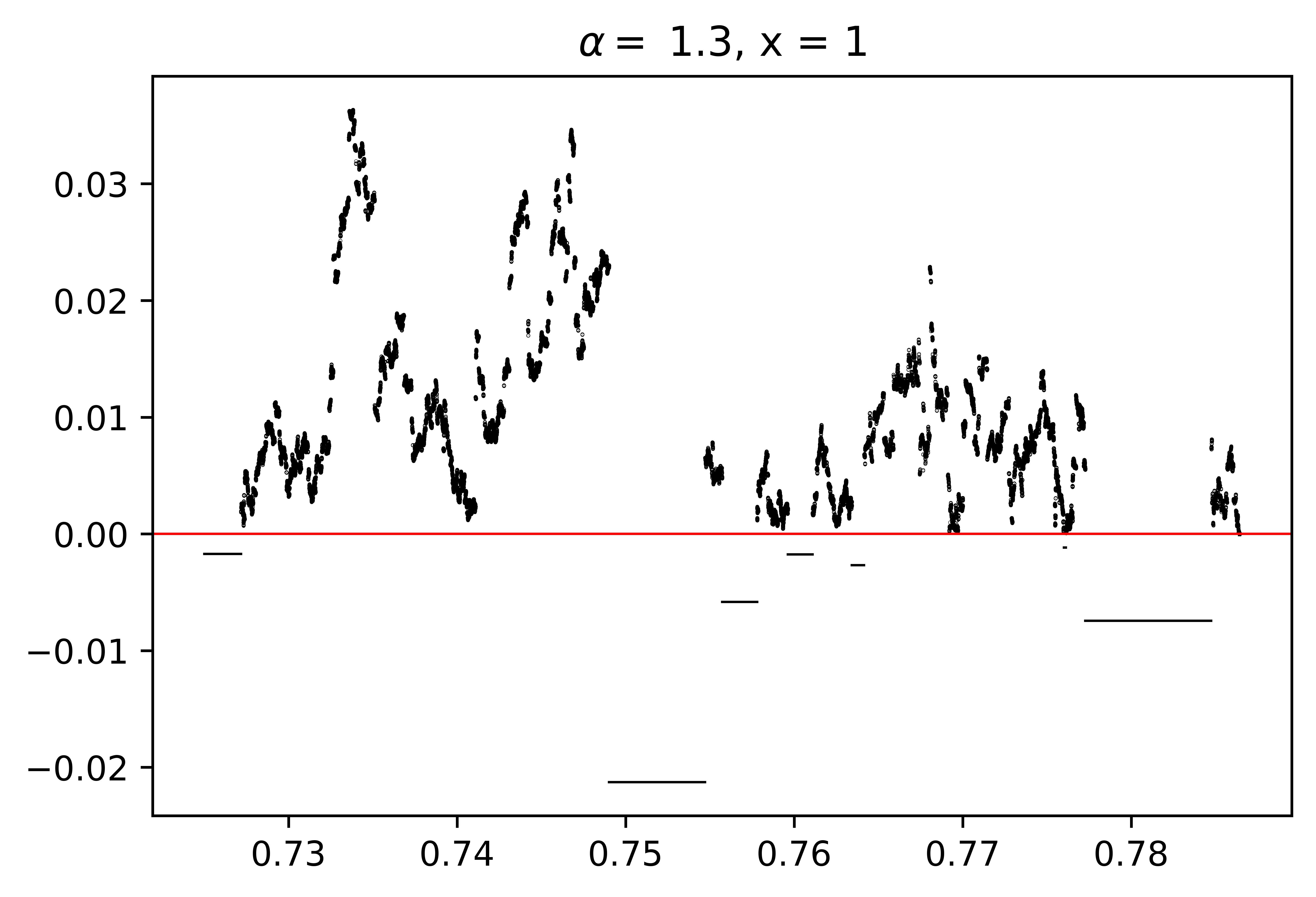}}} at (0pt,0pt);
\end{tikzpicture}
\end{minipage}
\caption{Trajectory of $X$ with $X_0=1$, $\alpha = 1.3$ (left) and its zoom-in at the lifetime (right).}\label{Fig1}

~\vfill
\centering
\begin{minipage}{0.49\textwidth}
\begin{tikzpicture}[scale=1]
\pgftext{\hstretch{1}{\includegraphics[scale=0.6]{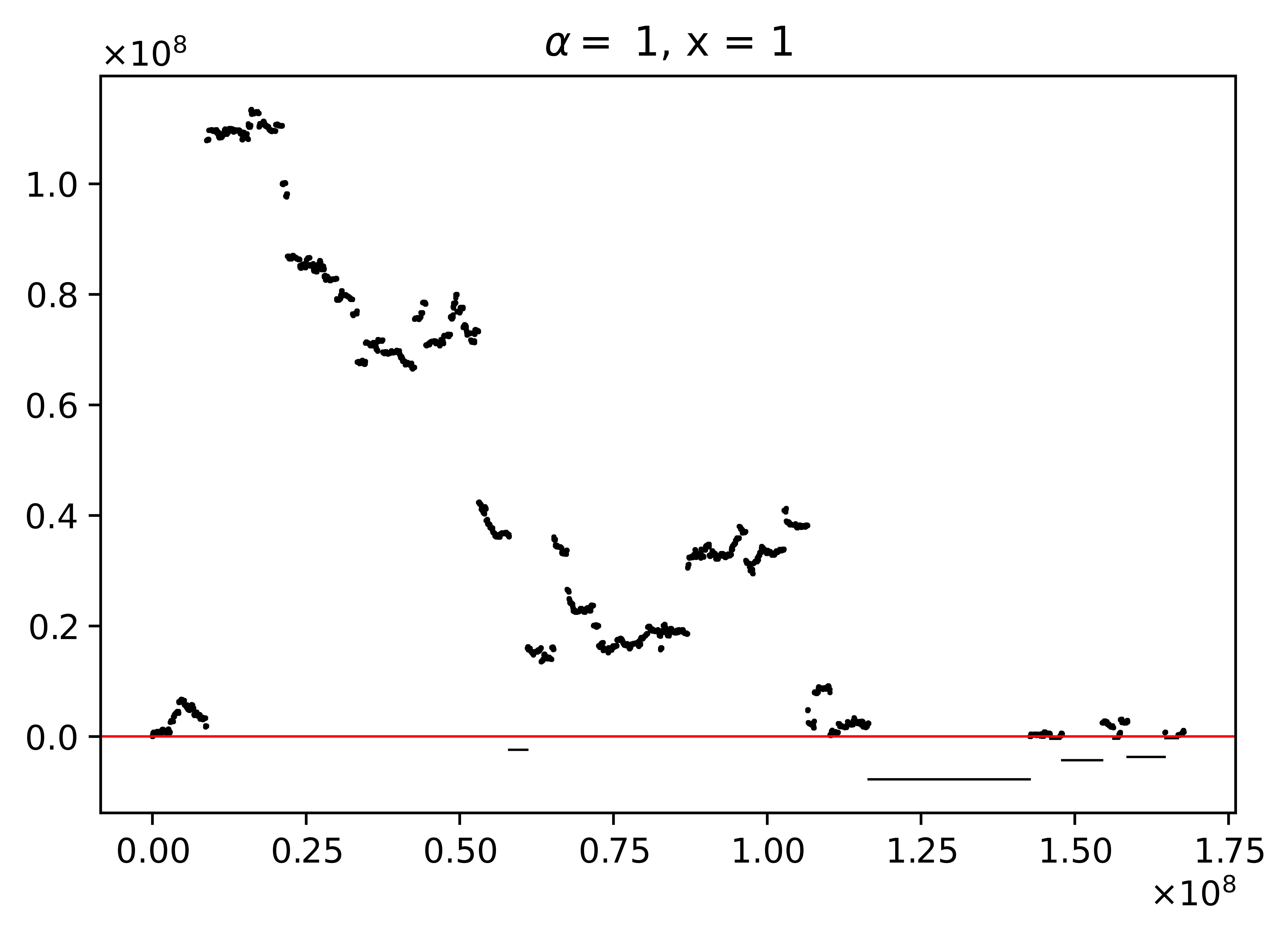}}} at (0pt,0pt);
\end{tikzpicture}
\end{minipage}
\begin{minipage}{0.49\textwidth}
\begin{tikzpicture}[scale=1]
\pgftext{\vstretch{1}{\includegraphics[scale=0.6]{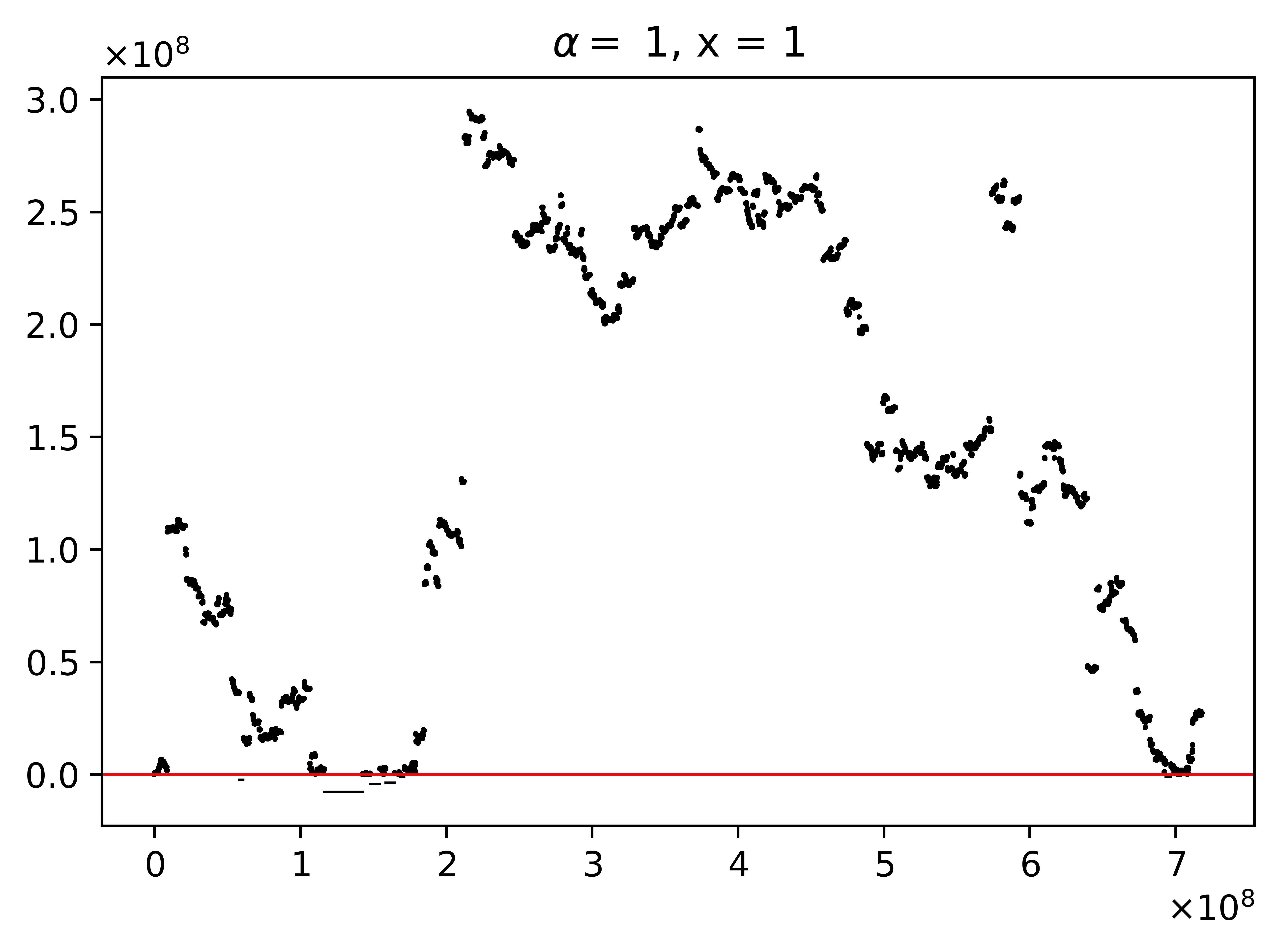}}} at (0pt,0pt);
\draw[draw=blue, thick] (-3.2,-2.05) rectangle ++(1.7,1.9);
\end{tikzpicture}
\end{minipage}
\caption{Trajectory of $X$ with $X_0=1$ for $\alpha = 1$ (left) and its zoom-out (right).}\label{Fig2}

~\vfill
\centering
\begin{minipage}{0.49\textwidth}
\begin{tikzpicture}[scale=1]
\pgftext{\hstretch{1}{\includegraphics[scale=0.6]{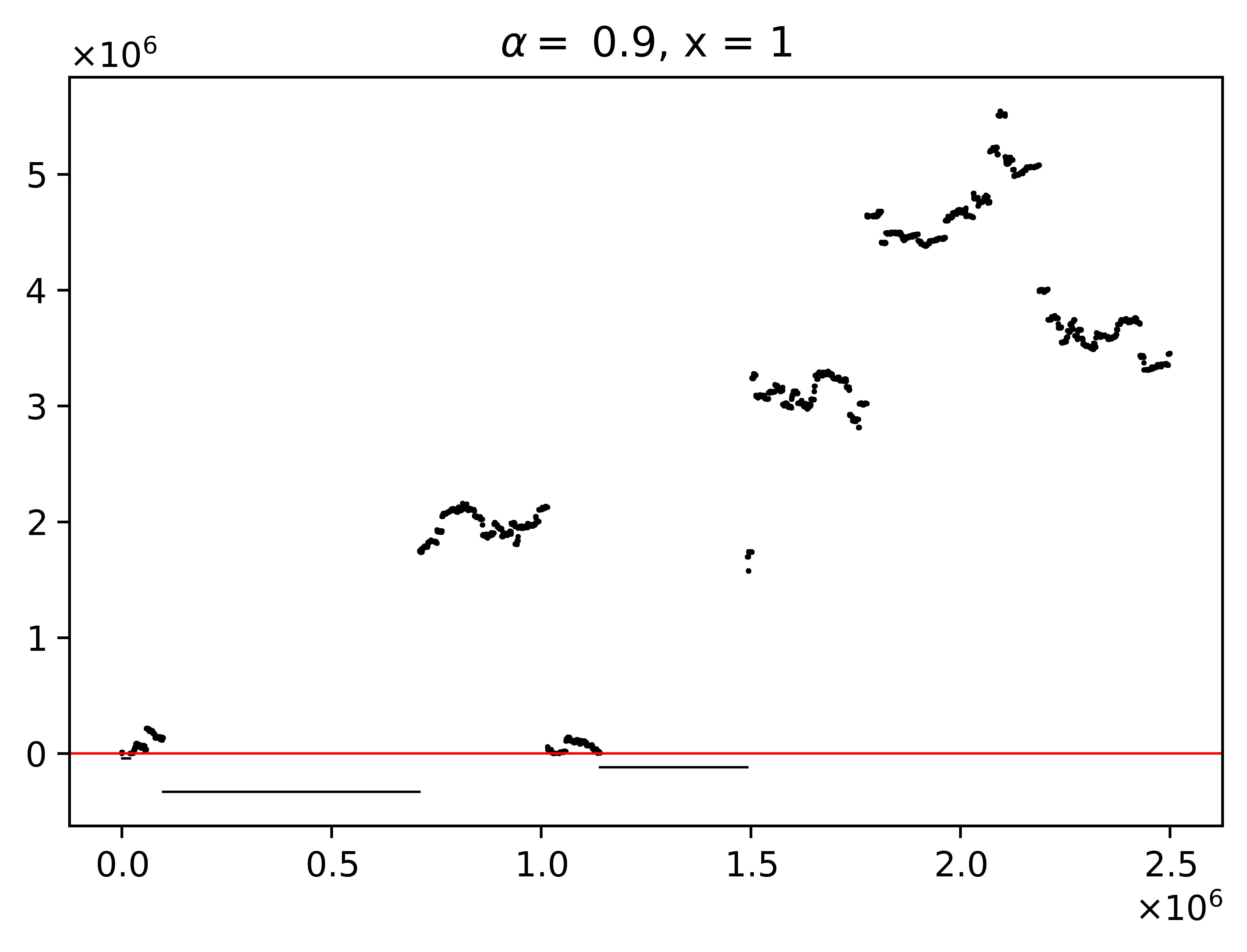}}} at (0pt,0pt);
\end{tikzpicture}
\end{minipage}
\begin{minipage}{0.49\textwidth}
\begin{tikzpicture}[scale=1]
\pgftext{\vstretch{1}{\includegraphics[scale=0.6]{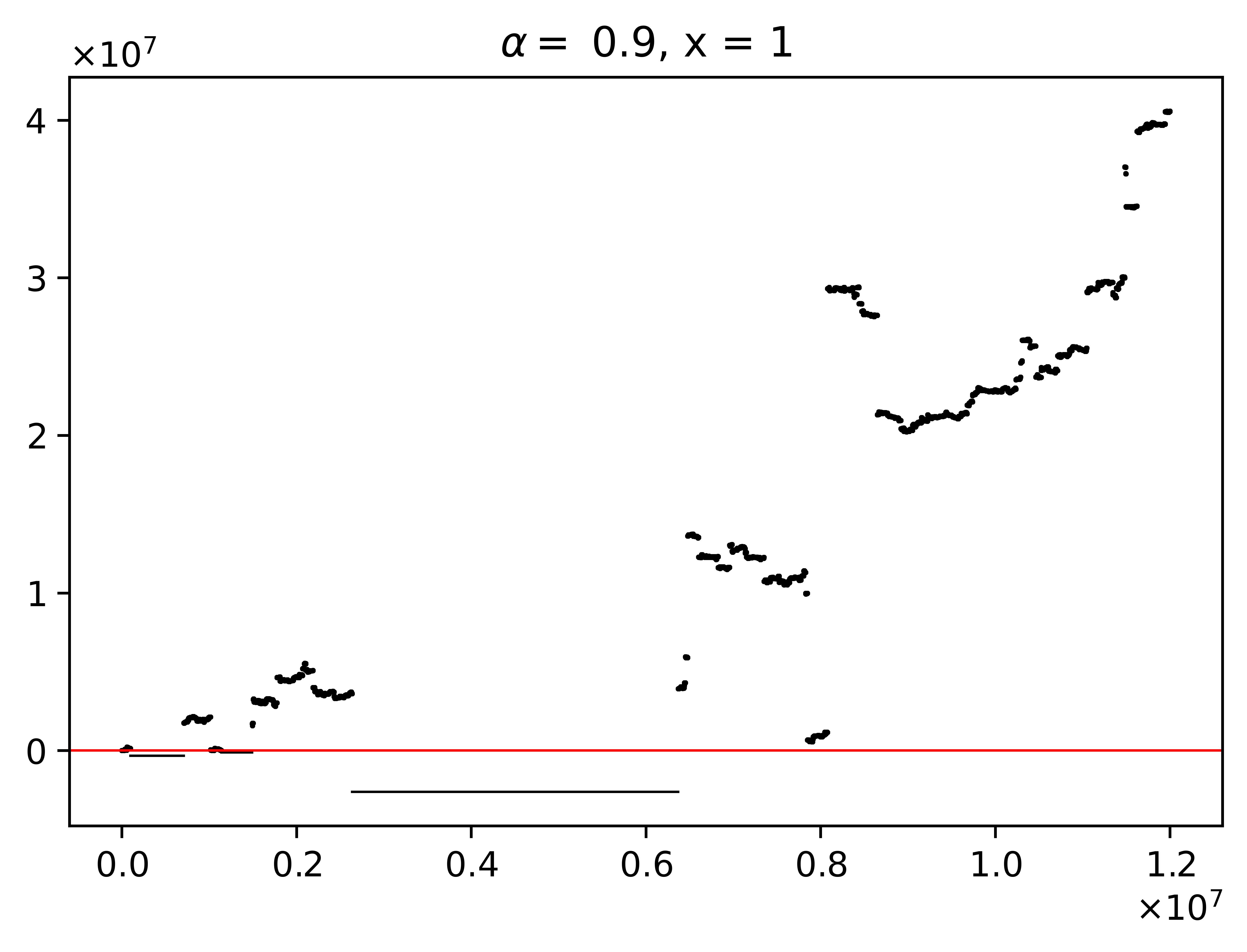}}} at (0pt,0pt);
\draw[draw=blue, thick] (-3.2,-1.825) rectangle ++(1.55,0.8);
\end{tikzpicture}
\end{minipage}
\caption{Trajectory of $X$ with $X_0=1$ for $\alpha = 0.9$ (left) and its zoom out (right).}\label{Fig3}
~\vfill
\end{figure}

Our variant of the Servadei--Valdinoci process $X$ on $\R\setminus\{0\}$ preserves the original nonlocal normal derivative $\mathcal{N}_\alpha$ and the time the process $X$
spends at $x\in \overline{D}^c$ is exponential with intensity $\nu(x,D)$. Thus, starting in $D$, $X$ moves as the isotropic $\alpha$-stable L\'evy process until the first exit time $\tau_D$ from $D$. At this moment, the process jumps from $X_{\tau_D-}\in D$ to a point $y\in \overline{D}^c$ according to the density function $\nu(X_{\tau_D-}, y)/\nu(X_{\tau_D-},D^c)$. Then it spends at $y$ an exponential time with mean $1/\nu(y,D)$, after which it jumps back to $z\in D$ according to density function $\nu(y,z)/\nu(y,D)$, and then starts afresh.
Typical trajectories of the process $X$ are presented in Figures \ref{Fig1}, \ref{Fig2} and \ref{Fig3}.
The main question that arises for this model is whether $X$ can undergo infinitely many reflections in finite time. As we shall see, it depends on $\alpha$.

In what follows, we specialize to the one-dimensional setting of the real line $\R$, so $d:=1$, and let
$$D := (0,\infty),$$ 
the positive open half-line. 
In this setting, we construct the genuine Servadei--Valdinoci  process $X=(X_t)_{t\ge 0}$, investigate the lifetime of $X$, show that $\E_D$ is the Dirichlet form of the process, and give solutions to the Neumann boundary problem \eqref{eq:Neumann_problem}.
Similar results should follow for general \textit{smooth} open sets in $\Rd$---see \cite{MR2006232}---but such extensions are nontrivial and are outside of the scope of the paper.

The structure of the paper is as follows. Below in this section, we state main results and provide necessary definitions. Section \ref{Chap_Servadei_process} presents a construction of $X$ as a strong Markov process by using \textit{concatenation of Markov processes} \cite{MR4247975}. Then we study the transition semigroup \( K = (K_t)_{t \geq 0} \) of $X$ and examine its excessive functions. They yield the lifetime and limiting behavior of $X$, see  Theorem~\ref{lifetime_of_the_process}, which is proved in Section \ref{chap_Lifetime}, Proposition \ref{gen_inequality} describes how the infinitesimal generator of $K$ acts on power functions. In Theorem \ref{nier_Hardyego}, proved in Section \ref{chap_dirichlet}, we discuss the Hardy inequality for the quadratic form $\E$ of $K$. In Section \ref{chap_dirichlet}, we also verify that $\E=\E_D$, propose various characterizations of the domain of the form and, in Proposition \ref{prop:equality_of_domains} and Corollary \ref{corollary_regular}, we show that the form is regular. Theorem~\ref{Neumann_problem}, proved in Section \ref{chap_Neumann}, gives solutions to the Neumann boundary problem  \eqref {eq:Neumann_problem}.

\medskip\medskip
\noindent \textbf{Acknowledgments.} We thank Wolfhard Hansen and Markus Kunze for discussions on the topic of the paper.

\subsection{Main results}
Let us state our main results announced above. 
For $t>0$, we define \[
\mathcal{E}^{(t)}(u,v) := \frac{1}{t}\langle u-K_tu,v\rangle_{L^2(\R)}, \qquad u,v\in L^2(\R).
\]
Let $\E[u]:=\sup_{t>0}\mathcal{E}^{(t)}(u,u)$ for $u\in L^2(\R)$, 
$\F := \{u\in L^2(\R): \E[u]<\infty\}$, and  
$\mathcal{E}(u,v):= \lim_{t\to 0} \mathcal{E}^{(t)}(u,v)$, for $u,v \in  \F$. Note that $\E=\E_D$ on $\F$.

\begin{theorem}
    [Hardy inequality]\label{nier_Hardyego}
    For $u\in L^2(\R)$ and $\alpha\in (0,1)\cup(1,2)$,
    \begin{align}\label{eq:HardyInequality}
    \mathcal{E}[u] \geq (\mathcal{C}_\alpha + \mathcal{D}_\alpha) \int_D u^2(x)|x|^{-\alpha}\,\dx + \mathcal{C}_\alpha \int_{D^c} u^2(x)|x|^{-\alpha}\,\dx,
    \end{align}
    where
    \[
        \mathcal{C}_{\alpha} = \mathcal{A}_{1,\alpha} \Big[\alpha^{-1} - \frac{\big(\Gamma(\tfrac{\alpha+1}{2})\big)^2}{\Gamma(\alpha+1)}\Big]> 0 \qquad and \qquad
        \mathcal{D}_{\alpha} = \mathcal{A}_{1,\alpha} \int_0^1 \frac{(1-t^{(\alpha-1)/2})^2}{(1-t)^{\alpha+1}}\,\dt > 0.
    \]
\end{theorem}

The proof of Theorem~\ref{nier_Hardyego} is given in Section~\ref{sec:dnH}.

Let $R_\infty$ be the lifetime of the Servadei--Valdinoci process $X$. 

\begin{theorem}\label{lifetime_of_the_process}
    The following statements hold $\mathbb{P}_x$-a.s. for every $x\neq 0$.
        \begin{enumerate}
            \item If $\alpha\in (0,1)$ then $R_\infty = \infty$ and $\lim\limits_{t\to\infty} |X_t| = \infty$.
            \item If $\alpha\in (1,2)$ then $0<R_\infty<\infty$ and $\lim\limits_{t\nearrow R_\infty} X_t = 0$.
            \item If $\alpha=1$ then $R_\infty = \infty$ and $\lim\limits_{t\to\infty} |X_t|$ does not exist. In particular, 
            \[
            \liminf\limits_{t\to\infty} |X_t|=0 \quad \mathrm{and} \quad \limsup\limits_{t\to\infty} |X_t|=\infty.
            \]
        \end{enumerate}
\end{theorem}
The result is a close counterpart of Bogdan et al. \cite[Proposition 4.2]{MR2006232} for the censored stable processes. The proof of Theorem \ref{lifetime_of_the_process} is given in Section~\ref{sec:main_theorem}. 

Let $G$ be the 0-potential operator of $K,$ i.e., 
$Gf(x):=\int_0^\infty K_tf(x)\,\dt.$ 

\begin{theorem}\label{Neumann_problem}
Let $\alpha \in (0,1)\cup (1,2)$ and $f\in C_c(\R\setminus\{0\})$. Then, $G|f|<\infty$ on $\R\setminus\{0\}$ and $u = G f$ is the solution of the following Neumann problem for the fractional Laplacian
    \begin{align*}
        \begin{cases}
            (-\Delta)^{\alpha/2} u &= f, \quad \mathrm{in } ~D, \\
            \phantom{\,\,\,\,\,\,\,\,\,} \mathcal{N}_{\alpha/2} u &= f, \quad \mathrm{in } ~\overline{D}^c.
        \end{cases}
    \end{align*}
\end{theorem}
The proof of Theorem \ref{Neumann_problem} is given at the end of the paper.

\subsection{Notation}

For $A\subset \R$, we let $A^c := \R\setminus A$. 
The \emph{open ball} of radius $r>0$ and centered at $x\in\R$  is denoted by $B(x,r) := \{y\in\R:~ |x-y|<r\}$. We let $D = (0,\infty)$ and $\R_* := \R\setminus\{0\}$.
Below we only consider Borel sets, functions, and measures. As usual, $\ln{x} := \log_e{x}$, $x>0$. For $a,b\in\R$, we set $a\wedge b := \min(a,b)$, $a\vee b := \max(a,b)$, $a_+ := a\vee 0$ and $a_- := (-a)\vee 0$.

For functions $f,g\geq 0$, we write $f(x)\lesssim g(x)$ if there exists a number $C\in (0,\infty)$, called \emph{constant}, such that $f(x)\leq Cg(x)$ for all the considered arguments $x$. We denote constants by $C,C_1,C_2,\ldots$ and they may change value from place to place. We are rarely interested in exact values of the constants, but we write $C = C(a_1,\ldots, a_n)$ if the constant $C$ can be so chosen to depend only on $a_1,\ldots, a_n$; we may write $C_a$ if $C=C(a)$. We similarly interpret the inequality $f(x)\gtrsim g(x)$ and we may write $f(x)\approx g(x)$ if $f(x)\lesssim g(x)$ and $f(x)\gtrsim g(x)$. 

If \( \mathcal{U} \subset \R \) is a nonempty open set, then \( \mathcal{B}(\mathcal{U}) \) denotes the class of all the Borel measurable functions on \( \mathcal{U} \). By \( \mathcal{B}^+(\mathcal{U}) \) and \( \mathcal{B}_b(\mathcal{U}) \), we denote the subclasses of \( \mathcal{B}(\mathcal{U}) \) which consist of non-negative and bounded functions, respectively.
Moreover, $\mathcal{B}_b^+(\mathcal{U}) := \mathcal{B}_b(\mathcal{U}) \cap \mathcal{B}^+(\mathcal{U})$. The class of all the real-valued continuous functions on $\mathcal{U}$ will be denoted by $C(\mathcal{U})$. We also consider subclasses of $C(\mathcal{U})$. Namely, $C_b(\mathcal{U})$ is the class of all the bounded continuous functions on $\mathcal{U}$, $C_c(\mathcal{U})$ is the class of all the compactly supported functions on $\mathcal{U}$, $C_c^\infty(\mathcal{U})$ is the class of all smooth compactly supported functions on $\mathcal{U}$, and $C_0(\mathcal{U})$ is the class of all the continuous functions \emph{vanishing at infinity}, i.e.,
\[
C_0(\mathcal{U}) = \{ f\in C(\mathcal{U}):~ \forall \,{\varepsilon>0} ~~\exists \,\mathrm{compact} \, K \subset \mathcal{U} \,\mathrm{such \, that}\, |f(x)|<\varepsilon \,\,\mathrm{for} \,\, x\notin K\}.
\]
In particular, $C_0(\R_*)$ consists of all the continuous functions $f$ on $\R_*$ such that $\lim\limits_{|x|\to\infty} f(x) = 0$ and $\lim\limits_{x\to 0} f(x) = 0$.
Of course, $C_0(\mathcal{U})$ is a Banach space with the supremum norm $\norm{f}_\infty := \sup\limits_{x\in \mathcal{U}} |f(x)|$. By $L^2(\mathcal{U})$, we denote the class of all the (Borel) square integrable functions on $\mathcal{U}$ equipped with the norm $\norm{u}_{L^2(\mathcal{U})} = \sqrt{\int_\mathcal{U} |u(x)|^2\,\dx}$, and we identify functions equal a.e. with respect to $\dx$.

\medskip
Let $(\mathscr{X},\mathcal{A})$ and $(\mathscr{Y}, \mathcal{B})$ be two measurable spaces. As usual (see, e.g.,  Getoor \cite{getoor1975}, Bliedtner and Hansen \cite[Section II.1]{MR850715}), a (\emph{probability} or \emph{Markov}) \emph{kernel} from $(\mathscr{X},\mathcal{A})$ to $(\mathscr{Y},\mathcal{B})$ is a function $K:\mathscr{X}\times \mathcal{B} \to [0,\infty)$ such that 
\begin{enumerate}
    \item[(i)]   for every $B\in\mathcal{B}$, the function $x\mapsto K(x,A)$ is $\mathcal{A}$-measurable,
    \item[(ii)] for every $x\in \mathscr{X}$, the function $B\mapsto K(x,B)$ is a (probability) measure on $(\mathscr{Y}, \mathcal{B})$.
\end{enumerate}
If $K(x,\mathscr{Y}) \leq 1$ for every $x\in \mathscr{X}$, then $K$ is called \emph{subprobability} or \emph{sub-Markov} kernel. By a kernel on $(\mathscr{X},\mathcal{A})$, we simply mean a kernel from $(\mathscr{X},\mathcal{A})$ to $(\mathscr{X},\mathcal{A})$. 

For every kernel $K$, $x\in \mathscr{X}$ and $f\in \mathcal{B}^+(\mathscr{Y})$,
\begin{align}\label{eq:def_oper}
Kf(x) := \int_\mathscr{Y} f(y) K(x,\dy),
\end{align}
defines the corresponding \emph{integral operator} from $\mathcal{B}^+(\mathscr{Y})$ to $\mathcal{B}^+(\mathscr{X})$, which is additive, positively homogeneous and monotone.\footnote{Monotone means here that $f_n, f\in\mathcal{B}^+(\mathscr{X})$ and $f_n\uparrow f$ imply $Kf_n \uparrow Kf$.} Conversely, every additive, positively homogeneous and monotone operator on $\mathcal{B}^+(\mathscr{X})$ is of the form \eqref{eq:def_oper} for some kernel \cite[Section II.1]{MR850715}. If $K$ is subprobability kernel, then \eqref{eq:def_oper} defines a linear bounded operator from $\mathcal{B}_b(\mathscr{Y})$ to $\mathcal{B}_b(\mathscr{X})$. If the kernel $K$ has a Radon--Nikodym density $k$ with respect to a measure $m$ on $\mathscr{Y}$, that is $K(x,A)=\int_A k(x,y)m(\dy)$, $A\in\mathcal{B}$, then we call $k$ a \emph{kernel density}. 

Let $(\mathscr{X},\mathcal{A}), (\mathscr{Y},\mathcal{B})$ and $(\mathscr{Z},\mathcal{C})$ be measurable spaces. Assume that $K$ is a kernel from $(\mathscr{X},\mathcal{A})$ to $(\mathscr{Y},\mathcal{B})$ and $L$ is a kernel from $(\mathscr{Y},\mathcal{B})$ to $(\mathscr{Z},\mathcal{C})$. Then the  \emph{composition} of kernels $K$ and $L$ is defined as the following kernel $KL$ from $(\mathscr{X},\mathcal{A})$ to $(\mathscr{Z},\mathcal{C})$,
\[
(KL)(x,C) := \int_\mathscr{Y} K(x,\dy)L(y,C),
\]
for $x\in \mathscr{X}$, $C\in\mathcal{C}$. If $K$ and $L$ are probability (resp. subprobability) kernels, then $KL$ is a probability (resp. subprobability) kernel. 

A family $(T_t)_{t\geq 0}$ of probability (resp. subprobability) kernels on $(\mathscr{X},\mathcal{A})$ is called a \emph{probability} (resp. \emph{subprobability}) \emph{transition kernel} if the kernels satisfy the semigroup property $T_{t+s} = T_tT_s$, $t,s\geq 0$. The corresponding family of integral operators, also denoted by $(T_t)_{t\geq 0}$, is called \emph{Markov} (resp. \emph{sub-Markov}) \emph{semigroup of operators}.

A sub-Markov semigroup $(T_t)_{t\geq 0}$ of operators is called a \emph{Feller semigroup} (see, e.g.,  Kallenberg \cite[p. 369]{MR1876169}) if for $0\leq f \leq 1$, we have $0\leq T_tf\leq 1$, $t\geq 0$, and 
\begin{enumerate}
    \item[(P1)] for $f\in C_0(\mathscr{X})$ and $t\geq 0$, $T_tf\in C_0(\mathscr{X})$,
    \item[(P2)] for $f\in C_0(\mathscr{X})$ and $x\in \mathscr{X}$, $T_tf(x)\to f(x)$, as $t\to 0^+$.
\end{enumerate}
It is well known (see \cite[Chapter 19]{MR1876169}), that (P1), (P2) and the semigroup property yield the strong continuity of $(T_t)_{t\geq 0}$: $\lim\limits_{t\to 0^+} \norm{T_tf-f}_\infty = 0$ for $f\in C_0(\mathscr{X})$.

We also consider \emph{strong Feller} semigroups of operators $(T_t)_{t\geq 0}$, i.e. semigroups with the property $T_t\mathcal{B}_b(\mathscr{X})\subset C_b(\mathscr{X})$, $t\geq 0$. Furthermore, \emph{double Feller semigroups} of operators mean the semigroups with both Feller and strong Feller property (see, e.g., Chen and Kuwae \cite{MR2604914}). 

\medskip
The \emph{Euler beta} function may be defined by
\[
\mathfrak{B}(x,y) := \int_0^\infty \frac{t^{x-1}}{(1+t)^{x+y}}\,\dt, \qquad x>0, ~y>0.
\]
It is well known that
\[
\mathfrak{B}(x,y) = \frac{\Gamma(x)\Gamma(y)}{\Gamma(x+y)},
\]
where $\Gamma$ denotes the Euler gamma function.

\subsection{Isotropic \texorpdfstring{$\alpha$}{}-stable L\'evy process}\label{sec:Isotropic_stable_process}

In this section, we introduce the isotropic $\alpha$-stable L\'evy process (as a general reference, see Kwa\'{s}nicki \cite{MR3613319}).

Let $\alpha\in (0,2)$. For $t>0$, let 
\begin{align}\label{p_t_Fourier}
    p_t(x) := \frac{1}{2\pi} \int_\R e^{-i\xi x}e^{-t|\xi|^\alpha}\,\mathrm{d}\xi, \quad t>0, ~x\in\R.
\end{align}
From \eqref{p_t_Fourier}, we see that the function $x\mapsto p_t(x)$ is smooth, being the inverse Fourier transform of a rapidly decreasing function. Furthermore, it is well known that $p_t\geq 0$, $\int_\R p_t(x)\,\dx = 1$,
\begin{align}\label{eq:scaling_pt}
    p_t(x) = t^{-1/\alpha}p_1(xt^{-1/\alpha}), \quad t>0, ~x\in\R,
\end{align}
and
\begin{align}\label{eq:p_t_approx}
    p_t(x) \approx t^{-1/\alpha} \wedge \frac{t}{|x|^{1+\alpha}}, \quad t>0, ~x\in\R,
\end{align}
see Blumenthal and Getoor \cite{MR119247} or P{\`o}lya \cite{Polya1923}; see also \cite{MR3357585}.
For $x,y\in\R$, we denote $p_t(x,y) := p_t(y-x)$. By \eqref{p_t_Fourier},
\[
\int_\R p_s(x,y)p_t(y,z)\,\dy = p_{t+s}(x,z), \quad s,t\geq 0, ~x,z\in\R.
\]
It is a kernel density and defines the transition kernel,
\[
P_t(x,A) := \int_A p_t(x,y)\,\dy, \quad x\in\R,~ A\subset \R.
\]
Of course, the corresponding Markov semigroup of operators is defined by
\begin{align*}
    P_tf(x) = \int_\R f(y)p_t(x,y)\,\dy, \quad f\in\mathcal{B}_b(\R)\cup \mathcal{B}^+(\R) \cup L^2(\R).
\end{align*}
It has the fractional Laplacian as infinitesimal generator (see Kwa\'{s}nicki \cite{MR3613319}):
\begin{align}\label{eq:fractional_laplacian}
     -(-\Delta)^{\alpha/2}\varphi(x) &=\lim_{t\to 0^+} \frac{P_t\varphi(x) - \varphi(x)}{t} =  \mathcal{A}_{1,\alpha} ~\mathrm{p.v.} \int_\R \frac{\varphi(x+y)-\varphi(x)}{|y|^{\alpha+1}}\,\dy
     \\
     &:= \mathcal{A}_{1,\alpha} \lim_{\varepsilon\to 0^+} \int_{|y|>\varepsilon} \frac{\varphi(x+y)-\varphi(x)}{|y|^{\alpha+1}}\,\dy, \quad x\in\R, \quad \varphi\in C_c^\infty(\R), \nonumber
\end{align}
where
\[
\mathcal{A}_{1,\alpha} := \frac{2^\alpha \Gamma((\alpha+1)/2)}{\sqrt{\pi} |\Gamma(-\alpha/2)|}.
\]

The \emph{isotropic $\alpha$-stable L\'evy process} $\big( (Y_t)_{t\geq 0}, (\mP_x^Y)_{x\in\R}\big)$ on $\R$ may be defined as a c\`adl\`ag Markov process satisfying $\mP_x^Y(Y_0 = x) = 1$ with the transition probability:
\[
P_t^Y(x,A) := P_t(x,A), \quad t>0, ~x\in\R, ~A\subset\R.
\]
Here and in what follows, $\mP_x^Y$ and $\mE_x^Y$ denote the distribution and expectation for the process $Y$ starting from $x\in\R$, respectively. Of course, $(Y_t)_{t\geq 0}$ has independent increments;  see \eqref{p_t_Fourier}. It is well known that $(Y_t)_{t\geq 0}$ is a strong Markov process with respect to the standard filtration \cite{MR0264757}. Moreover, for $a>0$, $aY_t \stackrel{\mathrm{d}}{=} Y_{a^\alpha t}$ as processes (in distribution with respect to $\mP_0^Y$), which follows from \eqref{eq:scaling_pt}.

For (Borel) set $B\subseteq \R$, we let 
\[
\tau_B := \inf\{t\geq0:~ Y_t\notin B\},
\]
be the time of the first exit of $Y$ from $B$. If $r>0$ and $t>0$, then, with respect to $\mP_0^Y$,
\begin{align}\label{eq:pair_scling}
    (Y_t, \tau_{B(0,r)}) &= \big( Y_t, \,\inf\{s\geq 0: ~ |Y_s|\geq r\}\big) \nonumber \\
    &= \big( r\cdot \tfrac1r Y_t, \, \inf\{s\geq 0:~ |\tfrac1r Y_s| \geq 1\}\big) \nonumber \\
    &\stackrel{\mathrm{d}}{=} \big( r\cdot Y_{r^{-\alpha}t},  \, \inf\{s\geq 0:~ |Y_{r^{-\alpha}s}|\geq 1\}\big) \nonumber \\
    &=  \big( r\cdot Y_{r^{-\alpha}t},  \, r^\alpha \inf\{r^{-\alpha}s\geq 0:~ |Y_{r^{-\alpha}s}|\geq 1\}\big) \nonumber \\
    &= \big( r\cdot Y_{r^{-\alpha}t},  \, r^\alpha \tau_{B(0,1)} \big).
\end{align}
Similarly we prove that, with respect to $\mP_0^Y$,
\begin{align}\label{eq:scaling_Y_tau}
    Y_{\tau_{B(0,r)}} \stackrel{\mathrm{d}}{=} rY_{\tau_{B(0,1)}}.
\end{align}

\medskip
The measure $\nu(\dx) := \nu(x)\,\dx$, with the density function
\begin{align}\label{eq:nu_def}
    \nu(x) := \mathcal{A}_{1,\alpha} |x|^{-1-\alpha}, \quad x\in\R_*,
\end{align}
is the L\'evy measure of $Y$. In particular, $\int_\R (1\wedge |x|^2)\nu(\dx)<\infty$. As in the case of $p_t$, we slightly abuse the notation and write $\nu(x,y) :=\nu(y-x)$. Furthermore, for $x\in\R$, $A\subset\R$, we let 
\begin{align}\label{eq:nu(x,A)}
    \nu(x,A) := \nu(A-x) = \int_A \nu(x,y)\,\dy.
\end{align}

Recall that $D = (0,\infty) \subset\R$. For later use, we note that
\begin{align}\label{nu_scaling3}
    \nu(kx,ky) = k^{-\alpha-1}\nu(x,y), \quad k>0, ~x,y\in\R,
\end{align}
\begin{align}\label{nu_scaling2}
    \nu(kx,D) = k^{-\alpha} \nu(x,D), \quad k>0, ~x\in\R,
\end{align}
\begin{align}\label{nu_scaling}
    \nu(x,D) = \mathcal{A}_{1,\alpha} |x|^{-\alpha} \int_0^\infty \frac{\dy}{(1+y)^{\alpha+1}} =\mathcal{A}_{1,\alpha} \alpha^{-1}|x|^{-\alpha}, \qquad x<0,
\end{align}
and
\begin{align}\label{nuDc_scaling}
    \nu(x,D^c) = \nu(-x,D) = \mathcal{A}_{1,\alpha} \alpha^{-1} x^{-\alpha}, \qquad x>0.
\end{align}
Of course, $\nu(0,D) = \infty$.


\subsection{Potential theory}
Using $\tau_D$ we define the transition density of the \emph{process killed} upon leaving $D$: 
\begin{align}
\label{eq:Hunt's_formula}
    p_t^D(x,y) = p_t(x,y) - \mE_x^Y \big[ \tau_D<t; ~p_{t-\tau_D}(Y_{\tau_D},y)\big], \quad t>0, ~x,y\in\R.
\end{align}
The above is called the \emph{Hunt formula} (references for this section and more details can be found in \cite{MR2602155}, \cite{104171}, or \cite[Chapter 2]{MR1329992}). We have the scaling property
\begin{equation}\label{eq:scaling}
    p_t^D(x,y) = t^{-1/\alpha} p_1^D(t^{-1/\alpha} x, t^{-1/\alpha} y), \quad x,y\in\R, ~t>0,
\end{equation}
or, equivalently,
\begin{equation}\label{eq:scaling2}
    p_{k^\alpha t}^D(kx,ky) = k^{-1} p_t^D(x,y), \quad x, y\in\R, ~k, t>0.
\end{equation}
We note that $p^D$ is jointly continuous in $t,x,y$. This follows from the scaling property \eqref{eq:scaling} and the continuity of the function $(x,y) \mapsto p_t(x,y)$ by the same proof as in Theorem 2.4 in Chung and Zhao \cite{MR1329992}. Moreover, for $x,y\in\R$ and $t>0$, $p^D$ is \emph{symmetric}, i.e. $p^D_t(x,y) = p^D_t(y,x)$, 
\begin{align}\label{eq:killed_leq_stable}
    0\leq p_t^D(x,y) \leq p_t(x,y),
\end{align}
and $p_t^D(x,y)=0$ whenever $x\leq 0$ or $y\leq 0$. Furthermore, $p^D$ satisfies the Chapman--Kolmogorov equations:
\begin{align}\label{eq:p^D_Ch-K}
    p_{t+s}^D(x,y) = \int_\R p_t^D(x,z)p_s^D(z,y)\,\dz,\quad x,y\in \R, ~s,t>0.
\end{align}
The function $p^D$ is called \emph{Dirichlet heat kernel} of the half-line $D$ for the fractional Laplacian, or transition density of the isotropic $\alpha$-stable process killed when leaving $D$, because
\begin{align}\label{eq:p_D_distribution}
    P_t^D(x,A) := \int_A p_t^D(x,y)\,\dy = \mP_x^Y(Y_t\in A, \tau_D>t), \quad t>0, ~x\in\R, ~A\subset\R.
\end{align}
For $x\in\R$, $t>0$, and nonnegative or bounded functions $f$, we as usual denote
\begin{align*}
    P_t^Df(x) := \int_\R f(y)P_t^D(x,y)\,\dy = \mE_x^Y[f(Y_t); \tau_D>t].
\end{align*}
We also let $P_0^D = I$, the identity operator. 

By Bogdan and Grzywny \cite{MR2602155},
\begin{equation}\label{eq:asymptotic}
    p_t^D(x,y) \approx \mP_x^Y(\tau_D>t)\mP_y^Y(\tau_D>t) p_t(x,y), \quad t>0, ~x,y>0,
\end{equation}
and
\begin{align}\label{eq:asymp_pstwo}
    \mP_x^Y(\tau_D>t) \approx 1\wedge \frac{|x|^{\alpha/2}}{\sqrt{t}}, \quad t>0, ~ x>0.
\end{align}
Thus,
\begin{align}\label{eq:approx_p_D}
    p_t^D(x,y) \approx \Big( 1 \wedge \frac{|x|^{\alpha/2}}{\sqrt{t}}\Big) \Big( 1 \wedge \frac{|y|^{\alpha/2}}{\sqrt{t}}\Big)\Big( t^{-1/\alpha} \wedge \frac{t}{|x-y|^{\alpha+1}}\Big), \quad t>0, ~x,y>0.
\end{align}

\medskip
The \emph{Green function} of $D$ is 
\begin{align}\label{eq:green_fun}
    G_D(x,y) := \int_0^\infty p_t^D(x,y)\,\dt, \quad x,y\in \R.
\end{align}
Obviously, $G_D(x,y) = 0$ whenever $x\leq 0$ or $y\leq 0$.
For each function $f\geq 0$, from Tonelli's theorem and \eqref{eq:p_D_distribution}, we have
\begin{align*}
    \int_D G_D(x,y)f(y)\,\dy = \mE_x^Y \int_0^{\tau_D} f(Y_t)\,\dt, \quad x>0,
\end{align*}
hence $G_D(x,y)$ is the occupation time density of the process $Y$ prior to its first exit from $D$. By \eqref{eq:scaling2},
\begin{align}\label{eq:green_scaling}
    G_{D}(x,y) = x^{\alpha-1} G_{D}(1,y/x), \qquad x,y>0.
\end{align}

It is well known (see \cite{MR0321198}, Corollary 3.5) that $\mP_x(Y_{\tau_D}=0)=0$, for every $x>0$, and from \cite{MR142153} or \cite{MR3737628}, the joint distribution of the triple $(\tau_D, Y_{\tau_D-}, Y_{\tau_D})$ is given by the \emph{Ikeda--Watanabe formula}: 
\begin{align}\label{eq:Ikeda_Watanabe}
    \mP_x^Y\big( \tau_D\in I, ~Y_{\tau_D-}\in A, Y_{\tau_D}\in B\big) = \int_I \dt \int_A\dy \int_B \dz \,p_t^D(x,y)\nu(y,z), \quad x>0.
\end{align}
By \eqref{eq:p_D_distribution} and \eqref{eq:Ikeda_Watanabe}, we thus have
\begin{align}\label{eq:p_survival}
    P_{t}^D(x,D) = \int_D p_t^D(x,y)\,\dy = \mP^Y_{x}(\tau_D > t) = \int_t^\infty \dr \int_D \da \int_{D^c} \db \, p_{r}^D(x,a)\nu(a,b).
\end{align}

We further define the \emph{Poisson kernel} of $D$ for the fractional Laplacian:
\begin{align}\label{eq:poisson_kernel}
    P_D(x,y) = \int_D G_D(x,z)\nu(z,y)\,\dz, \quad x>0, ~y\leq 0.
\end{align}
Then, for $A\subset (-\infty, 0]$,
\begin{align*}
    \mP_x^Y(Y_{\tau_D}\in A) = \int_A P_D(x,y)\,\dy, \quad x>0.
\end{align*}
Using \eqref{eq:green_scaling}, we obtain the following scaling property of the Poisson kernel:
\begin{align}\label{poisson_kernel_scaling}
    P_{D}(x,y) = x^{-1} P_{D}(1,y/x), \qquad x>0, ~y\leq 0.
\end{align}
In fact, the Poisson kernel of the half-line is known explicitly (see \cite[(3.40)]{MR1704245}):
\begin{align}\label{eq:poisson_kernel_halfline}
    P_{D}(x,y) = \frac{1}{\pi}\sin\Big(\frac{\pi\alpha}{2}\Big) \frac{x^{\alpha/2}}{|y|^{\alpha/2}|x-y|}, \quad x>0, ~y\leq 0.
\end{align}

The following results assert the double Feller property of $P^D$. For the convenience of the reader, we give short proofs.

\begin{lemma}
    \label{theorem_p^D_feller_1}
    For every $t>0$, we have $P_t^D\mathcal{B}_b(D)\subset C_b(D)$.
\end{lemma}

\begin{proof}
    Let $f\in \mathcal{B}_b(D)$. We define $f=0$ on $D^c$. For $t>0$, the function $D\ni x\mapsto P_tf(x) = p_t*f(x)$ is continuous, because $p_t\in L^1(\R)$ and $f\in L^\infty(\R)$. From Vitali's theorem (see Schilling \cite[Theorem 16.6]{Schilling2}), it follows that the functions $y\mapsto p_t(x,y)f(y)$ are uniformly in $x$ integrable with respect to the measure $\ind_D(y)\,\dy$. From \eqref{eq:killed_leq_stable} it follows that $y\mapsto p_t^D(x,y)f(y)$ are also uniformly integrable (see \cite[Definition 16.1 or Theorem 16.8]{Schilling2}). Since the mapping $D\ni x\mapsto p_t^D(x,y)$ is continuous, Vitali's theorem gives the continuity of the function $D\ni x \mapsto P_t^Df(x)$. Of course, $\lVert P_t^Df\rVert_\infty \leq \norm{f}_\infty$, by \eqref{eq:killed_leq_stable}.
\end{proof}

\begin{lemma}\label{theorem_p^D_feller}
$(P_t^D)_{t\geq 0}$ is a Feller semigroup on $C_0(D)$.
\end{lemma}

\begin{proof}
We have $P_t^DC_0(D) \subset C_0(D)$ for every $t>0$. Indeed, let $f\in C_0(D)$. The continuity of $P_t^Df$ follows from Lemma \ref{theorem_p^D_feller_1}. 

Let $x\to +\infty$. By \eqref{eq:killed_leq_stable},
\[
|P_t^Df(x)| \leq \int_D p_t^D(x,y) |f(y)|\,\dy \leq \int_D p_t(x,y) |f(y)|\,\dy \to 0,
\]
which follows from the Feller property of $P_t$ (see \cite[Theorem 19.10]{MR1876169}).

Moreover, let $x\to 0^+$. By \eqref{eq:p_survival} and \eqref{eq:asymp_pstwo},
\[
|P_t^Df(x)| \leq \int_D p_t^D(x,y) |f(y)|\,\dy \leq \norm{f}_\infty p_t^D(x,D) = \norm{f}_\infty \mP_x^Y(\tau_D > t) \approx 1\wedge \frac{|x|^{\alpha/2}}{\sqrt{t}} \to 0,
\]
as needed.

Next, we will verify that $P_t^Df\to f$ in $C_0(D)$, as $t\to 0^+$. For $x>0$, by \eqref{eq:killed_leq_stable},
\begin{align*}
    \big| P_t^Df(x) - f(x)\big| &= 
    \Big| \int_D (f(y)-f(x))p_t^D(x,y)\,\dy - f(x)(1-p_t^D(x,D)) \Big| \\
    &\leq  \int_D |f(y) - f(x)| p_t(x,y)\,\dy + \norm{f}_\infty \big| 1 - p_t^D(x,D)\big|.
\end{align*}
From \eqref{eq:p_survival}, $p_t^D(x,D)\to 1$ as $t\to 0^+$. 

The estimate of the remaining integral is the same as in the proof of Theorem 1.7 in \cite{MR1329992}: Let $\varepsilon>0$. Since $f\in C_0(D)$, there is $\delta>0$ such that $|f(y)-f(x)|<\varepsilon$ for $|y-x|<\delta$. Therefore, from \eqref{eq:p_t_approx} we get
\begin{align*}
    \int_D &|f(y) - f(x)| p_t(x,y)\,\dy \\
    &\leq \varepsilon \int_{D\cap \{|y-x|<\delta\}} p_t(x,y)\,\dy + \int_{D\cap \{|y-x|\geq\delta\}} |f(y)-f(x)|p_t(x,y)\,\dy \\
    &\leq \varepsilon P_t(x,\R) + 2\norm{f}_\infty \int_{D\cap \{|y-x|\geq\delta\}}p_t(x,y)\,\dy \\
    &\leq \varepsilon + 2C\norm{f}_\infty t\int_{|w|\geq\delta} |w|^{-\alpha-1}\,\dw \\
    &= \varepsilon + \frac{4C}{\alpha} \norm{f}_\infty t \delta^{-\alpha}. \qedhere
\end{align*}
\end{proof}


\subsection{Markov processes and the strong Markov property}

In what follows, we work with Markov processes, so we introduce the basic notions for such processes. Our notation is similar to that of  Werner in \cite{MR4247975}, since we use his results below.

Let $\mathcal{U}\subset\R$ be an open set. A \emph{Markov process} $X$ on $\mathcal{U}$ is defined as a following sextuple (see, e.g.,  \cite{MR0264757}, \cite{MR958914}):
\[
X = (\Omega, \widetilde{\mathcal{F}}, (\mathcal{F}_t)_{t\geq 0}, (X_t)_{t\geq 0}, (\theta_t)_{t\geq 0}, (\mP_x)_{x\in\mathcal{U}}),
\]
where $(X_t)_{t\geq 0}$ is a right continuous, $\mathcal{U}$-valued stochastic process on a measurable space $(\Omega,\widetilde{\mathcal{F}})$, adapted to filtration $(\mathcal{F}_t)_{t\geq 0}$ and equipped with shift operators $(\theta_t)_{t\geq 0}$ on $\Omega$. Moreover, $(\mP_x)_{x\in\mathcal{U}}$ is a family of sub-probability measures such that $\mP_x(X_0=x) = 1$ for all $x\in\mathcal{U}$ ($\mE_x$ are the corresponding expectations). Furthermore, for arbitrary $t\geq 0$ and $A\in\widetilde{\mathcal{F}}$, the function $x\mapsto \mP_x(X_t\in A)$ is measurable, and we have the following \emph{Markov property}: For $f\in\mathcal{B}_b(\mathcal{U})$,
\[
\mE_x[f(X_{t+s})| \mathcal{F}_s] = \mE_{X_s}[f(X_t)], \quad x\in\mathcal{U}, ~s,t\geq 0.
\]

We make a few further standard assumptions. We assume that the filtration $(\mathcal{F}_t)_{t\geq 0}$ is augmented by the null sets and right continuous, and we add an isolated, absorbing \emph{cemetery} point $\Delta$ to $\mathcal{U}$. Thus, we let $\mathcal{U}_\Delta := \mathcal{U}\cup\{\Delta\}$. The lifetime of the process is
\[
\zeta := \inf\{t\geq 0:~X_t=\Delta\},
\]
and we have $X_t = \Delta$ if $t\geq\zeta$. Moreover, there exists a \emph{dead path} $[\Delta] \in \Omega$ for which $\zeta([\Delta]) = 0$.\footnote{If we identify $\omega$ with the trajectory $t\mapsto X_t(\omega)$, then $[\Delta]$ is the trajectory $X_t\equiv \Delta$.} We also assume that $X_\infty = \Delta$, $\theta_\infty\omega = [\Delta]$ for all $\omega\in\Omega$ and for any measurable numerical function $f$, $f(\Delta) = 0$.

The semigroup $(T_t)_{t\geq 0}$ corresponding to the stochastic process $X$ is defined as
\[
T_tf(x) = \mE_x[f(X_t)], \quad t\geq 0, ~x\in\mathcal{U},
\]
where $f$ is bounded or non-negative. 

Let $\lambda\in [0,\infty)$. We say that measurable function $f$ on $\mathcal{U}$ is \emph{$\lambda$--excessive} (see, e.g., \cite{MR850715}[Section II.3]) if $f\ge 0$, $e^{-\lambda t}T_tf\le f$ for $t>0$,  and $e^{-\lambda t}T_tf\to f$ as $t\to 0^+$. Let $\mathcal{S}_\lambda$ be the set of all excessive functions. The process $X$ is called \emph{right process} if it is a Markov process equipped with an augmented and right continuous filtration and for all $\lambda>0$, $f\in \mathcal{S}_\lambda$, the map $t\mapsto f(X_t)$ is a.s. right continuous. 

Right processes have the following \emph{strong Markov property} (see \cite{MR4247975}): for every stopping time $\tau$ with respect to the filtration $(\mathcal{F}_t)_{t\geq 0}$ and bounded $\mathcal{F}_0$-measurable function $f$,\footnote{As usual, $X_t\circ \theta_\tau(\omega)$ means $X_t(\theta_\tau\omega)$.}
\begin{align}
    \mE_x[f(X_t\circ\theta_\tau)|\mathcal{F}_\tau] = \mE_{X_\tau}[f(X_t)], \quad t\geq 0, ~x\in\mathcal{U}.
\end{align}
Here $\mathcal{F}_0$ is the universal completion of $\sigma(X_s:~s\geq 0)$ with respect to all $\mP^x$.
We call $X$ a \emph{Feller process} or a \emph{strong Feller process} if the corresponding semigroup $(T_t)_{t\geq 0}$ is Feller or strong Feller, respectively.

It is well known that the isotropic $\alpha$-stable L\'evy process killed when leaving $D$, namely,
\[
Y_t^D := \begin{cases}  
Y_t, &t<\tau_D,\\
\Delta, &t\geq\tau_D,
\end{cases}
\]
is a right process with lifetime $\xi = \tau_D$. Indeed, it is a Feller process, so a Hunt process (see Chung \cite{MR648601}), and every Hunt process is a right process (see Getoor \cite[Section 9]{getoor1975}).


\section{The Servadei--Valdinoci process}\label{Chap_Servadei_process}

In this section, we construct a stochastic (Servadei--Valdinoci) process on $\R_*$, which is the main focus of the paper. The construction relies on concatenation of right Markov processes, as presented in \cite{MR4247975}. Further, we give the corresponding transition kernel and semigroup of operators, and study their properties. In particular, we prove they are sub-Markovian, strong continuous, 
and introduce auxiliary excessive functions for the semigroup.

\subsection{Construction of the process}

We start with a construction of the stochastic process. Its behaviour depends on the starting point: starting from $x>0$ the process behaves like a  certain right process $Z^1$, while starting from $x<0$ it behaves like a different right process $Z^2$. The processes $Z^1$ and $Z^2$ are defined as follows. Let 
\[
Z^1 = (\Omega^1, \mathcal{F}^1, (\mathcal{F}_t^1)_{t\geq 0}, (Z_t^1)_{t\geq 0}, (\theta_t^1)_{t\geq 0}, (\mP_x^1)_{x>0}),
\]
be the isotropic $\alpha$-stable L\'evy process killed when leaving $D = (0,\infty)$, i.e. $Z^1 := Y^D$. We denote the transition kernel of $Z^1$ by
\[
P^1_t(x,A) := P_t^D(x,A) = \int_A p_t^D(x,y)\,\dy,\quad t>0, ~x>0,~A\subset \R.
\]
Of course, the lifetime of $Z^1$ is $\zeta_1^Z := \tau_D = \sup\{t\geq 0:~Z^1_t>0\}$.

\medskip
The second process we consider lives on $(-\infty,0)$, stays at the starting point $x<0$ for an exponential time with mean $1/\nu(x,D)$ and afterwards it is killed. The process will be denoted by
\[
Z^2 = (\Omega^2, \mathcal{F}^2, (\mathcal{F}_t^2)_{t\geq 0}, (Z_t^2)_{t\geq 0}, (\theta_t^2)_{t\geq 0}, (\mP_x^2)_{x<0}).
\]
Its transition kernel is 
\[
P^2_t(x,A) := \delta_x(A)e^{-\nu(x,D)t}, \quad x<0,~A\subset\R, ~t>0,
\]
and its lifetime $\zeta_2^Z$ is the above exponential random variable with mean $1/\nu(x,D)$.

\medskip
For $t>0$, we define the integral kernel on $\R_*$,
\begin{align*}
    \P_t(x,A) &=
    \begin{cases}
    P^1_t(x,A), &\textrm{ if } x>0, \\
    P^2_t(x,A), &\textrm{ if } x< 0,
    \end{cases}
\end{align*}
where $A\subset \R_*$.
We note that $\P_t(x, \overline{D}^c) = 0$ if $x>0$ and $\P_t(x,D) = 0$ if $x<0$. 

\begin{lemma}\label{F1}
For $t>0$, $\P_t$ is a subprobability transition kernel.
\end{lemma}

\begin{proof}
We clearly have $\P_t(x,A)\geq 0$ for all $x\in\R$ and $A\subset\R$. Moreover,
\[
\P_t(x,\R) = \int_{\R} p_t^D(x,y)\,\dy \leq \int_{\R} p_t(x,y)\,\dy = 1,\quad x>0,
\]
and
\[
\P_t(x,\R) = e^{-t\nu(x,D)} \leq 1,\quad x< 0.
\]
It remains to verify the Chapman--Kolmogorov equations. Let $s,t>0$ and $A\subset \R$. For $x>0$, we use \eqref{eq:p^D_Ch-K}. For $x<0$, we have
\begin{align*}
    \int_{\R} \P_t(x,\dy)\P_s(y,A) &=  e^{-t\nu(x,D)} \P_s(x,A) = e^{-t\nu(x,D)} e^{-s\nu(x,D)}\delta_x(A) \\
    &= \delta_x(A) e^{-(s+t)\nu(x,D)} = \P_{s+t}(x,A),
\end{align*}
which completes the proof.
\end{proof}

As usual, we let 
\[
\widehat{P}_tf(x) := \int_\R f(y)\widehat{P}_t(x,\dy),
\]
if $f$ is nonnegative or bounded. Additionally, we let $\widehat{P}_0 := I$.

\begin{lemma}\label{P_t_feller}
$(\widehat{P}_t)_{t\geq 0}$ is a Feller semigroup on $C_0(\R_*)$.
\end{lemma}

\begin{proof}
By Lemma \ref{F1}, $(\widehat{P}_t)_{t\geq 0}$ is a sub-Markov semigroup on $C_0(\R_*)$. Let $f\in C_0(\R_*)$. We will show that $\widehat{P}_tf\in C_0(\R_*)$. The function $D\ni x\mapsto \P_tf(x) = P_t^Df(x)$ is in $C_0(D)$ by Lemma \ref{theorem_p^D_feller}, and of course, $\overline{D}^c\ni x\mapsto \P_tf(x) = e^{-\nu(x,D)t} f(x)$ is in $C_0(\overline{D}^c)$.

To finish the proof, we show that for $f\in C_0(\R_*)$,  $\P_tf(x) \to f(x)$, as $t\to 0^+$. For $x>0$,
\[
|\P_t f(x) - f(x)| = |P_t^Df(x) - f(x)| \to 0,
\]
as $t\to 0^+$, which follows from Lemma \ref{theorem_p^D_feller}. In the case $x<0$, we have 
\[
|\P_tf(x) - f(x)| \leq |f(x)| \cdot |e^{-\nu(x,D)t} - 1| \to 0,
\]
as $t\to 0^+$. This proves the lemma.
\end{proof}

\medskip
It follows from Kallenberg \cite[Chapter 19]{MR1876169} or Getoor and Blumenthal \cite[Theorem I.9.4]{MR0264757} that  there exists a Hunt process (hence, a right process \cite[(47.3)]{MR958914}) $Z$ on $\R_*$, with transition function given by $\P_t$. The lifetime of this process will be denoted by 
\[
\zeta = 
\begin{cases}
    \zeta_1^Z, &\textrm{if } Z_0>0, \\
    \zeta_2^Z, &\textrm{if } Z_0<0.
\end{cases}
\]
Of course, $\zeta<\infty$ a.s. The process $Z$ starting from $x>0$ is the isotropic $\alpha$-stable L\'evy process $Z^1=Y$ killed when leaving $D$, while starting from $x<0$ it is the compound--Poisson type process $Z^2$. We will extend $Z$ beyond its lifetime by \emph{concatenation}.

\medskip
A concatenation of functions: $x_1$, defined on $[0,t_1)$, $x_2$, defined on $[0,t_2)$, etc., is a function $x$ defined on $[0,t_1+t_2+\ldots)$ by letting $x(t) = x_n\big( t - (t_1+\ldots + t_n)\big)$ for $t_1+\ldots+t_n\leq t<t_1+\ldots+t_n+t_{n+1}$. Here $t_1, t_2,\ldots \in [0,\infty)$ and if $t_i=0$ then the function $x_i$ has no affect on $x$. 

\medskip
We will next concatenate a countable sequence of stochastic processes. For $n\in \mathbb{N}$, suppose that the processes $X^n = (\Omega^{(n)}, \mathcal{F}^{(n)}, (\mathcal{F}_t^{(n)})_{t\geq 0}, (X_t^{(n)})_{t\geq 0}, (\theta_t^{(n)})_{t\geq 0}, (\mP_x^{(n)})_{x\in \R_*})$ on $\R_*$ are independent copies of the right process $Z$. By $\zeta^{(n)}$ we denote the lifetime  of each process $X^n$, $n\in\mathbb{N}$. For each process $X^n$, we consider the dead path $[\Delta^n]\in \Omega^{(n)}$ with the property $\zeta^{(n)}([\Delta^n]) = 0$, and let $\theta_\infty^{(n)}(\omega):= [\Delta^n]$ for all $\omega\in \Omega^{(n)}$.\footnote{We distinguish dead paths for each process but the cemetery point attached to $\R_*$ may be the same, denoted by $\Delta$.} 

\medskip
We define the transfer kernel $k$ on $\R_*$ (also called \emph{instantaneous kernel}) to specify jumps between $D$ and $\overline{D}^c$ in the following way: for $A\subset\R$,
\[
k(x,A) = \begin{cases}
\displaystyle \frac{\nu(x,A\cap D^c)}{\nu(x,D^c)}, &x>0, \\[7pt]
\displaystyle \frac{\nu(x,A\cap D)}{\nu(x,D)}, &x<0.
\end{cases}
\]
If $x>0$, then $k(x,\cdot)$ gives the starting distribution of the process $Z^2$ initiated after $Z^1$ died at $x$. Similarly, for $x<0$, $k(x,\cdot)$ is the initial distribution of $Z^1$ initiated after $Z^2$ died at $x$. To be more precise, for $k = 1,2,\ldots$, we define a \emph{transfer kernel} $K^k:\Omega^{(k)}\times \mathcal{B} \to [0,1]$ by
\begin{align}\label{eq:transition_semi}
    K^k(\omega_k, A) := k(X^k_{\zeta^{(k)}-}(\omega_k), A), \quad \omega_k\in \Omega^{(k)}, ~A\subset \R_*,
\end{align}
where $\mathcal{B}$ denotes the $\sigma$-algebra of Borel measurable subsets of $\R_*$.

\medskip
Following Werner \cite{MR4247975} (see also Sharpe \cite[Chapter 14]{MR958914}), we define the process $X$ on $\R_*$ which is called a \emph{concatenation} of the processes $(X^n)_{n\in\mathbb{N}}$, as follows.

We set $\Omega = \prod_{n\in\mathbb{N}}\Omega^{(n)}$ and $\widetilde{\mathcal{F}} = \bigotimes_{n\in\mathbb{N}} \mathcal{F}^{(n)}$. For $\omega = (\omega_n)_{n\in\mathbb{N}}\in\Omega$, $t\geq 0$ and $n\geq 1$, we define $\widetilde{R}_n(\omega) := \sum_{k=1}^n \zeta^{(k)}(\omega_k)$ and $\widetilde{R}_0(\omega) = 0$, and $\widetilde{R}_\infty(\omega) = \lim\limits_{n\to\infty} \widetilde{R}_n(\omega) = \sum_{k=1}^\infty \zeta^{(k)}(\omega_k)$. To simplify notation, we make the convention $\zeta^{(k)}(\omega) = \zeta(\omega_k) = \zeta^{(k)}(\omega_k)$. 

Following the construction in \cite{MR4247975} for all $t\geq 0$, $\omega = (\omega_1,\omega_2,\ldots)\in\Omega$, we let
\begin{align*}
    X_t(\omega) =
    \begin{cases}
    X_t^1(\omega_1), &\widetilde{R}_0(\omega)\leq t<\widetilde{R}_1(\omega),\\
    X_{t-\widetilde{R}_1(\omega)}^2(\omega_2), &\widetilde{R}_1(\omega)\leq t<\widetilde{R}_2(\omega), \\
    X_{t-\widetilde{R}_2(\omega)}^3(\omega_3), &\widetilde{R}_2(\omega)\leq t<\widetilde{R}_3(\omega), \\
    \vdots &\vdots \\
    \Delta, & t\geq \widetilde{R}_\infty(\omega).
    \end{cases}
\end{align*}
Here $\Delta$ is the isolated, absorbing cemetery state attached to $\R_*$. We set $\R_\Delta := \R_*\cup \{\Delta\}$. Of course,
\[
\widetilde{R}_\infty = \inf\{t\geq 0:~ X_t = \Delta\},
\]
and $X_t = \Delta$ for all $t\geq \widetilde{R}_\infty$.

For $t\in [0,\infty)$ and $\omega = (\omega_1, \omega_2,\ldots) \in\Omega$, we define the \emph{shift operator} as follows:

\begin{align*}
    \theta_t(\omega) =
    \begin{cases}
    (\theta^{(1)}_t(\omega_1), \omega_2, \omega_3,\omega_4,\ldots), &\widetilde{R}_0(\omega)\leq t<\widetilde{R}_1(\omega),\\
    ([\Delta^1], \theta^{(2)}_{t-\widetilde{R}_1(\omega)}(\omega_2), \omega_3, \omega_4,\ldots), &\widetilde{R}_1(\omega)\leq t<\widetilde{R}_2(\omega), \\
    ([\Delta^1],[\Delta^2],\theta^{(3)}_{t-\widetilde{R}_2(\omega)}(\omega_3), \omega_4, \ldots), &\widetilde{R}_2(\omega)\leq t<\widetilde{R}_3(\omega), \\
    \vdots &\vdots 
    \end{cases}
\end{align*}

We also let $X_\infty = \Delta$ and $\theta_\infty \equiv [\Delta]$. 

Further details of this construction, especially the construction of the filtration $(\mathcal{F}_t)_{t\geq 0}$, and the family of measures $(\mathbb{P}_x, x\in\R_*)$ can be found in Werner \cite{MR4247975} (see also Ikeda et al. \cite{MR0202197}, Meyer \cite{MR0415784} and Sharpe \cite[Chapter II.14]{MR958914}). We only note that the \emph{law} of the concatenated process, i.e. a suitable family of measures $(\mathbb{P}_x, x\in \R_*)$ on $(\Omega,\widetilde{\mathcal{F}})$, is defined in terms of the distributions $(\mathbb{P}_x^{(n)}, x\in\R_*)$, $n\in\mathbb{N}$, and the transfer kernel $K$. 

From \cite{MR4247975} it follows that $X = (\Omega, \widetilde{\mathcal{F}}, (\mathcal{F}_t)_{t\geq 0}, (X_t)_{t\geq 0}, (\theta_t)_{t\geq 0}, (\mP_x)_{x\in \R_*})$, constructed above, is a right process on $\R_*$ with the lifetime $\widetilde{R}_\infty$. In particular, $X$ has the strong Markov property and for $n\in\mathbb{N}$, $x\in\R_*$ and $f\in \mathcal{B}_b(\R_*)$,
\begin{align}\label{eq:Werner_eq}
    \mE_x\big[ f(X_{\widetilde{R}_n}) ~\big|~\mathcal{F}_{{\widetilde{R}_n}^-}\big] = K^n f\circ \pi^n,
\end{align}
where $\pi^n$ denotes the projection of $\omega = (\omega_1, \omega_2,\ldots)$ onto the $n$-th coordinate. Moreover, from \eqref{eq:transition_semi}, for $\omega\in\Omega$, $A\subset \R_*$ and $f(x) = \ind_A(x)$, we have\footnote{Recall that $\mathcal{F}_{\tau^-}$ is generated by $\mathcal{F}_{0^+}$ and the class of sets $A\cap \{t<\tau\}$, where $A\in \mathcal{F}_t$, $t\geq 0$ (see \cite[p. 16]{chung2005markov}).}
\begin{align}\label{eq:Werner_eq1}
    \mP_x\big[ X_{\widetilde{R}_n}\in A~\big|~\mathcal{F}_{{\widetilde{R}_n}^-}\big](\omega) = K^n(\omega_n,A) = k(X^n_{\zeta^{(n)}-}(\omega_n),A).
\end{align}
We call $X$ the Servadei--Valdinoci process.

For the process $X$, we define a sequence of stopping times
\begin{align*}
  R_1 & := \inf\{t\geq 0:\: X_0 X_t < 0 \} = \tau_D + \tau_{\bar{D}^c}, \\
  R_{n+1} & := R_n + R_1\circ \theta_{R_n},\quad n=1,2,....   
\end{align*}
Let $m(\omega)=\min\{k\in\mathbb{N}:\: \omega_k\neq [\Delta^k]\}$. From the construction, $R_1=\widetilde{R}_m$.
Hence, for all sets $A\subset  \R_*, T \subset [0,\infty),$ by \eqref{eq:Werner_eq1} we get
\begin{align}\label{distributionR_1}
\mP_x\left[R_1 \in T, X_{R_1}\in A\right] & = \sum_{k=1}^\infty \mP_x\left[R_1\in T, X_{R_1}\in A, R_1=\widetilde{R}_k \right]
\\ \nonumber
& = \sum_{k=1}^\infty \mP_x\left[\widetilde{R}_k\in T, X_{\widetilde{R}_k}\in A, R_1=\widetilde{R}_k \right]
\\ \nonumber
& = \sum_{k=1}^\infty \mE_x\left[ \mE_x\left[ \ind_A(X_{\widetilde{R}_k})  \big|~\mathcal{F}_{\widetilde{R}_k-}\right]; \widetilde{R}_k\in T, R_1=\widetilde{R}_k \right] \\ \nonumber
& = \sum_{k=1}^\infty \mE_x\left[ k(X_{\widetilde{R}_k-},A),\widetilde{R}_k\in T, R_1=\widetilde{R}_k\right] \\ \nonumber
& =  \mE_x\left[ k(X_{R_1-},A),R_1\in T \right].
\end{align}
Of course we have $\lim\limits_{n\to\infty} R_n = \widetilde{R}_\infty.$ In what follows, we will use the times $R_n$ and not $\widetilde{R}_n$.

\subsection{Transition semigroup of the process}\label{sec:semigroup}

Let
\begin{align*}
    \widehat{\nu}(x,y) = \begin{cases}
    \nu(x,y), &\textrm{if } x\in D, ~y\in D^c \textrm{ or } x\in D^c, ~y\in D,\\
    0, &\textrm{otherwise.}
    \end{cases}
\end{align*}
By the usual abuse of notation, we also use $\widehat{\nu}$ to denote the corresponding kernel on $\R$:
\begin{align*}
    \widehat{\nu}(x,A) = \int_A \widehat{\nu}(x,y)\,\dy, \quad x\in\R, ~ A\subset\R.
\end{align*}
Of course, $\widehat{\nu}(x,A) = 0$ if $x\in D$, $A\subset D$. Similarly, $\widehat{\nu}(x,A) = 0$ if $x\in D^c$, $A\subset D^c$. 

\medskip
With the kernel meaning of $\widehat{\nu}$, we consider the following perturbation series: $K_0 := I$ and for $t>0$,
\begin{align}\label{eq:K_t_def}
    K_t &:= \P_t + \int_0^t \P_{t_1} \widehat{\nu} \P_{t-t_1}\,\mathrm{d}t_1 + \int_0^t \int_{t_1}^{t} \P_{t_1} \widehat{\nu} \P_{t_2-t_1} \widehat{\nu} \P_{t-t_2}\,\mathrm{d}t_2\,\mathrm{d}t_1 + \ldots := \sum_{n=0}^\infty K_{t,n}.
\end{align}
Thus $K_{t,0} = \P_t$ and, for $n\geq 1$,
\begin{align}\label{eq:K_t_n_def}
K_{t,n} := \idotsint_{0<t_1<t_2<\ldots<t_n<t} \P_{t_1} \widehat{\nu} \P_{t_2-t_1} \widehat{\nu} \ldots \widehat{\nu} \P_{t-t_n}\,\mathrm{d}t_1\ldots\mathrm{d}t_n.
\end{align}

\medskip
From Bogdan and Sydor \cite{MR3295773}, $K_t$ is a transition kernel. In what follows, we need a connection of $(K_t)_{t\geq 0}$ to right processes, so we will not use \cite{MR3295773}, but the reader interested in a purely analytic approach to $(K_t)_{t\geq 0}$ may find \cite{MR3295773} instructive. For the probabilistic proof of the fact that $K_t$, $t\geq 0$, is a transition kernel, see Lemma \ref{K_subprobability}. The deficiency of \cite{MR3295773} for our goals is that the strong Markov property of the Markov process given by $(K_t)_{t\geq 0}$ is not addressed in \cite{MR3295773}. 

\medskip
In what follows, in case of non-negative integrands, we often use Tonelli's theorem without mention.

\begin{lemma}\label{proposition_recurrence}
For $t>0$, $n=0,1,\ldots$,
\begin{equation}
    \label{eq:K_recurence}
    K_{t,n+1} = \int_0^t  \P_{r} \widehat{\nu} K_{t-r,n}\,\dr = \int_0^t K_{r,n}\widehat{\nu}\P_{t-r}\,\dr.
\end{equation}
\end{lemma}
\begin{proof}
By the substitution $u_1=t_1+r$, $u_2=t_2+r, \ldots, u_n=t_n+r$,
\begin{align*}
\int_0^t  \P_{r} \widehat{\nu} K_{t-r,n}\,\dr &=
\int_0^t \dr \int_0^{t-r}\dt_1 \int_{t_1}^{t-r}\dt_2 \ldots \int_{t_{n-1}}^{t-r}\dt_n \,
\P_{r} \widehat{\nu} \P_{t_1} \widehat{\nu} \P_{t_2-t_1} \widehat{\nu} \ldots \widehat{\nu} \P_{t-r-t_n} \\
&= \int_0^t \dr \int_r^t \du_1 \int_{u_1}^t \du_2 \ldots \int_{u_{n-1}}^t\du_n \, \P_r\widehat{\nu}\P_{u_1-r}\widehat{\nu}\P_{u_2-u_1}\widehat{\nu}\ldots\widehat{\nu}\P_{t-u_n} \\
&=\idotsint_{0<r<u_1<u_2<\ldots<u_{n-1}<u_n<t} \P_r\widehat{\nu}\P_{u_1-r}\widehat{\nu}\P_{u_2-u_1}\widehat{\nu}\ldots\widehat{\nu}\P_{t-u_n} \\
&= K_{t,n+1}.
\end{align*}
The second equality in \eqref{eq:K_recurence} arises similarly.
\end{proof}

By \eqref{eq:K_t_def} and \eqref{eq:K_recurence}, we get the following perturbation formula for $K_t$.

\begin{corollary}\label{perturbation_formula}
For $t>0$,
\begin{equation}\label{eq:perturbation_formula}
    K_t = \P_t + \int_0^t \P_r \widehat{\nu} K_{t-r}\,\dr = \P_t + \int_0^t K_r \widehat{\nu} \P_{t-r}\,\dr.
\end{equation}
\end{corollary}

From the definition, for $x>0$, $t>0$, $K_{t,0}(x,\dy)$ is absolutely continuous with respect to the Lebesgue measure on $\R$ with density $K_{t,0}(x,y) = p_t^D(x,y)$. In particular, $K_{t,0}(x,y) = K_{t,0}(y,x)$, $x,y>0$. 

\begin{lemma}\label{symmetry_of_K}
    For $t>0$, $f,g\in \mathcal{B}^+(\R)$,
    \begin{align}\label{eq:symmetricity}
    \int_\R (K_tf)(x)g(x)\,\dx = \int_\R f(x) (K_tg)(x)\,\dx.
    \end{align}
\end{lemma}

\begin{proof}
Note that
\begin{align*}
\int_\R (K_tf)(x) g(x)\,\dx &= \int_\R \dx \int_\R  \, f(y)g(x)K_t(x,\dy) \\
&= \sum_{n=0}^\infty \int_\R \dx \int_\R  \, f(y)g(x)K_{t,n}(x,\dy).
\end{align*}
It suffices to show that for every $n\in\mathbb{N}$,
\begin{align}
    \label{eq:symmetricity_n}
    \int_\R \dx \int_\R  \, f(y)g(x)K_{t,n}(x,\dy) = \int_\R \dx \int_\R  \, f(x)g(y)K_{t,n}(x,\dy).
\end{align}
We proceed by induction. 
By the definition and symmetry of $p^D$, we have
\begin{align*}
    &\int_\R \dx \int_\R  \, f(y)g(x)K_{t,0}(x,\dy) \\
    &= \int_D \dx \int_D \dy \, f(y)g(x)p_t^D(x,y) + \int_{D^c} \dx \int_{D^c}  \, \delta_x(\dy) e^{-\nu(x,D)t} f(y)g(x) \\
    &= \int_D \dx \int_D \dy \, f(y)g(x)p_t^D(y,x) + \int_{D^c} f(x)g(x) e^{-\nu(x,D)t} \,\dx \\
    &= \int_D \dy \int_\R \dx \, f(y)g(x)p_t^D(y,x) + \int_{D^c} \dy \int_{\R} \delta_y(\dx) e^{-\nu(y,D)t} f(y)g(x) \\
    &= \int_\R \dy \int_\R f(y)g(x)K_{t,0}(y,\dx) = \int_\R \dx \int_\R f(x)g(y)K_{t,0}(x,\dy).
\end{align*}
Put differently,
\begin{align}
    \label{eq:n0_Pt}
    \int_\R \dx \int_\R f(y)g(x)\widehat{P}_t(x,\dy) = \int_\R \dx \int_\R f(x)g(y)\widehat{P}_t(x,\dy).
\end{align}
Similarly, we prove that
\begin{align}
    \label{eq:nu}
    \int_\R \dx \int_\R f(y)g(x)\widehat{\nu}(x,\dy) = \int_\R \dx \int_\R f(x)g(y)\widehat{\nu}(x,\dy).
\end{align}

Assume that \eqref{eq:symmetricity_n} holds for some $n\geq 0$. From Lemma \ref{proposition_recurrence}, \eqref{eq:n0_Pt} and \eqref{eq:nu},
\begin{align*}
    &\int_\R \dx \int_\R f(y)g(x) K_{t,n+1}(x,\dy) \\
    &= \int_0^t \dr \int_\R \dx \int_\R f(y)g(x) K_{r,n}\widehat{\nu}\P_{t-r}(x,\dy) = \int_0^t \dr \int_\R \dx \int_\R \, g(x) \big(\widehat{\nu}\P_{t-r}f\big)(y) K_{r,n}(x,\dy) \\
    &= \int_0^t \dr \int_\R \dx \int_\R \, g(y) \big(\widehat{\nu}\P_{t-r}f\big)(x) K_{r,n}(x,\dy) = \int_0^t \dr \int_\R \dx \, \big(\widehat{\nu}\P_{t-r}f\big)(x) \big(K_{r,n}g\big)(x) \\
    &= \int_0^t \dr \int_\R \dx \int_\R \dy \, \widehat{\nu}(x,\dy) \big(\P_{t-r}f\big)(y) \big(K_{r,n}g\big)(x) \\
    &= \int_0^t \dr \int_\R \dx \int_\R \dy \, \widehat{\nu}(x,\dy) \big(\P_{t-r}f\big)(x) \big(K_{r,n}g\big)(y) = \int_0^t \dr \int_\R \dx \, \big(\P_{t-r}f\big)(x) \big(\widehat{\nu}K_{r,n}g\big)(x) \\
    &= \int_0^t \dr \int_\R \dx \, f(x) \big(\widehat{P}_{t-r}\widehat{\nu}K_{r,n}g\big)(x) = \int_0^t \dr \int_\R \dx \, f(x) \big(\widehat{P}_{r}\widehat{\nu}K_{t-r,n}g\big)(x).
\end{align*}
By Lemma \ref{proposition_recurrence},
\begin{align*}
    \int_\R \dx \int_\R f(y)g(x) K_{t,n+1}(x,\dy) = \int_\R \dx \int_\R\, f(x)g(y) K_{t,n+1}(x,\dy),
\end{align*}
which completes the proof.
\end{proof}

\begin{proposition}\label{Th:semigroup}
For $x\neq 0$, $t>0$ and $f\in \mathcal{B}^+(\R_*)$, we have
\begin{align}\label{eq:semigroup_of_the_process}
\mE_x f(X_t) = K_tf(x).
\end{align}
\end{proposition}

In order to prove Proposition \ref{Th:semigroup}, we need the following auxiliary results. 

For $n=1,2,\ldots$, $x\in\R$ and Borel sets $T\subset [0,\infty)$ and $A\subset\R\setminus \{0\}$, we define 
\[
P_n(x, T, A) := \mP_x(R_n\in T, X_{R_n}\in A),
\]
the (joint) distribution of the pair $(R_n, X_{R_n})$ under $\mP_x$.

\begin{lemma}\label{lem:L1}
For $x\neq 0$, $T\subset [0,\infty)$ and $A\subset \R_*$,
\begin{align*}
    P_1(x,T,A) &= \int_T \ds \int_A (\P_s \widehat{\nu})(x,\da), \\
    P_{n+1}(x,T,A) &= \iint P_n(x,\ds,\dv) \int_T \dr \int_A (\P_{r-s} \widehat{\nu})(v,\da)\ind_{\{s<r\}},\quad n\geq 1.
\end{align*}
\end{lemma}

\begin{proof}
From 
\eqref{distributionR_1} we have
\[
 P_1(x,T,A) = \mP_x(R_1\in T, X_{R_1}\in A)
 = \mE_x\big[\ind_{T}(R_1(\omega)) k(X_{R_1-}(\omega),A)].
\]
For $x>0$, from \eqref{eq:transition_semi} and the Ikeda--Watanabe formula \eqref{eq:Ikeda_Watanabe},
\begin{align*}
    P_1(x,T,A) &=\mE_x\big[\ind_{T}(R_1) k(X_{R_1-},A) \big] 
    = \mE_x^1\big[\ind_{T}(\tau_D) k(Z^1_{\tau_D-},A) \big] \\
    &= \int_T \ds \int_D \dy \,p_s^D(x,y)\nu(y,D^c)k(y,A)
    = \int_T \ds \int_D \dy \,p_s^D(x,y)\nu(y,A\cap D^c) \\
    &= \int_T \ds \int_D \dy \int_A \da  \, p_s^D(x,y)\widehat{\nu}(y,a) 
    = \int_T \ds \int_A  (\P_s\widehat{\nu})(x,\da).
\end{align*}
Similarly, for $x<0$,
\begin{align*}
    P_1(x,T,A) &= \mE_x\big[\ind_{T}(R_1) k(X_{R_1-},A) \big] 
    = \mE_x^2\big[\ind_{T}(\zeta_2^Z) k(Z^2_{\zeta_2^Z-},A) \big] \\
    &=  k(x,A) \mP_x^2\big[\zeta_2^Z\in T \big] 
    = \int_T \ds \int_A \da \, e^{-\nu(x,D)s} \widehat{\nu}(x,a) \\
    &= \int_T \ds \int_A \da \, (\P_s\widehat{\nu})(x,\da).
\end{align*}

Furthermore, for $n\geq 1$, 
\begin{align*}
    P_{n+1}(x,T,A) &= \mP_x(R_{n+1}\in T, X_{R_{n+1}}\in A) \\
    &= \int_\Omega \mE_x\big[ \ind_{T}(R_n + R_1\circ\theta_{R_n}) \ind_{A}(X_{R_{n+1}})  \big| \mathcal{F}_{R_n}\big](\omega) \mP_x(\mathrm{d}\,\omega).
\end{align*}
Let $G(\omega,\omega') := \ind_T(R_n(\omega) + R_1(\omega'))\ind_A(X_{R_1(\omega')}(\omega'))$, $\omega, \omega' \in\Omega$. 
For $\omega' := \theta_{R_n}\omega$, we have
\[
X_{{R_1}(\omega')}(\omega') = X_{R_1\circ\theta_{R_n}(\omega)}(\theta_{R_n}\omega)  = X_{R_{n+1}(\omega)}(\omega).
\]
Moreover, for $H(\omega) := G(\omega, \theta_{R_n}\omega)$ we have,
\begin{align*}
\mE_x[H | \mathcal{F}_{R_n}](\omega) = 
\mE_x\big[ \ind_{T}(R_n + R_1\circ\theta_{R_n}) \ind_{A}(X_{R_{n+1}})  \big| \mathcal{F}_{R_n}\big](\omega).
\end{align*}
Hence,
\[
P_{n+1}(x,T,A) = \int_\Omega \mE_x[H | \mathcal{F}_{R_n}](\omega) \mP_x(\mathrm{d}\,\omega).
\]
From the strong Markov property  \cite[Exercise 8.16]{MR0264757}, we obtain that 
\begin{align*}
    P_{n+1}(x,T,A) &= \int_\Omega \int_\Omega G(\omega, \omega') \mP_{X_{R_n}(\omega)}(\mathrm{d}\,\omega') \mP_x(\mathrm{d}\,\omega) \\
    &= \int_\Omega \int_\Omega \ind_T(R_n(\omega) + R_1(\omega'))\ind_A(X_{R_1(\omega')}(\omega')) \mP_{X_{R_n}(\omega)}(\mathrm{d}\,\omega') \mP_x(\mathrm{d}\,\omega) \\
    &= \int_\Omega \mE_{X_{R_n}(\omega)}\big[\ind_T(R_n(\omega) + R_1)\ind_A(X_{R_1}) \big] \mP_x(\mathrm{d}\,\omega)  \\
    &=\mE_x\big[ \mE_{X_{R_n}}\big[\ind_T(s + R_1)\ind_A(X_{R_1}) \big]\big|_{s=R_n} \big] \\
    &= \mE_x\big[ \mP_{X_{R_n}} \big( R_1 + s\in T, X_{R_1}\in A\big) \big|_{s=R_n}\big] \\
    &=\iint P_n(x,\ds, \dv) \mP_v(R_1\in (T-s)\cap (0,\infty), X_{R_1}\in A).
\end{align*}
Moreover, we have
\begin{align*}
    \mP_v(R_1 \in (T-s)\cap (0,\infty), X_{R_1}\in A) &= P_1(v, (T-s)\cap (0,\infty), A)\\
    &= \int_{(T-s)\cap (0,\infty)} \dr \int_A (\widehat{P}_r\widehat{\nu})(v,\da) \\
    &= \int_T \dr \int_A (\P_{r-s}\widehat{\nu})(v,\da)\ind_{\{s<r\}}.
\end{align*}
Therefore, 
\begin{align*}
    P_{n+1}(x,T,A) &= \iint P_n(x,\ds, \dv) \int_T \dr \int_A  (\P_{r-s} \widehat{\nu})(v,\da)   \ind_{\{s<r\}},
\end{align*}
which is the desired conclusion.
\end{proof}

\begin{remark}\label{rem:P_n_prob1}
By induction, $P_n(x, (0,\infty), \R_*) = 1$ for every $x\in\R_*$. In particular, $R_n>0$ a.s. 
\end{remark}

\begin{lemma}\label{lem:L2}
For $x\neq 0$ and $n\geq 2$,
\begin{align}\label{eq:lem6}
    P_n(x,\dr,\da) = \idotsint\limits_{0<t_1<\ldots<t_{n-1}<r} (\P_{t_1}\widehat{\nu} \P_{t_2-t_1}\widehat{\nu} \ldots \widehat{\nu} \P_{r-t_{n-1}}\widehat{\nu})(x,\da) \mathrm{d}t_1\ldots\mathrm{d}t_{n-1}\,\dr.
\end{align}
\end{lemma}

\begin{proof}
From Lemma \ref{lem:L1} we have
\begin{align*}
    P_2(x,\dr,\da) &= \iint P_1(x,\ds,\dv) (\P_{r-s}\widehat{\nu})(v,\da) \ind_{\{s<r\}} \dr \\
    &=\iint (\P_s\widehat{\nu})(x,\dv) \,\ds (\P_{r-s}\widehat{\nu})(v,\da) \ind_{\{s<r\}}\dr \\
    &= \int_0^r (\P_s\widehat{\nu} \P_{r-s}\widehat{\nu})(x,\da)\,\ds \,  \dr.
\end{align*}

For $n>2$, we use induction. Assume that the equality \eqref{eq:lem6} holds for some $n\geq 2$. From Lemma \ref{lem:L1} we then have
\begin{align*}
    &P_{n+1}(x,\dr,\da) \\ &= \iint P_n(x,\ds,\dv) (\P_{r-s} \widehat{\nu})(v,\da) \ind_{\{s<r\}}\dr \\
    &= \int_0^r \ds \idotsint\limits_{0<t_1<\ldots<t_{n-1}<s} \int  (\P_{t_1}\widehat{\nu} \P_{t_2-t_1}\widehat{\nu} \ldots \widehat{\nu} \P_{s-t_{n-1}}\widehat{\nu})(x,\dv) \mathrm{d}t_1\ldots\mathrm{d}t_{n-1} (\P_{r-s} \widehat{\nu})(v,\da) \dr \\
    &=\int_0^r \ds \idotsint\limits_{0<t_1<\ldots<t_{n-1}<s}   (\P_{t_1}\widehat{\nu} \P_{t_2-t_1}\widehat{\nu} \ldots \widehat{\nu} \P_{s-t_{n-1}}\widehat{\nu}\P_{r-s} \widehat{\nu})(x,\da)
    \mathrm{d}t_1\ldots\mathrm{d}t_{n-1}\dr \\
    &=  \idotsint\limits_{0<t_1<\ldots<t_n< r}   (\P_{t_1}\widehat{\nu} \P_{t_2-t_1}\widehat{\nu} \ldots \widehat{\nu} \P_{t_n-t_{n-1}}\widehat{\nu}\P_{r-t_n} \widehat{\nu})(x,\da)
    \mathrm{d}t_1\ldots\mathrm{d}t_{n}\dr,
\end{align*}
which ends the proof.
\end{proof}

\begin{proof}[Proof of Proposition \ref{Th:semigroup}]
Let $x\neq 0$, $t>0$ and $f\in\mathcal{B}_b^+(\R)$. Then we have
\begin{align*}
    \mE_xf(X_t) &= \mE_x[f(X_t), 0\leq t<R_1] + \sum_{n=1}^\infty \mE_x[f(X_t), R_n\leq t<R_{n+1}].
\end{align*}
Obviously, for $x>0$,
\begin{align*}
    \mE_x[f(X_t), 0\leq t<R_1] = \mE_x^1[f(Z_t^1), 0\leq t<\tau_D] = \int_D f(y)p_t^D(x,y)\,\dy = \P_tf(x).
\end{align*}
Similarly, for $x<0$,
\begin{align*}
    \mE_x[f(X_t), 0\leq t<R_1] &= \mE_x^2[f(Z_t^2), 0\leq t<\zeta_2^Z] = \int_{D^c} \delta_x(\dy)e^{-\nu(x,D)t} f(y) \\
    &= f(x)e^{-\nu(x,D)t} = \P_tf(x). 
\end{align*}

Now, let $I_n(t,x) := \mE_x[f(X_t), R_n\leq t<R_{n+1}]$, $n\geq 1$. We will show that $I_n(t,x) = K_{t,n}f(x)$. 
Indeed, using the same method as in the proof of Lemma \ref{lem:L1}, in particular using \cite[Exercise 8.16]{MR0264757}, we get
\begin{align}\label{eq:I_n}
    I_n(t,x) &=\mE_x\big[ \mE_x[f(X_t), R_n\leq t<R_{n+1} \big| \mathcal{F}_{R_n}]\big] \nonumber \\
    &=\mE_x\big[ \mE_x[ f(X_{t-R_n}\circ \theta_{R_n}), t<R_n + R_1\circ\theta_{R_n} \big| \mathcal{F}_{R_n}], R_n\leq t\big] \nonumber\\
    &= \mE_x\big[ \mE_{X_{R_n}}[ f(X_{t-s}, t-s<R_1] \big|_{s=R_n}, R_n\leq t\big] \nonumber\\
    &= \iint P_n(x,\ds,\dv) \mE_v[f(X_{t-s}), t-s<R_1]\ind_{\{s\leq t\}} \nonumber\\
    &= \iint P_n(x,\ds,\dv) \P_{t-s}f(v)\ind_{\{s\leq t\}}.
\end{align}
From Lemma \ref{lem:L1} and \eqref{eq:K_t_n_def},
\begin{align*}
    I_1(t,x) &= \int_0^t\ds \int (\P_s\widehat{\nu})(x,\dv) \P_{t-s}f(v) = \int_0^t (\P_s\widehat{\nu} \P_{t-s}f)(x)\,\ds = K_{t,1}f(x).
\end{align*}
It remains to show our claim for $n\geq 2$. From \eqref{eq:I_n}, Lemma \ref{lem:L2} and \eqref{eq:K_t_n_def} it is easy to verify that
\begin{align*}
    I_n(t,x) &= \iint P_n(x,\ds,\dv) \P_{t-s}f(v)\ind_{\{s\leq t\}} \\
    &= \int_0^t \ds \idotsint\limits_{0<t_1<\ldots<t_{n-1}<s} \mathrm{d}t_1\ldots\mathrm{d}t_{n-1} \int (\P_{t_1}\widehat{\nu} \P_{t_2-t_1}\widehat{\nu} \ldots \widehat{\nu} \P_{s-t_{n-1}}\widehat{\nu})(x,\dv) \P_{t-s}f(v) \\
    &= \int_0^t \ds \idotsint\limits_{0<t_1<\ldots<t_{n-1}<s} \mathrm{d}t_1\ldots\mathrm{d}t_{n-1} (\P_{t_1}\widehat{\nu} \P_{t_2-t_1}\widehat{\nu} \ldots \widehat{\nu} \P_{s-t_{n-1}}\widehat{\nu} \P_{t-s}f)(x) \\
    &= \idotsint\limits_{0<t_1<t_2<\ldots<t_n<t} (\P_{t_1} \widehat{\nu} \P_{t_2-t_1} \widehat{\nu} \ldots \widehat{\nu} \P_{t-t_n}f)(x) \,\mathrm{d}t_1\ldots\mathrm{d}t_n = K_{t,n} f(x),
\end{align*}
which completes the proof in case of $f\in \mathcal{B}_b^+(\R)$.
For $f\in \mathcal{B}^+(\R)$, we consider $f_n := f\wedge n$ and 
using the first part of the proof we obtain $\mE_x f_n(X_t) = K_tf_n(x).$
Then the desired equality follows from the monotone convergence.
\end{proof}

\begin{lemma}\label{K_subprobability}
The family $(K_t)_{t\geq 0}$ is a semigroup of subprobability transition kernels on $\R_*$.
\end{lemma}

\begin{proof}
At first, we will prove the subprobability property. It suffices to show (by induction) that for any $N=0,1,2,\ldots$ and $t>0$, $S_N(x,t) := \sum_{n=0}^N K_{t,n}\mathbf{1}(x)\leq 1$, $x\in\R_*$. Here $\mathbf{1}(x) = 1$, $x\in\R_*$.
For $x>0$, from \eqref{eq:p_D_distribution} it follows that for any $t>0$, $S_0(x,t) = K_{t,0}\mathbf{1}(x) = p_t^D(x,D) = \mP^Y_x(\tau_D>t)\leq 1$. For $x<0$, it is obvious that for any $t>0$ we have $S_0(x,t) = K_{t,0}\mathbf{1}(x) = e^{-\nu(x,D)t} \leq 1$. Hence, $S_0(x,t)\leq 1$ for $x\neq 0$, $t>0$.

Assume that for some $N\in\mathbb{N}$ and all $t>0$, $x\neq 0$ we have $S_N(x,t)\leq 1$. Then, from Lemma \ref{proposition_recurrence}, 
\begin{align*}
S_{N+1}(x,t) = \P_t\mathbf{1}(x) + \int_0^t  \widehat{P}_r \widehat{\nu} S_{N}(x,t-r)\,\dr \leq \P_t\mathbf{1}(x)  + \int_0^t \P_r\widehat{\nu}\mathbf{1}(x)\,\dr.
\end{align*}
For $x>0$, from \eqref{eq:p_survival} and \eqref{eq:Ikeda_Watanabe}, we have
\begin{align*}
    S_{N+1}(x,t) \leq \mP_x^Y(\tau_D>t) + \mP_x^Y(\tau_D\leq t) = 1.
\end{align*}
In case $x<0$, we have
\begin{align*}
S_{N+1}(x,t) &\leq e^{-\nu(x,D)t} + \int_0^t \ds \int_D \dy \, e^{-\nu(x,D)s}\nu(x,y) \\
&= \int_t^\infty \nu(x,D)e^{-\nu(x,D)s}\ds + \int_0^t \nu(x,D)e^{-\nu(x,D)s}\ds = 1,
\end{align*}
which ends the induction.

\medskip
The fact that the transition kernels $(K_t)_{t\geq 0}$ satisfy the Chapman--Kolmogorov equation follows directly from Proposition \ref{Th:semigroup}. Indeed, from the Markov property and from mentioned theorem we have the following Chapman--Kolmogorov equation for $K_t$: for $f\in \mathcal{B}_b^+(\R_*)$, $s,t>0$ and $x\neq 0$, we have
\begin{align*}
    K_{t+s}f(x) &= \mE_xf(X_{t+s}) = \mE_x\big[ \mE_x \big(f(X_s) \circ \theta_t~|~\mathcal{F}_t\big)\big] \\
    &= \mE_x\big[ \mE_{X_t}f(X_s)\big] = \mE_x\big[ K_sf(X_t)\big] = K_tK_sf(x).
\end{align*}
Further, for $f=\ind_A$, $A\subset\R_*$, we get the equality
\begin{align}\label{eq:K_t_ChK_eq}
K_{t+s}(x,A) = K_{t+s}\ind_A(x) = K_tK_s\ind_A(x) = \int_\R K_t(x,\dy)K_s(y,A),
\end{align}
which proves the lemma.
\end{proof}

We note that the first part of the above proof is similar to proof of \cite[Lemma 3.3]{KB_MK_2023}. For a direct proof of \eqref{eq:K_t_ChK_eq}, see 
\cite{MR3295773}.

The following result is an extension of Proposition \ref{Th:semigroup}.

\begin{corollary}
For $x\neq 0$, $t>0$, and $f\in \mathcal{B}_b(\R_*)$, we have $\mE_x f(X_t) = K_tf(x)$.
\end{corollary}

\begin{proof}
By Proposition \ref{Th:semigroup} and Lemma \ref{K_subprobability}, both sides are equal and finite for $f\in\mathcal{B}_b^+(\R_*)$. The result follows by considering $f = f_+ - f_-$.
\end{proof}

\begin{lemma}\label{lem4}
For $u\in \mathcal{B}^+(\R)$,
\[
(K_tu)^2(x) \leq (K_tu^2)(x), \qquad x\in \R_*, ~t>0,
\]
and 
\[
\int_\R (K_tu)^2(x)\,\dx \leq \int_\R u^2(x)\,\dx.
\]
\end{lemma}

\begin{proof}
From the Cauchy--Schwarz inequality and Lemma \ref{K_subprobability}, we get
\begin{align*}
    (K_tu)^2(x) &= \Big[ \int_\R u(y)K_t(x,\dy)\Big]^2 \leq \Big[ \int_\R u^2(y)K_t(x,\dy)\Big] \Big[ \int_\R K_t(x,\dy)\Big] \\
        &\leq \int_\R u^2(y) K_t(x,\dy) = (K_tu^2)(x).
\end{align*}
Moreover, from Lemma \ref{symmetry_of_K} and Lemma \ref{K_subprobability} we get
\begin{align*}
\int_\R (K_tu)^2(x)\,\dx \leq \int_\R (K_tu^2)(x)\,\dx = \int_\R u^2(x) K_t\mathbf{1}(x)\,\dx \leq \int_\R u^2(x)\,\dx,
\end{align*}
which completes the proof.
\end{proof}

From the above lemma it follows that $K_tu(x)$ is finite for a.e. $x$ if $u\in L^2(\R)$, and
\begin{align}\label{lem5}
\norm{K_tu}_{L^2(\R)} \leq \norm{K_t|u|}_{L^2(\R)} \leq \norm{u}_{L^2(\R)}.
\end{align}

The next corollary follows from Lemma \ref{symmetry_of_K}.

\begin{corollary}
For $t>0$ and
$f,g \in L^2(\R)$,
\begin{align*}
    \int_\R (K_tf)(x)g(x)\,\dx = \int_\R f(x) (K_tg)(x)\,\dx.
\end{align*}
\end{corollary}

\begin{proposition}
 $(K_t)_{t\geq 0}$ is a strongly continuous semigroup on $L^2(\R)$, i.e. for $t,s\geq 0$, $K_{t+s}u = K_tK_su$ in $L^2(\R)$ and 
\[
\lim_{t\to 0^+} \norm{K_tu - u}_{L^2(\R)} = 0, \qquad u\in L^2(\R).
\]
\end{proposition}

\begin{proof}
Let $u\in L^2(\R)$ and $u_n := u\wedge n \vee (-n)$. 
Of course, $\norm{u-u_n}_{L^2(\R)}\to 0$ as $n\to\infty$. Then from Lemma \ref{K_subprobability}, for $x\neq 0$, we have
\begin{align*}
    K_{t+s}u_n(x) = \int_\R K_{t+s}(x,\dy) u_n(y) = \int_\R (K_tK_s)(x,\dy)u_n(y) = K_tK_su_n(x).
\end{align*}
From \eqref{lem5}, we know that $K_{t+s}$, $K_t$ and $K_s$ are bounded, hence continuous, operators on $L^2(\R)$. Therefore, $\norm{K_{t+s}u_n - K_{t+s}u}_{L^2(\R)} \to 0$ and $\norm{K_tK_su_n - K_tK_su}_{L^2(\R)}\to 0$ as $n\to\infty$. As a result $K_{t+s}u = K_tK_su$ in $L^2(\R)$.

\medskip
Now we will prove the strong continuity. 
Since the family $(K_t)_{t\geq 0}$ forms a contraction semigroup of operators, then from Proposition 1.3 in Engel and Nagel \cite[Section 1]{engel2006short} it suffices to show a desired convergence for $u\in C_c^\infty(\R_*)$, because $C_c^\infty(\R_*)$ is a dense subspace of $L^2(\R)$. 

\medskip
We first show that for $u\in \mathcal{B}^+(\R) \cap C_c^\infty(\R_*)$ and $x\neq 0$,
\begin{equation}
    \lim_{t\to 0^+} K_tu(x) = u(x).
\end{equation}
Let $t\to 0^+$. From Lemma \ref{theorem_p^D_feller},
\begin{align*}
    K_{t,0}u(x) = \mathbf{1}_D(x) \int_{D} p_t^D(x,\dy)u(y)\,\dy + \mathbf{1}_{D^c}(x) u(x)e^{-\nu(x,D)t} \to u(x).
\end{align*}
Furthermore, from Corollary \ref{perturbation_formula} and Lemma \ref{K_subprobability},
\begin{align*}
    \sum_{n=1}^\infty K_{t,n}u(x) \leq \norm{u}_\infty \int_0^t   \P_r \widehat{\nu}\mathbf{1}(x) \,\dr.
\end{align*}
Then for $x>0$,
\begin{align*}
    \int_0^t   \P_r \widehat{\nu}\mathbf{1}(x) \,\dr = \int_0^t \dr \int_D \dy \int_{D^c} \dz \, p_r^D(x,y)\nu(y,z) = \mP_x^Y(\tau_D\leq t) \to 0.
\end{align*}
Similarly, for $x<0$,
\begin{align*}
    \int_0^t   \P_r \widehat{\nu}\mathbf{1}(x) \,\dr = \int_0^t  \nu(x,D) e^{-\nu(x,D)r}\,\dr \to 0.
\end{align*}

Let $u\in C_c^\infty(\R_*)$ be arbitrary. We get
\[
K_tu(x) = K_tu_+(x) - K_tu_-(x) \to u_+(x) - u_-(x) = u(x),
\]
as $t\to 0^+$. Then, by \eqref{lem5},
\begin{align*}
\norm{K_tu-u}_{L^2(\R)}^2 &= \norm{K_tu}_{L^2(\R)}^2 + \norm{u}_{L^2(\R)}^2 - 2\langle K_tu, u\rangle \\
&\leq 2\big[ \norm{u}_{L^2(\R)}^2 - \langle K_tu,u\rangle \big] 
= 2\langle u - K_tu, u\rangle \\
&=2\int_{\R\, \cap \, \supp{u}} \big[u(x)-K_tu(x)\big]u(x)\,\dx \to 0,
\end{align*}
by the dominated convergence theorem.
\end{proof}

\begin{lemma}\label{K_t_scaling}
For $t>0$, $x\neq 0$, $k>0$, $A\subset\R$,
\begin{enumerate}
\item[(a)] $K_{t,n}(kx,kA) = K_{tk^{-\alpha},n}(x,A)$, $n=0,1,2,\ldots$,
\item[(b)] $K_t(kx,kA) = K_{tk^{-\alpha}}(x,A)$.
\end{enumerate}
\end{lemma}

\begin{proof}
    It suffices to show the equality (a). For $n=0$, from \eqref{eq:scaling2}, it follows that for $x>0$,
    \begin{align*}
    K_{t,0}(kx,kA) &= \int_{kA\cap D} p_t^D(kx,y)\,\dy = \int_{kA\cap D} k^{-1} p_{tk^{-\alpha}}^D(x,y/k)\,\dy \\
    &= \int_{A\cap D} p_{tk^{-\alpha}}^D(x,z)\,\dz = K_{tk^{-\alpha}, 0}(x,A).
    \end{align*}
    Similarly, for $x<0$, from \eqref{nu_scaling2},
    \[
    K_{t,0}(kx,kA) = \delta_{kx}(kA) e^{-\nu(kx,D)t} = \delta_x(A) e^{-\nu(x,D)tk^{-\alpha}} = K_{tk^{-\alpha},0}(x,A).
    \]

    \medskip
    Assume that for some $n\in \{ 0,1,\ldots\}$ we have $K_{t,n}(kx,kA) = K_{tk^{-\alpha},n}(x,A)$, $x\neq 0$, $t>0$, $k>0$.
    
    Let $x>0$. From Lemma \ref{proposition_recurrence}, \eqref{eq:scaling2} and \eqref{nu_scaling3}, we have the following equality
    \begin{align*}
    K_{t,n+1}(kx,kA) &=  \int_0^t \dr \int_D \dy \int_{D^c} \dz \, p^D_{r}(kx,y) \nu(y,z) K_{t-r,n}(z,kA) \\
    &=  \int_0^t \dr \int_D \dy \int_{D^c} \dz \, k^{-1} p^D_{rk^{-\alpha}}(x,y/k) k^{-\alpha-1} \nu(y/k, z/k) K_{t-r,n}(z,kA).
    \end{align*}
    Using the substitution $s = rk^{-\alpha}$, $w = y/k$ and $v = z/k$ we get
    \begin{align*}
    K_{t,n+1}(kx,kA) &= \int_0^{tk^{-\alpha}} \ds \int_D \dw \int_{D^c} \dv \, p^D_{s}(x,w) \nu(w, v) K_{t-sk^\alpha,n}(kv, kA),
    \end{align*}
    so
    \begin{align*}
    K_{t,n+1}(kx,kA) &= \int_0^{tk^{-\alpha}} \ds \int_D \dw \int_{D^c} \dv \, p^D_{s}(x,w) \nu(w, v) K_{tk^{-\alpha}-s,n}(v, A) \\
    &= K_{tk^{-\alpha}, n+1}(x,A).
    \end{align*}

    Now, consider the case $x<0$. Again, from Lemma \ref{proposition_recurrence}, \eqref{nu_scaling3} and \eqref{nu_scaling2}, we get
    \begin{align*}
    K_{t,n+1}(kx,kA) &= \int_0^t\dr \int_D \dz \,e^{-\nu(kx,D)r} \nu(kx,z) K_{t-r,n}(z,kA) \\
    &= \int_0^t\dr \int_D \dz \,e^{-\nu(x,D)rk^{-\alpha}} k^{-\alpha-1}\nu(x,z/k) K_{t-r,n}(z,kA).
    \end{align*}
    Using the substitution $s=rk^{-\alpha}$ and $w=z/k$ we get
    \begin{align*}
    K_{t,n+1}(kx,kA) &= \int_0^{tk^{-\alpha}} \ds \int_D \dw \,e^{-\nu(x,D)s} \nu(x,w) K_{t-sk^\alpha,n}(kw,kA),
    \end{align*}
    so
    \begin{align*}
    K_{t,n+1}(kx,kA) &= \int_0^{tk^{-\alpha}} \ds \int_D \dw \,e^{-\nu(x,D)s} \nu(x,w) K_{tk^{-\alpha}-s,n}(w,A) = K_{tk^{-\alpha},n+1}(x,A),
    \end{align*}    
    which yields our claim.
\end{proof}

\begin{proposition}\label{K_bounded_continuity}
    For every $t>0$ and $f\in C_b(\R_*)$, we have $K_tf\in C_b(\R_*)$.
\end{proposition}

\begin{proof}
    Let $f\in C_b(\R_*)$ and $t>0$. The boundedness of $K_tf$ is asserted in Lemma \ref{K_subprobability}. Therefore, it suffices to show that $K_tf$ is continuous on $\R_*$. 

    \medskip
    Assume that $x>0$. Then $x\mapsto \P_tf(x) = P_t^Df(x)$ is a continuous function on $D$, which follows directly from Lemma \ref{theorem_p^D_feller_1}. From Corollary \ref{perturbation_formula} it suffices to show a continuity of the function     
    \begin{align*}
        D\ni x\mapsto F_t(x) := \int_0^t \P_s \widehat{\nu} K_{t-s}f(x)\,\ds = \int_0^t \ds \int_D \dy \int_{D^c}\,\dz \, p_s^D(x,y)\nu(y,z)K_{t-s}f(z).
    \end{align*}
    From Lemma \ref{K_subprobability}, for $0\leq s\leq t$, $y\in D$ and $z\in D^c$, 
    \begin{align}\label{eq:K_feller_est_1}
        p_s^D(x,y)\nu(y,z)K_{t-s}f(z) \leq \norm{f}_\infty p_s^D(x,y)\nu(y,z).
    \end{align}
    Moreover, from Ikeda--Watanabe formula \eqref{eq:Ikeda_Watanabe} and \eqref{eq:p_survival},
    \begin{align*}
        G_t(x) := \int_0^t \ds \int_D \dy \int_{D^c}\,\dz \,p_s^D(x,y)\nu(y,z) = 1- p_t^D(x,D) = 1 - P_t^D\mathbf{1}(x).
    \end{align*}
    By Lemma \ref{theorem_p^D_feller_1}, the function $D\ni x\mapsto G_t(x)$ is continuous. From the Vitali's theorem (see \cite[Theorem 16.6]{Schilling2}) it follows that the function $p_t^D(x,y)\nu(y,z)$ is uniformly integrable in $x$ with respect to the measure $\ind_{[0,t]}(s)\ind_D(y)\ind_{D^c}(z)\,\ds\,\dy\,\dz$ (see \cite[Definition 16.1]{Schilling2}). By \eqref{eq:K_feller_est_1}, the functions $p_s^D(x,y)\nu(y,z)K_{t-s}f(z)$ are also locally in $x$ uniformly integrable. Since the function $D\ni x\mapsto p_s^D(x,y)$ is continuous, from the (direct proof of) Vitali's theorem, the function $D\ni x \mapsto F_t(x)$ is continuous. 

    \medskip
    Let $x<0$. Of course, $\overline{D}^c\ni x \mapsto \widehat{P}_t f(x) = f(x)e^{-\nu(x,D)t}$ is a continuous function. We will verify the continuity of the function 
    \[
    \overline{D}^c\ni x\mapsto F_t(x) := \int_0^t \P_s \widehat{\nu} K_{t-s}f(x)\,\ds = \int_0^t \ds \int_D \dy \, e^{-\nu(x,D)s} \nu(x,y) K_{t-s}f(y).
    \]
    Again from Lemma \ref{K_subprobability}, $K_{t-s}f(y) \leq \norm{f}_\infty$ and
    the function
    \[
    \overline{D}^c \ni x\mapsto \norm{f}_\infty \int_0^t \nu(x,D)e^{-\nu(x,D)s}\,\ds = \norm{f}_\infty \big[ 1-e^{-\nu(x,D)t}\big]
    \]
    is continuous. Hence, from Vitali's theorem, the functions $\norm{f}_\infty e^{-\nu(x,D)s}\nu(x,y)$ are locally uniformly in $x$ integrable with respect to the measure $\ind_{[0,t]}(s)\ind_D(y)\,\ds\,\dy$ and so are the functions $e^{-\nu(x,D)s}\nu(x,y)K_{t-s}f(y)$. By the (direct proof of) Vitali's theorem, we get the continuity of the function $\overline{D}^c\ni x \mapsto F_t(x)$.
\end{proof}

By the above proof, $K_tf\in C_b(D)$ for all $f\in\mathcal{B}_b(\R_*)$.

In passing, let us point out an alternative approach to concatenation by Bogdan and Kunze \cite{bogdan2024stableprocessesreflection}, where the process is constructed by means of the semigroup, see also Bogdan and Kunze \cite{KB_MK_2023} and Kim et al. \cite{MR4520527, MR4693939}.

\subsection{Excessive functions}

For $x\in\R$ and $\beta\in\R$, we define function $h_{\beta}(x) = |x|^\beta$. We first observe that $h_{\alpha-1}$ is \emph{supermedian} for $\widehat{P}_t$.

\begin{lemma}\label{Wniosek_ptD}
For $x\neq 0$ and $t>0$, $\widehat{P}_th_{\alpha-1}(x) \leq h_{\alpha-1}(x)$.
\end{lemma}

\begin{proof}
Recall that $\alpha\in (0,2)$. Let $\gamma\in (0,1)$ and $\delta\in (\gamma+\alpha/2, 1+\alpha/2)$. Of course, the interval is nonempty. Let 
\[
g(x) := \int_0^\infty \dt \int_D \dy \, p_t^D(x,y)f(t)y^{-\delta}, \quad x>0,
\]
where $f(t) :=
\mathcal{C}^{-1}t^{(-\alpha/2-\gamma+\delta)/\alpha}$ and
\[
\mathcal{C} = \int_0^\infty \dt \int_D \dy \, p_t^D(1,y)t^{(-\alpha/2-\gamma+\delta)/\alpha}y^{-\delta}.
\]
By the proof of \cite[Lemma 3.1]{MR4660833}, $\mathcal{C}<\infty$. From \cite[Lemma 3.4]{MR4660833}, we have $g(x) = x^{\alpha/2-\gamma}$, so by taking $\gamma = 1-\alpha/2 \in (0,1)$, we get $g(x) = h_{\alpha-1}(x)$, $x>0$.
Moreover, for $x>0$, we have
\begin{equation*}
    \begin{split}
        \widehat{P}_t h_{\alpha-1}(x) &= \int_0^\infty p_t^D(x,\dy) \int_0^\infty \ds \int_0^\infty \dz \, f(s) p_s^D(y,z) z^{-\delta} \\
        &= \int_0^\infty \ds \int_0^\infty \dz \, f(s) p_{t+s}^D(x,z) z^{-\delta} \\
        &= \int_t^\infty \du \int_0^\infty \dz \,  f(u-t) p_u^D(x,z) z^{-\delta} \\
        &\leq h_{\alpha-1}(x),
    \end{split}
\end{equation*}
because $f$ is increasing. 

For $x<0$, 
\[
\P_t h_{\alpha-1}(x) = e^{-\nu(x,D)t} h_{\alpha-1}(x) \leq h_{\alpha-1}(x). \qedhere
\]
\end{proof}

\begin{lemma}\label{harm_1} For $x\neq 0$,
\[
\int_0^\infty  \widehat{P}_r\widehat{\nu}h_{\alpha-1}(x)\, \dr = h_{\alpha-1}(x).
\]
\end{lemma}

\begin{proof}
Let $x>0$. By \eqref{eq:poisson_kernel},
\begin{equation*}
    \begin{split}
        &\int_0^\infty  \widehat{P}_r\widehat{\nu}h_{\alpha-1}(x)\, \dr = \int_0^\infty \dr \int_0^\infty \da \int_{-\infty}^0 \dy \, p_r^D(x,a)\nu(a,y) h_{\alpha-1}(y) \\
        &= \int_{-\infty}^0 P_{D}(x,y)h_{\alpha-1}(y)\,\dy.
    \end{split}
\end{equation*}
From \eqref{eq:poisson_kernel_halfline}, we have
\begin{align*}
\int_0^\infty  \widehat{P}_r\widehat{\nu}h_{\alpha-1}(x)\, \dr &= \frac{\sin(\pi\alpha/2)}{\pi} x^{\alpha/2} \int_{-\infty}^0  \frac{|y|^{\alpha/2-1}}{|x-y|}\,\dy = \frac{\sin(\pi\alpha/2)}{\pi} x^{\alpha/2} \int_0^\infty  \frac{y^{\alpha/2-1}}{x+y}\,\dy.
\end{align*}
By changing variables $y=xz$, we obtain that
\begin{align*}
 \int_0^\infty  \widehat{P}_r\widehat{\nu}h_{\alpha-1}(x)\, \dr &=\frac{\sin(\pi\alpha/2)}{\pi} x^{\alpha-1} \int_0^\infty \frac{z^{\alpha/2-1}}{1+z}\dz = \frac{\sin(\pi\alpha/2)}{\pi} x^{\alpha-1} \mathfrak{B}(\alpha/2, 1-\alpha/2) \\[3pt]
 &= \frac{\sin(\pi\alpha/2)}{\pi} x^{\alpha-1} \Gamma(\alpha/2)\Gamma(1-\alpha/2) = h_{\alpha-1}(x).
\end{align*}

Let $x<0$. Then from \eqref{nu_scaling},
\begin{align*}
    \int_0^\infty  \widehat{P}_r\widehat{\nu}h_{\alpha-1}(x)\, \dr &= \int_0^\infty e^{-\nu(x,D)r} \,\dr \int_0^\infty   \nu(x,y) y^{\alpha-1} \,\dy = \frac{1}{\nu(x,D)} \int_0^\infty   \nu(x,y) y^{\alpha-1} \,\dy \\
    &= \alpha |x|^\alpha \int_0^\infty \frac{y^{\alpha-1}}{|x-y|^{\alpha+1}}\,\dy.
\end{align*}
Using the substitution $y=|x|z$, we get
\begin{align*}
    \int_0^\infty  \widehat{P}_r\widehat{\nu}h_{\alpha-1}(x)\, \dr &= \alpha |x|^{\alpha-1} \int_0^\infty \frac{z^{\alpha-1}}{(z+1)^{\alpha+1}}\,\dz = \alpha |x|^{\alpha-1} \mathfrak{B}(\alpha,1) = h_{\alpha-1}(x),
\end{align*}
which is our claim.
\end{proof}

From Lemma \ref{lem:L1} and Lemma \ref{harm_1}, it follows that $\mE_x \big[h_{\alpha-1}(X_{R_1})\big] = h_{\alpha-1}(x)$.

The next result is an analogue of Lemma \ref{K_subprobability}, which asserts that $K_t\mathbf{1} \leq 1$ on $\R_*$.

\begin{proposition}\label{theorem_excessive_function}
For $\alpha\in (0,2)$, the function $h_{\alpha-1}(x) = |x|^{\alpha-1}$ is supermedian for $K_t$, i.e. 
\[
K_th_{\alpha-1}(x) \leq h_{\alpha-1}(x), \qquad t>0, ~x\neq 0.
\]
\end{proposition}

\begin{proof}
We prove by induction that for $N=0,1,2,\ldots$ and $t>0$,
\begin{equation}\label{S_N}
S_N(x,t) := \sum_{n=0}^N K_{t,n} h_{\alpha-1}(x) \leq h_{\alpha-1}(x), \quad x\neq 0.
\end{equation}
Let $N=0$. From Lemma \ref{Wniosek_ptD}, it follows that
\begin{equation*}
    \begin{split}
        S_0(x,t)  = \widehat{P}_th_{\alpha-1}(x) \leq h_{\alpha-1}(x), \quad x\neq 0.
    \end{split}
\end{equation*}
Furthermore, using Lemma \ref{harm_1}, we obtain for $x\neq 0$,
\begin{align}\label{eq2}
    \P_th_{\alpha-1}(x) = \int_0^\infty \P_t\P_r\widehat{\nu} h_{\alpha-1}(x)\,\dr = \int_0^\infty \P_{t+r}\widehat{\nu}h_{\alpha-1}(x)\,\dr = \int_t^\infty \P_r\widehat{\nu}h_{\alpha-1}(x)\,\dr.
\end{align}

Assume that for some $N\in \{ 0,1,2,\ldots\}$ and all $t>0$ we have $S_N(x,t) \leq h_{\alpha-1}(x)$. We will show that $S_{N+1}(x,t) \leq h_{\alpha-1}(x)$. Using \eqref{eq:K_recurence}, \eqref{eq2} and Lemma \ref{harm_1}, we obtain
\begin{align*}
    S_{N+1}(x,t) &= \P_th_{\alpha-1}(x)+ \sum_{n=0}^N \int_0^t \P_r\widehat{\nu} 
     K_{t-r,n}h_{\alpha-1}(x)\,\dr \\
    &=\int_t^\infty \P_r\widehat{\nu}h_{\alpha-1}(x)\,\dr + \int_0^t  \P_r\widehat{\nu} S_N(x,t-r)\,\dr \\
    &\leq \int_t^\infty \P_r\widehat{\nu}h_{\alpha-1}(x)\,\dr + \int_0^t \P_r\widehat{\nu}h_{\alpha-1}(x)\,\dr = h_{\alpha-1}(x).
\end{align*}
Then from the fact that the sequence $(S_N(x,t))_{N\geq 0}$ is non-decreasing and bounded, it follows that  
\[
K_th_{\alpha-1}(x) \leq h_{\alpha-1}(x), \quad x>0. \qedhere
\]
\end{proof}

\begin{corollary}\label{excesive_function_beta}
If $\alpha\in (0,2)$, $\beta(\alpha-\beta-1)\geq 0$ and $\lambda\in [0,\infty)$, then $h_\beta$ is $\lambda$-excessive for $K_t$.
\end{corollary}

\begin{proof}
The supermedian property in the case $\beta=0$ and $\beta=\alpha-1$ follows directly from Lemma \ref{K_subprobability} and Proposition \ref{theorem_excessive_function}. Hence, it suffices to prove it for $\beta\neq 0$ and $\beta\neq\alpha-1$.

    Let $\gamma := (\alpha-1)/\beta$. Of course $\gamma >1$ and from Jensen's inequality, we have
    \begin{align}\label{eq:excesywnosc_beta}
        K_th_\beta^\gamma(x) &= K_t(x,\R) \int_\R h_\beta^\gamma(y) \frac{K_t(x,\dy)}{K_t(x,\R)} \geq K_t(x,\R) \Big[ \int_\R h_\beta(y) \frac{K_t(x,\dy)}{K_t(x,\R)} \Big]^\gamma \nonumber \\[2pt]
        &= \big[K_t(x,\R)\big]^{1-\gamma} \big[K_th_\beta(x)\big]^\gamma \geq \big[K_th_\beta(x)\big]^\gamma.
    \end{align}
Hence, from \eqref{eq:excesywnosc_beta} and from Proposition \ref{theorem_excessive_function},
    \[
        K_th_\beta(x) \leq \big[ K_th_{\beta}^\gamma(x) \big]^{1/\gamma} = \big[ K_t h_{\alpha-1}(x)\big]^{1/\gamma} \leq h_{\alpha-1}^{1/\gamma}(x) = h_\beta(x). 
    \]
    
\medskip
Let $x\neq 0$ and $\varphi\in C_0(\R_*)$ be such that $0\leq \varphi\leq h_\beta$ and $\varphi(x) = h_\beta(x)$. By Lemma \ref{P_t_feller} and the first part of the proof, we have
\[
h_\beta(x) = \varphi(x) = \liminf_{t\to 0^+} \widehat{P}_t\varphi(x) \leq  \liminf_{t\to 0^+} \widehat{P}_th_\beta(x) \leq \limsup_{t\to 0^+} \widehat{P}_th_\beta(x) \leq \limsup_{t\to 0^+} K_th_\beta(x)\leq h_\beta(x).
\]
Hence $\lim\limits_{t\to 0^+} \widehat{P}_th_\beta(x) = h_\beta(x)$. Furthermore, 
\[
\widehat{P}_th_\beta(x) \leq K_th_\beta(x) \leq h_\beta(x),
\]
so $\lim\limits_{t\to 0^+} K_th_\beta(x) = h_\beta(x)$. 
\end{proof}

Note that the assumption $\alpha\in (0,2)$ and $\beta(\alpha-\beta-1)\geq 0$ in the previous corollary is equivalent to the following: $\beta\in [\alpha-1, 0]$ for $\alpha \in (0,1]$, and $\beta\in [0,\alpha-1]$ for $\alpha\in [1,2)$.

\section{The lifetime of the process 
}\label{chap_Lifetime}

In this section we study the lifetime of the process $X = (X_t)_{t\geq 0}$. We consider three cases. In the case $\alpha\in (0,1)$, we prove that the lifetime of the process $X$ is infinite a.s. and $|X_t|\to \infty$ as $t\to \infty$. For $\alpha\in (1,2)$, we show that the process $X$ is absorbed by the origin in finite time. In the case $\alpha=1$, we prove that the lifetime of $X$ is infinite and the limit of $\lim\limits_{t\to\infty} X_t$ does not exist. The precise statement of this fact is given in Section \ref{sec:main_theorem}. At the end of this section, we use the absorption result to prove the Feller property of the semigroup $K$ for $\alpha>1$.

\subsection{The position of the first return \texorpdfstring{$D$}{TEXT}}

Let $W = X_{R_2}$. It is the first return position to $D = (0,\infty)$ for the process $X = (X_t)_{t\geq 0}$ starting from $x>0$. By Lemma \ref{lem:L2},
\begin{align*}
    R(x,w) := \int_0^\infty \dt \int_0^\infty\dy \int_{-\infty}^0\dz  \, p_t^D(x,y)\nu(y,z) \frac{\nu(z,w)}{\nu(z,D)}, \quad w\in D,
\end{align*}
is the density function of $W$. By \eqref{eq:green_fun} and \eqref{eq:poisson_kernel},
\begin{align*}
    R(x,w) = \int_{-\infty}^0 P_D(x,z)\frac{\nu(z,w)}{\nu(z,D)}\,\dz.
\end{align*}
By changing variables $z=xs$,  \eqref{poisson_kernel_scaling} and \eqref{nu_scaling2}, we get
\begin{align*}
    R(x,w) &= x\int_{-\infty}^0 P_{D}(x,xs)\frac{\nu(xs,w)}{\nu(xs,D)}\,\ds = x^{1-1-\alpha-1+\alpha} \int_{-\infty}^0  P_D(1,s) \frac{\nu(s,w/x)}{\nu(s,D)}\,\ds \\
    &= x^{-1} R(1,w/x).
\end{align*}
Therefore, $\mP_x(W\in \dw) = \mP_1(xW\in \dw)$.

\begin{lemma}\label{lem:L3}
For $\alpha \in (0,1)\cup (1,2)$,
\begin{align*}
    \rho := \mE_1 W^{(\alpha-1)/2} <1.
\end{align*}
\end{lemma}

\begin{proof}
Let $\alpha\in (0,2)$. Using \eqref{eq:poisson_kernel_halfline} and \eqref{nu_scaling} we obtain
\begin{align*}
\mE_1 W^{(\alpha-1)/2} &= C \int_0^\infty \dw \int_{-\infty}^0 \dz \, |z|^{-\alpha/2} |1-z|^{-1} |z-w|^{-1-\alpha} |z|^\alpha w^{(\alpha-1)/2},
\end{align*}
where $C = \pi^{-1}\alpha \sin(\pi\alpha/2)$. By Tonelli's theorem and by changing variables $y = -z \in (0,\infty)$, $w = yv$, we get
\begin{align*}
    \mE_1 W^{(\alpha-1)/2} &= C \int_0^\infty \dy \int_0^\infty \dv \, y^{1-\alpha/2-1-\alpha+\alpha+\alpha/2-1/2} (1+y)^{-1} (v+1)^{-1-\alpha} v^{(\alpha-1)/2} \\
    &= C \int_0^\infty \frac{1}{\sqrt{y}(1+y)}\,\dy \, \int_0^\infty \frac{v^{(\alpha-1)/2}}{(v+1)^{1+\alpha}} \,\dv  \\
    &= \frac{\alpha}{\pi} \sin\frac{\pi\alpha}{2} \cdot \pi \cdot B\big(\alpha/2+1/2, \alpha/2+1/2)  \\
    &=\sin{\frac{\pi\alpha}{2}}\cdot  \frac{(\Gamma(\alpha/2+1/2))^2}{\Gamma(\alpha)}\\
    &\leq \frac{(\Gamma(\alpha/2+1/2))^2}{\Gamma(\alpha)} =: G(\alpha).
\end{align*}

We will show that $G(\alpha) < G(1) = 1$ for $\alpha\in (0,1)\cup (1,2)$. Let $\psi(x) := \Gamma'(x)/\Gamma(x)$ be the digamma function. It is well known that $x\mapsto \psi(x)$ is continuous and increasing for $x>0$ (see, e.g.,  Andrews et al. \cite[Theorem 1.2.5]{MR1688958}). 

In case $\alpha\in (0,1)$ we have $\alpha/2+1/2 > \alpha$ and then $\psi(\alpha/2+1/2) > \psi(\alpha)$. Then, of course, $\Gamma'(\alpha/2+1/2)\Gamma(\alpha) - \Gamma(\alpha/2+1/2)\Gamma'(\alpha) > 0$. Therefore,
\[
G'(\alpha) = \frac{\Gamma(\alpha/2+1/2)}{\Gamma^2(\alpha)} \big[\Gamma'(\alpha/2+1/2)\Gamma(\alpha) - \Gamma(\alpha/2+1/2) \Gamma'(\alpha)\big]>0,
\]
and so $G(\alpha)<G(1) = 1$.

In case $\alpha\in (1,2)$ we have $\alpha/2+1/2 < \alpha$ and then $\psi(\alpha/2+1/2)<\psi(\alpha)$, or $\Gamma'(\alpha/2+1/2)\Gamma(\alpha) - \Gamma(\alpha/2+1/2)\Gamma'(\alpha) <0$. Therefore, $G'(\alpha)<0$ and again $G(\alpha)<G(1) = 1$.
\end{proof}

\begin{proposition}\label{cor:wart_ocz}
For $\alpha\in (0,2)$, $\mE_1|\ln{W}| < \infty$ and
\begin{align}
\mE_1 \ln{W} &> 0, \quad \mathrm{ for } ~ \alpha\in (0,1), \label{eq:ln(W)_01} \\
\mE_1 \ln{W} &= 0, \quad \mathrm{ for } ~ \alpha=1, \label{eq:ln(W)_1} \\
\mE_1 \ln{W} &< 0, \quad \mathrm{ for } ~ \alpha\in (1,2). \label{eq:ln(W)_12}
\end{align}
\end{proposition}

\begin{proof}
First, we prove that $\mE_1|\ln{W}|<\infty$. From  \eqref{eq:poisson_kernel_halfline} and \eqref{nu_scaling} we get
\begin{align*}
\mE_1 |\ln{W}| &= C \int_0^\infty \dw \int_{-\infty}^0 \dz \, |z|^{-\alpha/2} |1-z|^{-1} |z-w|^{-1-\alpha} |z|^\alpha |\ln{w}|,
\end{align*}
where $C = \pi^{-1}\alpha \sin(\pi\alpha/2)$. By Tonelli's theorem and by changing variables $y = -z \in (0,\infty)$, we get
\begin{align*}
    \mE_1 |\ln{W}| = C \int_0^\infty \dy \int_0^\infty \dw  \, \frac{y^{\alpha/2}|\ln{w}|}{(y+1)(w+y)^{\alpha+1}}.
\end{align*}
By changing variables $w = yv$,
\begin{align*}
    \mE_1|\ln{W}| &=C \int_0^\infty \dy \int_{0}^\infty \dv \, \frac{|\ln{(yv)|}}{y^{\alpha/2}(y+1)(v+1)^{\alpha+1}}  \\
    &\leq C \int_0^\infty \dy \int_0^\infty \dv \, \frac{|\ln{y}| + |\ln{v}|}{y^{\alpha/2}(y+1)(v+1)^{\alpha+1}} \\
    &= 
    C\int_0^\infty  \frac{1}{(v+1)^{\alpha+1}}\, \dv 
    \int_0^\infty \frac{|\ln{y}|}{y^{\alpha/2}(y+1)}\, \dy 
    \\
    &+ 
    C\int_0^\infty \frac{|\ln{v}|}{(v+1)^{\alpha+1}}\,\dv
    \int_0^\infty \frac{y^{-\alpha/2}}{y+1}\, \dy  < \infty.
\end{align*}
Thus, $\mE_1|\ln{W}|<\infty$. 

Let $\alpha\in (0,1)\cup (1,2)$. Then, from Jensen's inequality for concave functions and from Lemma \ref{lem:L3}, we get
\[
\tfrac{\alpha-1}{2} \mE_1\ln{W} = \mE_1\ln{W^{(\alpha-1)/2}} \leq \ln{\mE_1{W^{(\alpha-1)/2}}} < \ln{1} = 0.
\]
Hence, for $\alpha\in (0,1)$, $\mE_1 \ln{W} > 0$, and for $\alpha \in (1,2)$, $\mE_1\ln{W} < 0$.

Now assume that $\alpha=1$. Then 
\begin{align}\label{alfa0}
C^{-1}\mE_1(\ln{W}) &= \int_0^\infty \dy \int_0^1 \dw  \, \frac{y^{1/2}\ln{w}}{(y+1)(w+y)^{2}} + \int_0^\infty \dy \int_1^\infty \dw  \, \frac{y^{1/2}\ln{w}}{(y+1)(w+y)^{2}} \nonumber \\
&=: I+II. 
\end{align}
Using substitution $w = 1/v$ and integrating by parts we get
\begin{align*}
I &= \int_0^\infty \dy \int_0^1 \dw  \, \frac{\sqrt{y}\ln{w}}{(y+1)(w+y)^{2}} = - \int_0^\infty \frac{\sqrt{y}}{y+1} \int_1^\infty \frac{\ln{v}}{(1+yv)^2}\,\dv \,\dy \\
&=- \int_0^\infty \frac{\ln(1+1/y)}{\sqrt{y}(y+1)}\,\dy = -\int_0^\infty \frac{\ln(y+1)}{\sqrt{y}(y+1)}\,\dy + \int_0^\infty \frac{\ln(y)}{\sqrt{y}(y+1)}\,\dy.
\end{align*}
Moreover, by changing variables $y = 1/v$,
\[
\int_0^1 \frac{\ln(y)}{\sqrt{y}(y+1)}\,\dy = -\int_1^\infty \frac{\ln{v}}{\sqrt{v}(v+1)}\,\dv, 
\]
hence
\begin{align}\label{alfa0_I}
I = -\int_0^\infty \frac{\ln(y+1)}{\sqrt{y}(y+1)}\,\dy.
\end{align}
By changing variables $w = yz$ in the integral $II$, we obtain
\begin{align}\label{alfa0_II}
II &= \int_0^\infty \dy \int_{1/y}^\infty \dz \, \frac{\ln{y}+\ln{z}}{\sqrt{y}(y+1)(z+1)^2} \nonumber\\
&= \int_0^\infty  \frac{\ln{y}}{\sqrt{y}(y+1)}\int_{1/y}^\infty \frac{\dz}{(z+1)^2}\,\dy + \int_0^\infty \frac{1}{\sqrt{y}(y+1)}\int_{1/y}^\infty \frac{\ln{z}}{(z+1)^2}\,\dz \,\dy \nonumber \\
&= \int_0^\infty  \frac{\sqrt{y}\ln{y}}{(y+1)^2}\,\dy + \int_0^\infty \frac{1}{\sqrt{y}(y+1)}\Big(\ln(y+1) - \frac{y\ln{y}}{y+1}\Big) \,\dy \nonumber \\
&= \int_0^\infty \frac{\ln(y+1)}{\sqrt{y}(y+1)}\,\dy.
\end{align}
From \eqref{alfa0}, \eqref{alfa0_I} and \eqref{alfa0_II} we get $\mE_1(\ln{W}) = 0$.
\end{proof}

\begin{lemma}\label{lem:E_log_square}
    For $\alpha = 1$, we have $\sigma^2 := \mE_1 \big[\ln^2{W}\big] = 4\pi^2/3$.
\end{lemma}

\begin{proof}
    We observe that
    \begin{align*}
        \mE_1\big[ \ln^2{W}\big] &= \int_0^\infty R(1,w) \ln^2{w}\,\dw = \frac{1}{\pi} \int_0^\infty \dw \int_{0}^\infty  \dz \,  \frac{z^{1/2} \ln^2{w}}{(1+z)(z+w)^2}.
    \end{align*}
    By the substitution $y=z^{1/2}$, we get
    \begin{align*}
        \mE_1\big[ \ln^2{W}\big] = \frac{1}{\pi} \int_0^\infty \dw \int_{0}^\infty \dy \frac{2y^2 \ln^2{w}}{(y^2+1)(y^2+w)^2} = \frac{1}{\pi} \int_0^\infty \dw \int_{-\infty}^\infty \dy \frac{y^2 \ln^2{w}}{(y^2+1)(y^2+w)^2}.
    \end{align*}
    Using the Cauchy's residue theorem, we obtain that
    \begin{align*}
        \int_{-\infty}^\infty \frac{y^2}{(y^2+1)(y^2+w)^2}\,\dy = \frac{\pi}{2\sqrt{w} (\sqrt{w}+1)^2}, \qquad w>0.
    \end{align*}
    Hence, also by the substitution $u=\sqrt{w}$,
    \begin{align*}
        \mE_1\big[ \ln^2{W}\big] &= \int_0^\infty \frac{\ln^2{w}}{2\sqrt{w} (\sqrt{w}+1)^2} \, \dw = 4 \int_0^\infty \frac{\ln^2{u}}{(u+1)^2}\,\du.
    \end{align*}
    By the substitution $u=e^z$, we get
    \begin{align*}
        \mE_1\big[ \ln^2{W}\big] = 4 \int_{-\infty}^\infty \frac{z^2e^z}{(e^z+1)^2}\,\dz = 8 \int_{0}^\infty \frac{z^2e^z}{(e^z+1)^2}\,\dz.
    \end{align*}
    Integrating by parts we obtain the equality
    \begin{align*}
        \mE_1\big[ \ln^2{W}\big] = 16 \int_0^\infty \frac{z}{e^z+1}\,\dz.
    \end{align*}
    From Bateman \cite[(1.12.5), p. 32]{Bateman} it follows that 
    \begin{equation*}
        \mE_1\big[ \ln^2{W}\big] = 8 \zeta(2) = \frac{4}{3}\pi^2. \qedhere
    \end{equation*}
\end{proof}

\subsection{Consecutive return positions to \texorpdfstring{$D$}{} and their limit}

For $i = 1,2,\ldots$, we define random variables
\begin{align}\label{def_W_i}
    W_i = \frac{X_{R_{2i}}}{X_{R_{2i-2}}}.
\end{align}
Recall that $R_0=0$ and then $W_1 = X_{R_2} / X_0$. 
Note that
\begin{align}
    W_i = W_1\circ\theta_{R_{2i-2}}.
\end{align}
Moreover, for $j<i$ we have
\begin{align}
    W_i = W_{i-j}\circ\theta_{R_{2j}}.
\end{align}
Indeed,
\[
W_i = W_1\circ\theta_{R_{2i-2}} = W_1 \circ \theta_{R_{2(i-j)-2}}\circ\theta_{R_{2j}} = W_{i-j}\circ\theta_{R_{2j}}.
\]

\begin{lemma}\label{lem_Wi_1}
Under $\mP_x$, $x>0$, every random variable $W_i$,  $i=1,2,\ldots$, has the same distribution and the density function of $W_i$ is given by $R(1,w)$. In particular, the distribution does not depend on $x$.
\end{lemma}

\begin{proof}
For $i=1,2,\ldots$ and a bounded function $f$, by the strong Markov property, we have
\begin{align*}
    \mE_x[f(W_i)] &= \mE_x\big[f\big(W_1 \circ \theta_{R_{2i-2}}\big)\big]  = \mE_x\Big[ \mE\Big[ f\big( W_1 \circ \theta_{R_{2i-2}}\big) \Big| \mathcal{F}_{R_{2i-2}}\Big]\Big] \\
    &= \mE_x\Big[ \mE_y\Big[ f\big(W_1\big)\Big] \Big|_{y = X_{R_{2i-2}}}\Big]  = \mE_x\Big[ \mE_y\Big[ f\Big(\frac{X_{R_2}}{X_0}\Big)\Big] \Big|_{y = X_{R_{2i-2}}}\Big] \\
    &= \mE_x\Big[ \mE_y\Big[ f\Big(\frac{W}{y}\Big)\Big] \Big|_{y = X_{R_{2i-2}}}\Big] =  \mE_x\big[ \mE_1[ f(W)]\big] \\
    &= \mE_1[f(W)],
\end{align*}
which proves the lemma.
\end{proof}

\begin{lemma}\label{lem_Wi_iid}
The random variables $\{W_i\}_{i\in\mathbb{N}}$ are independent under $\mathbb{P}_x$, $x>0$.
\end{lemma}

\begin{proof}
Let $f_i$, $i=1,2,\ldots$, be bounded functions. From Lemma \ref{lem_Wi_1} it suffices to show that for $x>0$, $n\in\mathbb{N}$, $i_1<i_2<\ldots<i_n$,
\begin{align}\label{lem_Wi_eq_n}
    \mE_x\big[ f_1(W_{i_1})f_2(W_{i_2})\ldots f_n(W_{i_n})\big] = \mE_x[f_1(W_1)]\mE_x[f_2(W_1)]\ldots\mE_x[f_n(W_1)].
\end{align}

We first show that for $i_1<i_2$,
\begin{align}\label{lem_Wi_two}
    \mE_x\big[ f_1(W_{i_1}) f_2(W_{i_2})\big] = \mE_x\big[f_1(W_1)]\mE_x[f_2(W_1)].
\end{align}
Indeed, from the strong Markov property and Lemma \ref{lem_Wi_1} we get
\begin{align*}
    \mE_x\big[ f_1(W_{i_1}) f_2(W_{i_2})\big] &= \mE_x\Big[ \mE\Big[f_1\big(W_{i_1}\big) f_2\big(W_{i_2-i_1}\circ \theta_{R_{2i_1}} \big) \Big| \mathcal{F}_{R_{2i_1}}\Big]\Big] \\
    &= \mE_x\Big[  f_1\big(W_{i_1}\big) \mE\Big[ f_2\big(W_{i_2-i_1}\circ \theta_{R_{2i_1}}\big) \Big| \mathcal{F}_{R_{2i_1}}\Big]\Big]\\
    &= \mE_x\Big[  f_1(W_{i_1}) \mE_y\big[ f_2(W_{i_2-i_1})\big] \Big|_{y=X_{R_{2i_1}}}\Big]\\
    &= \mE_x\Big[  f_1(W_{i_1}) \mE_y\big[ f_2(W_{1})\big] \Big|_{y=X_{R_{2i_1}}}\Big]\\
    &= \mE_x\Big[  f_1(W_{i_1}) \mE_y\Big[ f_2\Big(\frac{X_{R_2}}{y} \Big)\Big] \Big|_{y=X_{R_{2i_1}}}\Big] \\
    &= \mE_x\Big[  f_1(W_{i_1}) \mE_1[ f_2(X_{R_2})]\Big] \\
    &= \mE_1[ f_2(X_{R_2})] \mE_x[f_1(W_{i_1})] \\
    &= \mE_x[f_1(W_1)] \mE_x[f_2(W_1)].
\end{align*}

Assume now that the equality \eqref{lem_Wi_eq_n} holds for $n\geq 2$. We will prove it for $n+1$,
\begin{align*}
    &\mE_x[ f_1(W_{i_1}) f_2(W_{i_2}) \ldots f_{n+1}(W_{i_{n+1}})] \\
    &=  \mE_x\big[ f_1(W_{i_1}) \mE\big[f_2(W_{i_2-i_1} \circ \theta_{R_{2i_1}}) \ldots f_{n+1}(W_{i_{n+1}- i_1} \circ \theta_{R_{2i_1}}) \big| \mathcal{F}_{R_{2i_1}} \big]\big] \\
    &= \mE_x\big[ f_1(W_{i_1}) \mE_y\big[f_2(W_{i_2-i_1} ) \ldots f_{n+1}(W_{i_{n+1}- i_1} )\big] \big|_{y = X_{R_{2i_1}}}\big] \\
    &= \mE_x\big[ f_1(W_{i_1}) \mE_y[f_2(W_1)]\ldots\mE_y[f_{n+1}(W_1)] \big|_{y = X_{R_{2i_1}}}\big] \\
    &= \mE_x\Big[ f_1(W_{i_1}) \mE_y\Big[f_2\Big( \frac{X_{R_2}}{y}\Big)\Big]\ldots\mE_y\Big[f_{n+1}\Big( \frac{X_{R_2}}{y}\Big)\Big] \Big|_{y = X_{R_{2i_1}}}\Big] \\
    &= \mE_x\big[ f_1(W_{i_1}) \mE_1[f_2( X_{R_2})]\ldots\mE_1[f_{n+1}( X_{R_2})] \big]\\
    &= \mE_x[ f_1(W_{i_1})] \mE_1[f_2( X_{R_2})]\ldots\mE_1[f_{n+1}( X_{R_2})]\\
    &= \mE_x[ f_1(W_1)] \mE_x[f_2( W_1)]\ldots\mE_x[f_{n+1}(W_1)],
\end{align*}
which is our claim.
\end{proof}

From Lemma \ref{lem_Wi_1} and Lemma \ref{lem_Wi_iid} it follows that $\{W_i\}_{i\in\mathbb{N}}$ are i.i.d. 

Now we consider consecutive return positions $V_n$, $n=0,1,2,\ldots$, of the process $X$ to $D$, starting from $x>0$. Thus, $V_0=X_0 = x$ and for $n\
\geq 1$, $V_n = X_{R_{2n}}$. From \eqref{def_W_i} it follows that for $n\geq 1$, $V_n = W_nV_{n-1}$. In other words, there exist i.i.d. random variables $W_i$, $i=1,2,\ldots$, with the density function $R(1,w)$ such that for $n\geq 1$, $V_n = X_0 \prod_{i=1}^n W_i$.

\begin{proposition}\label{Th:hitting1}
The following statements hold $\mP_x$-a.s. for every $x>0$.
\begin{enumerate}
    \item If $\alpha\in (0,1)$, then $\lim\limits_{n\to\infty} X_{R_{2n}} = \lim\limits_{n\to\infty} V_n = +\infty$.
    \item If $\alpha\in (1,2)$, then $\lim\limits_{n\to\infty} X_{R_{2n}} = \lim\limits_{n\to\infty} V_n = 0$.
    \item If $\alpha=1$, then 
    \[
    \liminf_{n\to\infty} X_{R_{2n}} = 0, \qquad \limsup_{n\to\infty} X_{R_{2n}} = +\infty.
    \]
\end{enumerate}
\end{proposition}

\begin{proof}
Let $\alpha\in (0,1)$. By the Strong Law of Large Numbers and Proposition \ref{cor:wart_ocz}, it follows that \[
\ln\prod_{i=1}^n W_i = \sum_{i=1}^n \ln{W_i} \to +\infty ~~\mathrm{a.s.,}
\]
as $n\to\infty$, which implies the desired convergence. 

Now let $\alpha\in (1,2)$. Again, by the Strong Law of Large Numbers and Proposition \ref{cor:wart_ocz} it follows that \[
\ln\prod_{i=1}^n W_i = \sum_{i=1}^n \ln{W_i} \to -\infty ~~\mathrm{a.s.,}
\]
as $n\to\infty$, which implies the desired convergence.

Now assume that $\alpha=1$. Let $S_n' := \sum_{i=1}^n \ln{W_i}$. From Proposition \ref{cor:wart_ocz} we know that $\mE_1\ln{W} = 0$ and from Lemma \ref{lem:E_log_square}, $\mE_1 |\ln{W}|^2 = \sigma^2 \in (0,\infty)$. Hence, from the Law of the Iterated Logarithm (see Hartman and Wintner \cite{HartmanWinter} or Acosta \cite{MR690128}), it follows that
\begin{align*}
    \limsup_{n\to\infty} \frac{S_n'}{\sqrt{2n \ln\ln{n}}} = \sigma >0, \qquad \mathrm{and}\qquad \liminf_{n\to\infty} \frac{S_n'}{\sqrt{2n \ln\ln{n}}} = -\sigma <0.
\end{align*}
Therefore,
\begin{align*}
    \limsup_{n\to\infty} \sum_{i=1}^n \ln{W_i} = +\infty, \qquad \liminf_{n\to\infty} \sum_{i=1}^n \ln{W_i} = -\infty,
\end{align*}
and then $\limsup\limits_{n\to\infty} V_n = +\infty$, and $\liminf\limits_{n\to\infty} V_n = 0$.
\end{proof}

\medskip
\subsection{The lifetime of the process \texorpdfstring{$X$}{}}

Let $T = R_2$ be the random time of the first return to $D = (0,\infty)$ of the process $X = (X_t)_{t\geq 0}$ starting from $x>0$. Then, by Lemma \ref{lem:L2},
\begin{align*}
    S(x,t) := \int_0^t \dr \int_0^\infty \da \int_{-\infty}^0 \db \, p_r^D(x,a)\nu(a,b)\nu(b,D)e^{-\nu(b,D)(t-r)}
\end{align*}
is the density function of $T$. By changing variables $a=xc$, $b=xe$, $r=x^\alpha s$ and using \eqref{eq:scaling2}, \eqref{nu_scaling3} and \eqref{nu_scaling2}, we get
\begin{align*}
    S(x,t) &= \int_0^{tx^{-\alpha}} \ds \int_0^\infty \dc  \int_{-\infty}^0 \de \,
    x^{\alpha+1+1} p_{x^\alpha s}^D(x,xc)\nu(xc,xe)\nu(xe,D)e^{-\nu(xe,D)(t-x^\alpha s)} \\
    &= \int_0^{tx^{-\alpha}} \ds \int_0^\infty \dc  \int_{-\infty}^0 \de \,
    x^{\alpha+2} x^{-1} p_{ s}^D(1,c)x^{-1-\alpha}\nu(c,e) x^{-\alpha}\nu(e,D)e^{-\nu(e,D)(tx^{-\alpha}-s)} \\
    &= x^{-\alpha} \int_0^{tx^{-\alpha}} \ds \int_0^\infty \dc  \int_{-\infty}^0 \de \,
      p_{ s}^D(1,c)\nu(c,e) \nu(e,D)e^{-\nu(e,D)(tx^{-\alpha}-s)} \\
    &=x^{-\alpha} S(1,tx^{-\alpha}).
\end{align*}
Therefore, $\mP_x(T\in dt) = \mP_1(x^\alpha T\in dt)$.

\medskip
For $n=1,2,\ldots$, we define random variables
\begin{align}\label{def_T_n}
    T_n = \frac{R_{2n} - R_{2n-2}}{X_{R_{2n-2}}^\alpha}.
\end{align}
Note that for $n\geq 1$, $R_{2n} - R_{2n-2} = R_2\circ\theta_{R_{2n-2}}$, hence 
\begin{align}
T_n = T_1\circ\theta_{R_{2n-2}}.  
\end{align}
Moreover, for $k<n$,
\begin{align}
    T_{n} = T_{n-k}\circ\theta_{R_{2k}}.
\end{align}
Indeed,
\[
T_n = T_1\circ\theta_{R_{2n-2}} = T_1 \circ\theta_{R_{2(n-k)-2}} \circ\theta_{R_{2k}} = T_{n-k} \circ\theta_{R_{2k}}.
\]
\begin{lemma}\label{lem_T_n_1}
Under $\mP_x$, $x>0$, every random variable $T_n$, $n=1,2,\ldots$, has the same distribution and the density function of $T_n$ is given by $S(1,t)$. In particular, the distribution does not depend on $x$.
\end{lemma}

\begin{proof}
For a bounded or non-negative function $f$, by  the strong Markov property we have
\begin{align*}
    \mE_x[f(T_n)] &= \mE_x[ \mE[f( T_1\circ \theta_{R_{2n-2}})| \mathcal{F}_{R_{2n-2}}]] = \mE_x\big[\mE_y[f(T_1)] \big|_{y=X_{R_{2n-2}}}\big] \\
    &= \mE_x\Big[\mE_y\Big[f\Big(\frac{R_2}{X_0^\alpha}\Big)\Big] \Big|_{y=X_{R_{2n-2}}}\Big] = \mE_x\Big[\mE_y\Big[f\Big(\frac{T}{y^\alpha}\Big)\Big] \Big|_{y=X_{R_{2n-2}}}\Big] \\
    &= \mE_x\big[\mE_1[f(T)] \big] = \mE_1[f(T)],
\end{align*}
which proves the lemma.
\end{proof}

\begin{lemma}\label{lem_T_n_2}
The random variables $\{T_n\}_{n\in\mathbb{N}}$ are independent under $\mathbb{P}_x$, $x>0$.
\end{lemma}

\begin{proof}
Let $f_i$, $i=1,2,\ldots$, be bounded or non-negative functions. From Lemma \ref{lem_T_n_1} it suffices to show that for $x>0$, $k\in\mathbb{N}$, $n_1<n_2<\ldots<n_k$,
\begin{align}\label{lem_Ti_eq_n}
    \mE_x\big[ f_1(T_{n_1})f_2(T_{n_2})\ldots f_k(T_{n_k})\big]  = \mE_x[f_1(T_1)]\mE_x[f_2(T_1)]\ldots\mE_x[f_k(T_1)].
\end{align}

First, we will show that for $n_1<n_2$, 
\begin{align}
    \mE_x\big[ f_1(T_{n_1})f_2(T_{n_2})\big]  = \mE_x[f_1(T_1)]\mE_x[f_2(T_1)].
\end{align}
Indeed, from the strong Markov property and Lemma \ref{lem_T_n_1},
\begin{align*}
    \mE_x\big[ f_1(T_{n_1})f_2(T_{n_2})\big] &= \mE_x\Big[ f_1\big(T_{n_1} \big) \mE\Big[ f_2\big( T_{n_2-n_1}\circ\theta_{R_{2n_1}} \big) \Big| \mathcal{F}_{R_{2n_1}}\Big]\Big] \\
    &= \mE_x\Big[ f_1\big(T_{n_1} \big) \mE_y\big[ f_2\big( T_{n_2-n_1}\big) \big] \big|_{y=X_{R_{2n_1}}} \Big] \\
    &= \mE_x\Big[ f_1\big(T_{n_1} \big) \mE_y\big[ f_2\big( T_{1}\big) \big] \big|_{y=X_{R_{2n_1}}} \Big] 
     = \mE_x\Big[ f_1\big(T_{n_1} \big) \mE_y\Big[ f_2\Big( \frac{T}{y^\alpha}\Big) \Big] \Big|_{y=X_{R_{2n_1}}} \Big] \\
     &= \mE_x\Big[ f_1\big(T_{n_1} \big) \mE_1\big[ f_2\big(T\big) \big]  \Big] 
     = \mE_x\big[ f_1\big(T_{n_1} \big)\big] \mE_1\big[ f_2\big(T\big) \big] \\
     &= \mE_x\big[ f_1\big(T_{1} \big)\big] \mE_x\Big[ f_2\Big(\frac{R_2}{X_0^\alpha}\Big) \Big] 
     = \mE_x\big[ f_1(T_1)\big] \mE_x\big[ f_2(T_1)\big].
\end{align*}
Assume now that the equality \eqref{lem_Ti_eq_n} holds for some $k\geq 2$ and all $x>0$. We will show it for $k+1$. We have,
\begin{align*}
    &\mE_x\big[ f_1(T_{n_1})f_2(T_{n_2})\ldots f_{k+1}(T_{n_{k+1}})\big] \\
    &= \mE_x\big[ f_1(T_{n_1}) \mE\big[ f_2(T_{n_2-n_1}\circ\theta_{R_{2n_1}}) \ldots f_{k+1}(T_{n_{k+1}-n_1}\circ\theta_{R_{2n_1}}) \big| \mathcal{F}_{R_{2n_1}}\big] \big] \\
    &= \mE_x\big[ f_1(T_{n_1}) \mE_y\big[ f_2(T_{n_2-n_1})\ldots f_{k+1}(T_{n_{k+1}-n_1})\big] \big|_{y=X_{R_{2n_1}}}\big] \\
    &=\mE_x\big[ f_1(T_{n_1}) \mE_y[f_2(T_1)]\ldots \mE_y[f_{k+1}(T_1)]
    \big|_{y=X_{R_{2n_1}}}\big] \\ 
    &= \mE_x\big[ f_1(T_{n_1}) \mE_y\Big[f_2\Big(\frac{T}{y^\alpha}\Big)\Big]\ldots \mE_y\Big[f_{k+1}\Big(\frac{T}{y^\alpha}\Big)\Big]
    \big|_{y=X_{R_{2n_1}}}\big] \\
    &= \mE_x\big[ f_1(T_{n_1}) \mE_1\big[f_2(T)\big]\ldots \mE_1\big[f_{k+1}(T)\big]\big] \\
    &= \mE_x\big[ f_1(T_{n_1}) \big]\mE_1\big[f_2(T)\big]\ldots \mE_1\big[f_{k+1}(T)\big] \\ 
    &= \mE_x\big[ f_1(T_{1}) \big]\mE_x\big[f_2(T_1)\big]\ldots \mE_x\big[f_{k+1}(T_1)\big],
\end{align*}
which is the desired conclusion.
\end{proof}

From Lemma \ref{lem_T_n_1} and Lemma \ref{lem_T_n_2} it follows that the family $\{T_n\}_{n\in\mathbb{N}}$ is i.i.d.

Now we consider the increments of times $S_n$, $n=0,1,2,\ldots$, of consecutive returns of the process $X$ to $D$ starting from $x>0$. Thus, $S_0:=0$ and for $n\geq 1$, $S_n := R_{2n} - R_{2n-2}$. From \eqref{def_T_n} it follows that for $n\geq 1$, $S_n = V_{n-1}^\alpha T_n$. In other words, there exist i.i.d. random variables $T_n$ with the density function $S(1,t)$ and i.i.d. random variables $W_i$ with the density function $R(1,w)$ such that for $n\geq 1$, $S_n = X_0^\alpha \big( \prod_{0<k<n} W_k^\alpha\big) T_n$.

\begin{lemma}\label{Z_T_independence}
For $n\geq 1$, $x>0$,  and bounded or non-negative functions $f,g$,
\[
\mE_x\big[ f(V_{n-1})g(T_n)\big] = \mE_x\big[ f(V_{n-1})] \mE_x[g(T_n)\big].
\]
\end{lemma}

\begin{proof}
Note that $V_{n-1} = X_{R_{2n-2}}$ and $T_n = T_1\circ\theta_{R_{2n-2}}$. Then from the strong Markov property and Lemma \ref{lem_T_n_1} we have
\begin{align*}
    \mE_x\big[f(V_{n-1})g(T_n)\big] &= \mE_x\big[ f(X_{R_{2n-2}}) \mE\big[ g(T_1\circ \theta_{R_{2n-2}}) \big| \mathcal{F}_{R_{2n-2}}\big]\big] \\
    &= \mE_x\big[ f(X_{R_{2n-2}}) \mE_y\big[ g(T_1) \big] \big|_{y = X_{R_{2n-2}}}\big] \\
    &= \mE_x\Big[ f(X_{R_{2n-2}}) \mE_y\Big[ g\Big(\frac{R_2}{y^\alpha}\Big) \Big] \Big|_{y = X_{R_{2n-2}}}\Big] \\ 
    &= \mE_x\big[ f(V_{n-1})\big] \mE_1\big[ g(R_2) \big] 
    =  \mE_x\big[ f(V_{n-1})\big] \mE_x\Big[ g\Big(\frac{R_2}{x^\alpha}\Big)\Big]\\
    &= \mE_x\big[ f(V_{n-1})\big] \mE_x\big[g(T_1)\big] 
    = \mE_x\big[ f(V_{n-1})\big] \mE_x\big[g(T_n)\big],
\end{align*}
which completes the proof.
\end{proof}

Recall that $R_\infty = \lim\limits_{n\to\infty} R_{2n} = \sum_{n=1}^\infty S_n  = X_0^\alpha \sum_{n=1}^\infty \big(\prod_{0<k<n} W_k^\alpha\big) T_n$ is the lifetime of the process $X$. 

\begin{proposition}\label{Th:hitting2}
The following statements hold $\mP_x$-a.s. for every $x\neq 0$. 
\begin{enumerate}
    \item If $\alpha\in (0,1]$, then $R_\infty = \infty$.    
    \item If $\alpha\in (1,2)$, then $R_\infty <\infty$.
\end{enumerate}
\end{proposition}

\begin{proof}
Let $x > 0$. Assume that $\alpha\in (0,1)$. We have $\ln{S_n} = \alpha \ln{V_{n-1}} + \ln{T_n}$. By Proposition \ref{Th:hitting1}, $V_{n-1}\to \infty$ a.e. as $n\to\infty$. Moreover, for any $n = 1,2,\ldots$, $\mP_x(T_n>1) = c > 0$. Hence, $\sum_{n=1}^\infty \mP_x(T_n > 1) = \infty$. Therefore, by Lemma \ref{lem_T_n_2} and the Borel--Cantelli lemma we have $\mP_x\big( \limsup\limits_{n\to\infty} \{T_n>1\}\big) = 1$, which means that with probability one there exists a subsequence $(n_k)_{k\in\mathbb{N}}$ such that $T_{n_k}>1$ and then $\ln{T_{n_k}}>0$. Therefore, $\ln{S_{n_k}}\to \infty$ a.e., as $k\to\infty$. As a result, $R_\infty = \sum_{n=1}^\infty S_n = \infty$ $\mP_x$-a.s.

Assume that $\alpha\in (1,2)$. From the subadditivity of the function $r^{\frac{\alpha-1}{2\alpha}}$ for $r>0$, the Tonelli's theorem, Lemma \ref{Z_T_independence} and Lemma \ref{lem_T_n_1},
\begin{align*}
\mE_x\Big[ \sum_{n=1}^\infty S_n\Big]^{\frac{\alpha-1}{2\alpha}}    &\leq \mE_x \Big[\sum_{n=1}^\infty S_n^{\frac{\alpha-1}{2\alpha}} \Big] = \mE_x \Big[ \sum_{n=1}^\infty V_{n-1}^{\frac{\alpha-1}{2}} T_n^{\frac{\alpha-1}{2\alpha}} \Big] =
\sum_{n=1}^\infty \mE_x \Big[  V_{n-1}^{\frac{\alpha-1}{2}} T_n^{\frac{\alpha-1}{2\alpha}} \Big] \\
&= \sum_{n=1}^\infty \mE_x\Big[V_{n-1}^{\frac{\alpha-1}{2}}\Big] \mE_x\Big[T_n^{\frac{\alpha-1}{2\alpha}}\Big] = \mE_x\Big[T_1^{\frac{\alpha-1}{2\alpha}}\Big] \sum_{n=1}^\infty \mE_x\Big[V_{n-1}^{\frac{\alpha-1}{2}}\Big]\\
&=\mE_1\Big[T^{\frac{\alpha-1}{2\alpha}}\Big] \sum_{n=1}^\infty \mE_x\Big[V_{n-1}^{\frac{\alpha-1}{2}}\Big].
\end{align*}
Moreover, from Lemma \ref{lem_Wi_iid} and Lemma \ref{lem:L3},
\begin{align*}
\sum_{n=1}^\infty \mE_x\Big[V_{n-1}^{\frac{\alpha-1}{2}}\Big] &=\sum_{n=1}^\infty \mE_x\Big[X_0 \prod_{i=1}^{n-1} W_i\Big]^{\frac{\alpha-1}{2}} = x^{\frac{\alpha-1}{2}} \sum_{n=1}^\infty \prod_{i=1}^{n-1} \mE_1\big[  W_1^{\frac{\alpha-1}{2}}\big] \\
&= x^{\frac{\alpha-1}{2}} \sum_{n=1}^\infty \prod_{i=1}^{n-1} \mE_1\big[  W^{\frac{\alpha-1}{2}}\big] =  x^{\frac{\alpha-1}{2}} \sum_{n=1}^\infty \rho^{n-1} = x^{\frac{\alpha-1}{2}} \frac{1}{1-\rho}<\infty.
\end{align*}
It suffices to show that $\mE_1\big[ T^{\frac{\alpha-1}{2\alpha}}\big] <\infty$. 
For $T=R_2$, we have $T=R_1+R_1\circ\theta_{R_1}$.
Using this observation, we get 
\begin{align*}
    \mE_1 T^{\frac{\alpha-1}{2\alpha}} &= \mE_1[R_1 + R_1\circ \theta_{R_1}]^{\frac{\alpha-1}{2\alpha}} 
    \leq \mE_1\big[R_1\big]^{\frac{\alpha-1}{2\alpha}} + \mE_1\big[R_1\circ \theta_{R_1}\big]^{\frac{\alpha-1}{2\alpha}} \\
    &= \mE_1^Y\tau_D^{\frac{\alpha-1}{2\alpha}} + \mE_1\Big[\mE_{y}\big(R_1\big)^{\frac{\alpha-1}{2\alpha}} \Big|_{y=X_{R_1}}\Big],
\end{align*}
where $\tau_D$ is the first exit time from $D$ of the process $Y$. Recall that $R_1<\infty$ a.s.

By Ba{\~n}uelos and Bogdan \cite[Exercise 3.2 and Theorem 4.1]{MR2075671}, $\mE_1^Y\tau_D^{\frac{\alpha-1}{2\alpha}}<\infty$, because $\frac{\alpha-1}{2\alpha}<\frac{1}{2}$, for $\alpha\in (1,2)$. Alternatively, the estimates of $\mE_1^Y\tau_D^{\frac{\alpha-1}{2\alpha}}$ can be obtained from the estimates of survival probability in Bogdan et al. \cite{MR2722789} and the equality
\[
\mE_x^Y \tau_D^p = p \int_0^\infty t^{p-1}\mP_x^Y(\tau_D>t)\,\dt, \quad p>0, ~x>0.
\]

Furthermore, since $R_1$ for the starting point $y<0$ has the exponential distribution with mean $1/\nu(y,D)$, then from \eqref{nu_scaling} we have
\begin{align*}
    \mE_{y}\big(R_1\big)^{\frac{\alpha-1}{2\alpha}} &= \int_0^\infty t^{\frac{\alpha-1}{2\alpha}} \nu(y,D) e^{-\nu(y,D)t} \,\dt = \big[ \nu(y,D)\big]^{\frac{1-\alpha}{2\alpha}} \int_0^\infty s^{\frac{\alpha-1}{2\alpha}} e^{-s}\,\ds \\
    &\approx |y|^{\frac{\alpha-1}{2}} \int_0^\infty s^{\frac{\alpha-1}{2\alpha}} e^{-s}\,\ds = \Gamma\big(\tfrac{3\alpha-1}{2\alpha}\big) |y|^{\frac{\alpha-1}{2}} \approx |y|^{\frac{\alpha-1}{2}}.
\end{align*}
Hence,
\begin{align*}
    \mE_1\Big[\mE_{y}\big(R_1\big)^{\frac{\alpha-1}{2\alpha}} \Big|_{y=X_{R_1}}\Big] &\approx \mE_1 \big[ |X_{R_1}|^{\frac{\alpha-1}{2}}\big]  = \mE_1^Y \big[ |Y_{\tau_D}|^{\frac{\alpha-1}{2}}\big],
\end{align*}
and from \eqref{eq:poisson_kernel_halfline},
\begin{align*}
     \mE_1^Y |Y_{\tau_D}|^{\frac{\alpha-1}{2}} &= \int_{-\infty}^0 |y|^{\frac{\alpha-1}{2}}P_D(1,y)\,\dy \approx \int_{-\infty}^0 |y|^{-1/2}|1-y|^{-1}\,\dy  = \int_0^\infty \frac{\dy}{\sqrt{y}(1+y)} <\infty.
\end{align*}

Let $\alpha=1$. From Proposition \ref{Th:hitting1}, $\limsup\limits_{n\to\infty} V_n = +\infty$ $\mP_x$-a.s. Thus, there exists subsequence $(V_{n_k})_{k\in\mathbb{N}}$ such that $V_{n_k}\to+\infty$ $\mP_x$-a.s., as $k\to\infty$. By the same argument as in the case of $\alpha\in (0,1)$, $\ln{S_{n_k}} = \ln{V_{n_k-1}} + \ln{T_{n_k}} \to +\infty$ in probability. Therefore, there exists a subsequence $(\ln S_{n_{k_l}})_{l\in\mathbb{N}}$ such that $\ln S_{n_{k_l}}\to +\infty$ $\mP_x$-a.s., as $l\to\infty$, which implies that $R_\infty = \infty$.

\medskip
Now let $x<0$. Then,
\begin{align*}
    \mP_x\big(R_\infty < \infty\big) &= \mE_x\big[ \mE \big( \ind_{(0,\infty)}(R_\infty) ~\big|~ \mathcal{F}_{R_1} \big)\big] = \mE_x\big[ \mE \big( \ind_{(0,\infty)}(R_\infty\circ\theta_{R_1} + R_1) ~\big|~ \mathcal{F}_{R_1} \big)\big]\\
    &= \mE_x\left[ \mE_{X_{R_1}} \big( \ind_{(0,\infty)}(R_\infty + s)\big) \Big|_{s=R_1}\right] = \mE_x\Big[ \mP_{X_{R_1}}\big( R_\infty + s < \infty\big)\Big|_{s=R_1} \Big] \\
    &= \mE_x\big[ \mP_{X_{R_1}}\big( R_\infty < \infty\big) \big],
\end{align*}
since $R_1<\infty$ a.s.
Note that for $\alpha\in (0,1]$, from the first part of the proof, it follows that
\begin{align*}
    \mP_x\big(R_\infty < \infty\big) = \mE_x\big[ \mP_{X_{R_1}}\big( R_\infty < \infty\big) \big] = \mE_x0 = 0,
\end{align*}
and for $\alpha\in (1,2)$,
\begin{align*}
    \mP_x\big(R_\infty < \infty\big) = \mE_x\big[ \mP_{X_{R_1}}\big( R_\infty < \infty\big) \big] = \mE_x1 = 1,
\end{align*}
which completes the proof.
\end{proof}

\subsection{Proof of Theorem~\ref{lifetime_of_the_process}}\label{sec:main_theorem}

Let $\alpha=1$. The case $x>0$ follows directly from Proposition \ref{Th:hitting1} and Proposition \ref{Th:hitting2}. In case $x<0$, we have
\[
\liminf_{n\to\infty} X_{R_{2n+1}} = \big( \liminf_{n\to\infty} X_{R_{2n}}\big) \circ \theta_{R_1}
\]
and then
\begin{align*}
    \mP_x\big( \liminf_{n\to\infty} X_{R_{2n+1}} = 0\big) &= \mE_x\big( \mE\big( \ind_{\{0\}}\big(\liminf_{n\to\infty} X_{R_{2n+1}}\big)~\big|~ \mathcal{F}_{R_1}\big)\big) \\
    &= \mE_x\big( \mE\big( \ind_{\{0\}}\big(\liminf_{n\to\infty} X_{R_{2n}}\big) \circ \theta_{R_1}\big)~\big|~ \mathcal{F}_{R_1}\big) \\
    &= \mE_x\big( \mE_{X_{R_1}}\big( \ind_{\{0\}}\big(\liminf_{n\to\infty} X_{R_{2n}}\big)\big)\big) \\
    &= \mE_x\big( \mP_{X_{R_1}}\big(\liminf_{n\to\infty} X_{R_{2n}} = 0\big)\big) = 1.
\end{align*}
Similarly, we prove that $\mP_x\big( \limsup\limits_{n\to\infty} X_{R_{2n+1}} = \infty\big) = 1$.

\medskip
Now assume that $\alpha\neq 1$. Note that $Y_t:= h_{\alpha-1}(X_t) \geq 0$ is a supermartingale with right-continuous trajectories. Indeed, for $s<t$, from Markov property and Proposition \ref{theorem_excessive_function} we get
\begin{align*}
    \mE_x\big[ Y_t ~\big|~\mathcal{F}_s\big] &= \mE_x\big[ h_{\alpha-1}(X_t) ~\big|~\mathcal{F}_s\big] = \mE_x\big[ h_{\alpha-1}(X_{t-s}\circ \theta_s) ~\big|~\mathcal{F}_s\big] = \mE_x\big[ h_{\alpha-1}(X_{t-s})\circ \theta_s ~\big|~\mathcal{F}_s\big] \\
    &= \mE_{X_s}\big[ h_{\alpha-1}(X_{t-s})\big] = K_{t-s}h_{\alpha-1}(X_s) \leq h_{\alpha-1}(X_s) = Y_s.
\end{align*}
Hence, it is clear that $Z_t := -h_{\alpha-1}(X_t)$ is a submartingale.

Following \cite[Chapter II.2]{GVK25295551X}, we consider a function $f:\mathbb{T}\to\R$, where $\mathbb{T}\subseteq [0,\infty)\cap \mathbb{Q}$ is a countable set, and define the \emph{number of downcrossings} of the interval $[a,b]$ by the function $f$ as follows. Let $F := \{t_1, t_2,\ldots, t_m\} \subset\mathbb{T}$. For $a,b\in\R$, $a<b$, we define inductively,
\begin{align*}
    s_1 &= \inf\{ t_i:~ t_i\in F, ~f(t_i)>b\}, \\
    s_2 &= \inf\{t_i>s_1:~t_i\in F, ~f(t_i)<a\}, \\
    \vdots&\\
    s_{2n+1} &=\inf\{t_i> s_{2n}:~t_i\in F, ~f(t_i)>b\}, \\
    s_{2n+2} &= \inf\{t_i>s_{2n+1}:~t_i\in F, ~f(t_i)<a\},
\end{align*}
while $\inf(\emptyset) := t_m$. 
We set
\begin{align*}
    D_F(f,[a,b]) := \sup\{n:~s_{2n}<t_m\}.
\end{align*}
The \emph{number of downcrossings} of the interval $[a,b]$ by the function $f:\mathbb{T}\to\R$ we define as the number
\begin{align*}
    D_\mathbb{T}(f,[a,b]) := \sup\{D_F(f,[a,b]):~F\subseteq \mathbb{T}, ~F ~\mathrm{finite}\}.
\end{align*}

From the Doob's downcrossing lemma \cite[Proposition 2.1, p. 61]{GVK25295551X},
\begin{align*}
    \mE_x\big[ D_\mathbb{T}(Z, [a,b])\big] \leq \sup_{t\in\mathbb{T}} \frac{\mE_x\big[ (Z_t-b)^+ \big]}{b-a} = \sup_{t\in\mathbb{T}} \frac{\mE_x\big[ (h_{\alpha-1}(X_t)+b)^- \big]}{b-a} \leq \frac{|b|}{b-a} <\infty.
\end{align*}
Hence, $\mP_x\big( D_\mathbb{T}(Z, [a,b]) = \infty) = 0$. Let
\begin{align*}
    \mathcal{A}_\mathbb{T} := \bigcap_{a,b\in\mathbb{Q},~ a<b} \big\{\omega:~ D_{\mathbb{T}}(Z(\omega), [a,b]) < \infty\big\}.
\end{align*}
Then $\mP_x(\mathcal{A}_\mathbb{T}) = 1$, because $\mathcal{A}_\mathbb{T}$ is an intersection of countable many sets of the measure one. 

\medskip
Let $\xi$ denote the lifetime of the process $X$, i.e. $\xi = R_\infty$. Moreover, let $\mathbb{T} = [0,\infty)\cap\mathbb{Q}$. 
If $\omega\in\mathcal{A}_\mathbb{T}$, then $D_\mathbb{T}(Z(\omega), [a,b]) <\infty$ for any rational numbers $a<b$.
We claim that the limit $Z(\omega) := \lim\limits_{t\to \xi} Z_t(\omega)$ exists (but it may be infinite).
Indeed, assume that this limit does not exist. Then we can find rational numbers $\alpha<\beta$ such that
\begin{align*}
    \liminf_{t\to\xi} Z_t(\omega) <\alpha < \beta < \limsup_{t\to\xi} Z_t(\omega).
\end{align*}
From this fact we conclude that there exists an increasing sequence $(r_n)_{n\in\mathbb{N}} \subset \mathbb{T}$ such that $Z_{r_{2n}}(\omega) \leq \alpha <\beta \leq Z_{r_{2n-1}}(\omega)$. By taking $\mathbb{I} = \{r_1, r_2, r_3,\ldots\}$, we obtain that
\[
D_{\mathbb{T}}(Z(\omega), [\alpha,\beta]) \geq D_{\mathbb{I}}(Z(\omega), [\alpha,\beta]) = \infty,
\]
which is a contradiction. Hence, the limit $Z := -\lim\limits_{t\nearrow \xi} h_{\alpha-1}(X_t)$ exists with probability one. Therefore, $\lim\limits_{t\nearrow \xi} |X_t|$ exists with probability one.

\medskip
Assume that $\alpha\in (0,1)$. Then from Proposition \ref{Th:hitting2}, the lifetime of the process $X$ is infinite a.e. Moreover, from Proposition \ref{Th:hitting1} for $x>0$ we have $\lim\limits_{n\to\infty} X_{R_{2n}} = \infty$ a.e., hence $\lim\limits_{n\to\infty} h_{\alpha-1}(X_{R_{2n}}) = 0$ a.e. From the uniqueness of the limit we get that $\lim\limits_{t\to\infty} h_{\alpha-1}(X_t) = 0$ a.e., hence $\lim\limits_{t\to\infty} |X_t| = \infty$ a.e. For $x<0$, by the analogous computations as in the case $\alpha=1$, we get $\lim\limits_{n\to\infty} X_{R_{2n+1}} = \infty$ a.e., hence in the same way as above we get $\lim\limits_{t\to\infty} |X_t| = \infty$ a.e.

\medskip
Now assume that $\alpha\in (1,2)$. Then from Proposition \ref{Th:hitting2} for $x>0$, the lifetime of the process $X$ is finite a.e. Moreover, from Proposition \ref{Th:hitting1}, $\lim\limits_{n\to\infty} X_{R_{2n}} = 0$ a.e., hence $\lim\limits_{n\to\infty} h_{\alpha-1}(X_{R_{2n}}) = 0$ a.e. From the uniqueness of the limit we get that $\lim\limits_{t\nearrow \xi} h_{\alpha-1}(X_t) = 0$ a.e., hence $\lim\limits_{t\nearrow \xi} X_t = 0$ a.e. For $x<0$, we argue in the same way.

The proof is complete. In passing we note that the process $X$ crosses zero an infinite number of times before the lifetime because $0<R_n<\infty$ for every $n$ by Remark \ref{rem:P_n_prob1}. \hfill \qed

\subsection{Feller property} 

Here we prove that for $\alpha\in (1,2)$, $(K_t)_{t\geq 0}$ forms a Feller semigroup on the space $C_0(\R_*)$. For this purpose, we will use the fact that the process $X$ has finite lifetime for such $\alpha$.

\begin{corollary}\label{cor_zanik_masy}
For $\alpha\in (1,2)$, $x\neq 0$, and $t>0$,
\[
\lim_{x\to 0} K_t(x,\R) = \lim_{x\to 0}  \mP_x(R_\infty>t) = 0.
\]
\end{corollary}

\begin{proof}
From Lemma \ref{K_t_scaling} it follows that $K_{t}(x,\R) = K_{tx^{-\alpha}}(1,\R)$ for all $x>0$. By Proposition \ref{Th:hitting2},
\[
K_t(x,\R) = \mP_1(R_\infty > tx^{-\alpha}) \to 0, 
\]
as $x\to 0^+$.

Similarly, for all $x<0$, from Lemma \ref{K_t_scaling} it follows that $K_{t}(x,\R) = K_{t|x|^{-\alpha}}(-1,\R)$, and again by Proposition \ref{Th:hitting2},
\[
K_t(x,\R) = \mP_{-1}(R_\infty > t|x|^{-\alpha}) \to 0, 
\]
as $x\to 0^-$.
\end{proof}

\begin{lemma}\label{K_t_pointwise_convergence}
    For $f\in C_0(\R_*)$ and $x\neq 0$, $K_tf(x) \to f(x)$, as $t\to 0^+$.
\end{lemma}

\begin{proof}
    From Corollary \ref{perturbation_formula} and Lemma \ref{K_subprobability} we have
\begin{align*}
    |K_tf(x) - f(x)| &= \left| \P_tf(x) + \int_0^t \P_s\widehat{\nu} K_{t-s}f(x)\,\ds - f(x)\right| \\
    &\leq \big|\P_tf(x) - f(x)\big| + \norm{f}_\infty \int_0^t \P_s\widehat{\nu} \mathbf{1}(x)\,\ds. 
\end{align*}
Moreover, from Ikeda--Watanabe formula \eqref{eq:Ikeda_Watanabe}, for $x>0$, we have
\begin{align*}
    \int_0^t \P_s\widehat{\nu} \mathbf{1}(x)\,\ds &= \int_0^t \ds \int_D \dy \int_{D^c} \dz \,  p_s^D(x,y)\nu(y,z) = \mP_x^Y(\tau_D < t)\to 0,
\end{align*}
as $t\to 0^+$.

Similarly, for $x<0$, 
\begin{align*}
    \int_0^t \P_s\widehat{\nu} \mathbf{1}(x)\,\ds &= \int_0^t \nu(x, D)e^{-\nu(x,D)s} \,\ds = 1 - e^{-\nu(x,D)t}\to 0,
\end{align*}
as $t\to 0^+$. Hence, from Lemma \ref{P_t_feller} we obtain a desired convergence.
\end{proof}

\begin{proposition}\label{K_t_feller}
For $\alpha\in (1,2)$, $(K_t)_{t\geq 0}$ is a Feller semigroup on $C_0(\R_*)$.
\end{proposition}

\begin{proof}
From Lemma \ref{K_subprobability} it follows that for any $f\in C_0(\R_*)$, $0\leq f\leq 1$ we have $0\leq K_tf\leq 1$, $t\geq 0$. Hence $(K_t)_{t\geq 0}$, is a semigroup of nonnegative contraction operators on $C_0(\R_*)$. 

We claim that $K_t C_0(\R_*)\subset C_0(\R_*)$. Let $t>0$ and $f\in C_0(\R_*)$. Then from Proposition \ref{K_bounded_continuity}, $K_tf\in C(\R_*)$. Moreover, from Corollary \ref{cor_zanik_masy},
\[
\lim_{x\to 0} \big|K_tf(x)\big| \leq \norm{f}_\infty \lim_{x\to 0} K_t\mathbf{1}(x) = \norm{f}_\infty \lim_{x\to 0} K_t(x,\R) = 0.
\]
From Corollary \ref{perturbation_formula} and Lemma \ref{K_subprobability}, for $x>0$,
\begin{align*}
    \big| K_tf(x)\big| &\leq P_t^D|f|(x) + \norm{f}_\infty \int_0^t \ds \int_D \dy \int_{D^c} \dz \, p_s^D(x,y)\nu(y,z).
\end{align*}
From \eqref{eq:Ikeda_Watanabe} we get 
\begin{equation}\label{K_t_feller_eq_1}
    \big| K_tf(x)\big| \leq P_t^D|f|(x) + \norm{f}_\infty \mP_x^Y(\tau_D<t).
\end{equation}
It follows from \eqref{eq:scaling2} that 
$\mP_x^Y(\tau_D<t) = 1 - \int_D p_t^D(x,y)\, \dy = 1 - \int_D p_{x^{-\alpha}t}^D(1,y)\, \dy = \mP_1^Y(\tau_D<x^{-\alpha }t) \to 0,$ as $x\to\infty.$
From \eqref{K_t_feller_eq_1} and Lemma \ref{theorem_p^D_feller} it follows that $\lim\limits_{x\to +\infty} K_tf(x) = 0$. Similarly, from Corollary \ref{perturbation_formula} and Lemma \ref{K_subprobability}, for $x<0$,
\begin{align}\label{K_t_feller_eq_2}
    \big|K_tf(x)\big| &\leq |f(x)|e^{-\nu(x,D)t} + \norm{f}_\infty \int_0^t \nu(x,D)e^{-\nu(x,D)s}\,\ds \nonumber \\
    &= |f(x)|e^{-\nu(x,D)t} + \norm{f}_\infty \big[1 - e^{-\nu(x,D)t}\big].
\end{align}
From \eqref{K_t_feller_eq_2} and \eqref{nu_scaling} it follows that $\lim\limits_{x\to -\infty} K_tf(x) = 0$. Therefore, we obtained that $K_tf\in C_0(\R_*)$. 

Hence, with Lemma \ref{K_t_pointwise_convergence}, we finish the proof.
\end{proof}

\section{Short-time behavior} \label{chap_generator} 

In this section, we are  interested in infinitesimal behavior of the semigroup at $t=0$. 
Recall that the classical \emph{infinitesimal generator} (see, e.g.,  \cite{MRockner}) for a Feller semigroup $(T_t)_{t\geq 0}$ on $C_0(\mathscr{X})$ is 
\begin{align}\label{eq:generator_def}
\mathcal{L}f := \lim_{t\to 0^+} \frac{T_tf-f}{t} \quad\mathrm{ in } ~C_0(\mathscr{X}).
\end{align}
The domain $\mathcal{D}(\mathcal{L})$ of $\mathcal{L}$ consists of all the functions $f\in C_0(\mathscr{X})$ for which the limit \eqref{eq:generator_def} exists. Dealing with  generators is somehow difficult because for nonlocal operators, test functions as a rule are not in their domain; see Bäumer, Luks, Meerschaert \cite[Theorem 2.3]{MR3897925}. 

The main goal of this section is to derive the \textit{pointwise formula} 
\begin{align*}
    \lim_{t\to 0^+} \frac{K_th_\beta(x) - h_\beta(x)}{t}, \quad x\neq 0,
\end{align*}
for the functions $h_\beta$ defined in Section \ref{Chap_Servadei_process}. 
The precise statement is given in Proposition \ref{gen_inequality} after a number of auxiliary results.

\subsection{Estimation of integrals}

\begin{lemma}\label{lem_oszac_1_1}
    For $\alpha\in (0,2)$, $x>0$, $y<0$,
    \begin{align*}
        \int_0^{1/2} s^{-1} \Big(1\wedge \frac{|x|}{s^{1/\alpha}}\Big)^{\alpha/2} &\Big(1\wedge \frac{|y|}{s^{1/\alpha}}\Big)^{-\alpha/2} p_s(x,y)\,\ds \\[3pt]
        &\approx \big(1\wedge |x|\big)^{\alpha/2} \big(1\wedge |y|\big)^{-\alpha/2} \big(|x-y|^{-1}\wedge |x-y|^{-\alpha-1}\big).
    \end{align*}
\end{lemma}

\begin{proof}
Let $A := \int_0^{1/2} s^{-1} \big(1\wedge \frac{|x|}{s^{1/\alpha}}\big)^{\alpha/2} \big(1\wedge \frac{|y|}{s^{1/\alpha}}\big)^{-\alpha/2} p_s(x,y)\,\ds$ and consider three cases.

\medskip
\textsc{Case 1.} Let $x^\alpha \geq 1/2$. Then of course $|x-y|^\alpha \geq 1/2$ and then from \eqref{eq:p_t_approx},
\begin{align*}
    A &\approx \int_0^{1/2} s^{-1} \Big(1\wedge \frac{|y|}{s^{1/\alpha}}\Big)^{-\alpha/2} \frac{s}{|x-y|^{\alpha+1}}\,\ds \\
    &= |x-y|^{-\alpha-1} \Big[ \int_0^{|y|^\alpha \wedge 1/2} \ds  + |y|^{-\alpha/2} \int_{|y|^\alpha\wedge 1/2}^{1/2}  s^{1/2} \,\ds \Big] \\
    &= |x-y|^{-\alpha-1} \Big[ \big(|y|^\alpha\wedge 1/2\big) + \tfrac23 |y|^{-\alpha/2}  \Big( \Big(\frac{1}{2}\Big)^{3/2} - \big(|y|^\alpha\wedge 1/2\big)^{3/2} \Big)\Big].
\end{align*}
For $|y|^\alpha > 1/4$, 
\begin{align*}
    A &\approx |x-y|^{-\alpha-1} 
    \approx |x-y|^{-\alpha-1} \big(1 \vee |y|^{-\alpha/2}\big) \\
    &\approx \big(1\wedge |x|\big)^{\alpha/2} \big(1\wedge |y|\big)^{-\alpha/2} \big(|x-y|^{-1} \wedge |x-y|^{-\alpha-1}\big),
\end{align*}
and for $|y|^\alpha \leq 1/4$,
\begin{align*}
    A &\approx |x-y|^{-\alpha-1} \big[ |y|^\alpha + |y|^{-\alpha/2} \big]  
    = |x-y|^{-\alpha-1} |y|^{-\alpha/2} \big[ 1 + |y|^{3\alpha/2}\big] \\
    &\approx |x-y|^{-\alpha-1} |y|^{-\alpha/2} 
    \approx |x-y|^{-\alpha-1} \big(1 \vee |y|^{-\alpha/2}\big) \\
    &\approx \big(1\wedge |x|\big)^{\alpha/2} \big(1\wedge |y|\big)^{-\alpha/2} \big(|x-y|^{-1} \wedge |x-y|^{-\alpha-1}\big).
\end{align*}

\medskip
\textsc{Case 2.} Let $x^\alpha<1/2$ and $|x-y|^\alpha>1$.  Then $|y|^\alpha = \big( |x-y| - |x| \big)^\alpha > \big(1 - 2^{-1/\alpha}\big)^\alpha$. Moreover,
\begin{align*}
    A &\approx \int_0^{x^\alpha} |x-y|^{-\alpha-1} \,\ds + \int_{x^\alpha}^{1/2}  \frac{x^{\alpha/2}}{\sqrt{s}} |x-y|^{-\alpha-1}\,\ds = |x-y|^{-\alpha-1} \Big[ x^\alpha + 2x^{\alpha/2} \big( 2^{-1/2} - x^{\alpha/2} \big)
    \Big] \\
    &= |x-y|^{-\alpha-1} x^{\alpha/2} \big[ \sqrt{2} - x^{\alpha/2}\big].
\end{align*}
Note that $2^{-1/2} = \sqrt{2} - 2^{-1/2} \leq \sqrt{2} - x^{\alpha/2} \leq \sqrt{2}$. Therefore,
\begin{align*}
    A \approx |x-y|^{-\alpha-1} x^{\alpha/2} \approx \big(1\wedge |x|\big)^{\alpha/2} \big(1\wedge |y|\big)^{-\alpha/2} \big(|x-y|^{-1} \wedge |x-y|^{-\alpha-1}\big).
\end{align*}

\medskip
\textsc{Case 3.} Let $x^\alpha <1/2$ and $|x-y|^\alpha\leq 1$. Then $|y|^\alpha = \big(|x-y| - |x|\big)^\alpha \leq 1$. Assume that $|x| < |y|$. Then
\begin{align}
    A &\approx \int_0^{x^\alpha} |x-y|^{-\alpha-1}\,\ds + \int_{x^\alpha}^{|y|^\alpha\wedge 1/2} \frac{|x|^{\alpha/2}}{\sqrt{s}}  |x-y|^{-\alpha-1}\,\ds \nonumber \\
    &+ \int_{|y|^\alpha\wedge 1/2}^{|x-y|^\alpha \wedge 1/2}  |x|^{\alpha/2}|y|^{-\alpha/2} |x-y|^{-\alpha-1}\,\ds + \int_{|x-y|^\alpha \wedge 1/2}^{1/2} |x|^{\alpha/2}|y|^{-\alpha/2} s^{-1/\alpha-1}\,\ds \nonumber\\
    &= |x-y|^{-\alpha-1} |x|^\alpha + 2|x|^{\alpha/2} |x-y|^{-\alpha-1} \big( \big(|y|^{\alpha}\wedge 1/2\big)^{1/2} - |x|^{\alpha/2}\big) \nonumber\\
    &+ |x|^{\alpha/2} |y|^{-\alpha/2} |x-y|^{-\alpha-1} \big[ \big(|x-y|^\alpha\wedge 1/2\big) - \big(|y|^\alpha\wedge 1/2\big) \big] \nonumber \\
    &+ \alpha |x|^{\alpha/2} |y|^{-\alpha/2} \big[ \big(|x-y|^\alpha \wedge 1/2\big)^{-1/\alpha} - 2^{1/\alpha}\big] \label{eq:rownanie_na_A}\\
    &\leq \frac{|x|^\alpha}{|x-y|^{\alpha+1}}  + \frac{2|x|^{\alpha/2}|y|^{\alpha/2}}{|x-y|^{\alpha+1}} + \frac{|x|^{\alpha/2} |y|^{-\alpha/2}}{|x-y|} + \alpha |x|^{\alpha/2} |y|^{-\alpha/2} \big(|x-y|^\alpha \wedge 1/2\big)^{-1/\alpha}.\nonumber
\end{align}
Note that $|x|^\alpha = |x|^{\alpha/2}|x|^{\alpha/2} < |x|^{\alpha/2}|y|^{\alpha/2}$.
Moreover,
\begin{align*}
    |x-y|^\alpha \wedge 1/2 \geq |x-y|^\alpha \wedge \tfrac12 |x-y|^\alpha = \tfrac12 |x-y|^\alpha.
\end{align*}
Therefore,
\begin{align*}
    A &\lesssim \frac{|x|^{\alpha/2}|y|^{\alpha/2}}{|x-y|^{\alpha+1}} + \frac{|x|^{\alpha/2} |y|^{-\alpha/2}}{|x-y|} = \frac{|x|^{\alpha/2} |y|^{-\alpha/2}}{|x-y|} \left[ 1 + \frac{|y|^\alpha}{|x-y|^\alpha} \right] \lesssim \frac{|x|^{\alpha/2} |y|^{-\alpha/2}}{|x-y|} \\
    &= \big(1\wedge |x|\big)^{\alpha/2} \big(1\wedge |y|\big)^{-\alpha/2} \big(|x-y|^{-1} \wedge |x-y|^{-\alpha-1}\big).
\end{align*}

\medskip
For the estimate from below, we consider two cases. For $|x|<|y|< (8^{1/\alpha}-1)|x|$, from \eqref{eq:rownanie_na_A}, we have that
\begin{align*}
    A &\gtrsim |x-y|^{-\alpha-1} |x|^\alpha  = \frac{|x|^{\alpha/2} |y|^{-\alpha/2}}{|x-y|} \cdot \left[ \frac{|x| \, |y|}{(|x|+|y|)^2}\right]^{\alpha/2} \geq \frac{|x|^{\alpha/2} |y|^{-\alpha/2}}{|x-y|} \cdot \left[ \frac{|x|}{|x|+|y|}\right]^{\alpha} \\
    &\geq \frac{|x|^{\alpha/2} |y|^{-\alpha/2}}{|x-y|} \cdot \left[ \frac{|x|}{|x|+(8^{1/\alpha}-1)|x|}\right]^{\alpha} \approx \frac{|x|^{\alpha/2} |y|^{-\alpha/2}}{|x-y|} \\
    &= \big(1\wedge |x|\big)^{\alpha/2} \big(1\wedge |y|\big)^{-\alpha/2} \big(|x-y|^{-1} \wedge |x-y|^{-\alpha-1}\big).
\end{align*}
For $|y| \geq (8^{1/\alpha} - 1)|x|$, from \eqref{eq:rownanie_na_A}, we have that
\begin{align*}
    A &\gtrsim 2|x|^{\alpha/2} |x-y|^{-\alpha-1} \big( \big(|y|^{\alpha}\wedge 1/2\big)^{1/2} - |x|^{\alpha/2}\big) \\
    &\approx \frac{|x|^{\alpha/2} |y|^{-\alpha/2}}{|x-y|} \left( \frac{|y|}{|x|+|y|}\right)^{\alpha/2} \left[ \left( \frac{|y|^\alpha \wedge 1/2}{|x-y|^{\alpha}}\right)^{1/2} - \left( \frac{|x|}{|x-y|}\right)^{\alpha/2}\right] \\
    &\gtrsim \frac{|x|^{\alpha/2} |y|^{-\alpha/2}}{|x-y|} \left( \frac{1}{(8^{1/\alpha}-1)^{-1}+1 }\right)^{\alpha/2} \left[ \left( \frac{|y| \wedge 2^{-1/\alpha}}{\big((8^{1/\alpha}-1)^{-1}+1\big)|y|}\right)^{\alpha/2} - \left( \frac{1}{8^{1/\alpha}}\right)^{\alpha/2}\right] \\
    &\approx \frac{|x|^{\alpha/2} |y|^{-\alpha/2}}{|x-y|}  \left[ \big(8^{1/\alpha}-1)^{-1}+1\big)^{-\alpha/2} \big( 1 \wedge 2^{-1/\alpha}|y|^{-1}\big)^{\alpha/2} - 8^{-1/2} \right].
\end{align*}
Recall that $|y|\leq 1$ and then
\begin{align*}
    A &\gtrsim \frac{|x|^{\alpha/2} |y|^{-\alpha/2}}{|x-y|}  \left[ \big((8^{1/\alpha}-1)^{-1}+1\big)^{-\alpha/2} \big( 1 \wedge 2^{-1/\alpha}\big)^{\alpha/2} - 8^{-1/2} \right] \\[1.5pt]
    &= \frac{|x|^{\alpha/2} |y|^{-\alpha/2}}{|x-y|} \left[ \frac{(8^{1/\alpha}-1)^{\alpha/2}}{4} - \frac{1}{2\sqrt{2}}\right] \\[1.5pt]
    &\approx \big(1\wedge |x|\big)^{\alpha/2} \big(1\wedge |y|\big)^{-\alpha/2} \big(|x-y|^{-1} \wedge |x-y|^{-\alpha-1}\big).
\end{align*}

\medskip
Now assume that $|x| \geq |y|$. Then, similarly,
\begin{align}\label{A_est2}
    A &\approx \int_0^{|y|^\alpha} |x-y|^{-\alpha-1}\,\ds + \int_{|y|^\alpha}^{x^\alpha} s^{1/2}  |y|^{-\alpha/2} |x-y|^{-\alpha-1}\,\ds \nonumber \\[1.5pt]
    &+ \int_{x^\alpha}^{|x-y|^\alpha \wedge 1/2} |x|^{\alpha/2}|y|^{-\alpha/2} |x-y|^{-\alpha-1}\,\ds + \int_{|x-y|^\alpha \wedge 1/2}^{1/2} s^{-1/\alpha-1} |x|^{\alpha/2}|y|^{-\alpha/2}\,\ds \nonumber\\[1.5pt]
    &= |x-y|^{-\alpha-1} |y|^\alpha + \tfrac23|y|^{-\alpha/2} |x-y|^{-\alpha-1} \big[ |x|^{3\alpha/2} - |y|^{3\alpha/2}\big] \nonumber\\[1.5pt]
    &+ |x|^{\alpha/2} |y|^{-\alpha/2} |x-y|^{-\alpha-1} \big[ \big(|x-y|^\alpha \wedge 1/2\big) - |x|^\alpha\big] \nonumber\\[1.5pt]
    &+ \alpha |x|^{\alpha/2} |y|^{-\alpha/2} \big[ \big(|x-y|^\alpha \wedge 1/2\big)^{-1/\alpha} - 2^{1/\alpha}\big]  \nonumber\\[1.5pt]
    &\lesssim \frac{|y|^\alpha}{|x-y|^{\alpha+1}} + \frac{|x|^{\alpha/2} |y|^{-\alpha/2}}{|x-y|} + \frac{\alpha |x|^{\alpha/2} |y|^{-\alpha/2}}{\big(|x-y|^\alpha \wedge 1/2\big)^{1/\alpha}} \nonumber\\[1.5pt]
    &\lesssim \frac{|y|^\alpha}{|x-y|^{\alpha+1}} + \frac{|x|^{\alpha/2} |y|^{-\alpha/2}}{|x-y|}.
\end{align}
Note that from the assumption, we have that
\[
 \frac{|y|^\alpha}{|x-y|^{\alpha+1}} = \frac{|y|^{-\alpha/2}}{|x-y|} \cdot \left[ \frac{|y|}{|x-y|}\right]^\alpha |y|^{\alpha/2} \leq \frac{|x|^{\alpha/2} |y|^{-\alpha/2}}{|x-y|}.
\]
Hence,
\begin{align*}
    A \lesssim \frac{|x|^{\alpha/2} |y|^{-\alpha/2}}{|x-y|} = \big(1\wedge |x|\big)^{\alpha/2} \big(1\wedge |y|\big)^{-\alpha/2} \big(|x-y|^{-1} \wedge |x-y|^{-\alpha-1}\big).
\end{align*}

\medskip
For the estimate from below, we consider two cases. For $|x|\geq |y|\geq 2^{2/3-2/\alpha}|x|$, from \eqref{A_est2},
\begin{align*}
    A &\gtrsim \frac{|y|^\alpha}{|x-y|^{\alpha+1}} = \frac{|x|^{\alpha/2} |y|^{-\alpha/2}}{|x-y|} \left[ \frac{|y|^{3/2} |x|^{-1/2}}{|x-y|} \right]^\alpha \geq \frac{|x|^{\alpha/2} |y|^{-\alpha/2}}{|x-y|} \left[ \frac{|y|^{3/2} |x|^{-1/2}}{2|x|} \right]^\alpha \\[1.5pt]
    &= \frac{|x|^{\alpha/2} |y|^{-\alpha/2}}{|x-y|} \left( \frac{|y|}{2^{2/3}|x|} \right)^{3\alpha/2} \geq \frac18 \frac{|x|^{\alpha/2} |y|^{-\alpha/2}}{|x-y|} \\[1.5pt]
    &\approx \big(1\wedge |x|\big)^{\alpha/2} \big(1\wedge |y|\big)^{-\alpha/2} \big(|x-y|^{-1} \wedge |x-y|^{-\alpha-1}\big).
\end{align*}
Similarly, for $|y|< 2^{2/3-2/\alpha}|x|$, from \eqref{A_est2},
\begin{align*}
    A &\gtrsim |y|^{-\alpha/2} |x-y|^{-\alpha-1} \big[ |x|^{3\alpha/2} - |y|^{3\alpha/2}\big] \geq  |y|^{-\alpha/2} |x-y|^{-\alpha-1} |x|^{3\alpha/2}\big[ 1 - 2^{\alpha-3} \big] \\[1.5pt]
    &\approx \frac{|x|^{\alpha/2} |y|^{-\alpha/2}}{|x-y|} \left( \frac{|x|}{|x-y|}\right)^\alpha \geq \frac{|x|^{\alpha/2} |y|^{-\alpha/2}}{|x-y|} \left(\frac{|x|}{2|x|}\right)^\alpha \\[1.5pt]
    &\approx \big(1\wedge |x|\big)^{\alpha/2} \big(1\wedge |y|\big)^{-\alpha/2} \big(|x-y|^{-1} \wedge |x-y|^{-\alpha-1}\big),
\end{align*}
which completes the proof.
\end{proof}

\begin{lemma}\label{lem_oszac_calka_z_p^D_po_D}
    For $\alpha\in (0,2)$, $t>0$, $x>0$, and $y<0$, 
    \begin{align*}
        \int_0^\infty \Big( 1 \wedge \frac{z^{\alpha/2}}{\sqrt{t}}\Big) p_t(x,z)\nu(z,y)\,\dz \approx t^{-1} \Big(1\wedge \frac{|y|}{t^{1/\alpha}}\Big)^{-\alpha/2} p_t(x,y).
    \end{align*}
\end{lemma}

\begin{proof}
Let $I_t(x,y) := \int_0^\infty \big( 1 \wedge \frac{z^{\alpha/2}}{\sqrt{t}}\big) p_t(x,z)\nu(z,y)\,\dz$. By substitution $z = t^{1/\alpha}w$, from \eqref{eq:scaling_pt}, we get
\begin{align}\label{eq:I_t_scaling1}
    I_t(t^{1/\alpha}x, t^{1/\alpha}y) = t^{-1/\alpha-1} \int_0^\infty \big( 1 \wedge w^{\alpha/2}\big) p_1(x,w)\nu(w,y)\,\dw = t^{-1/\alpha-1} I_1(x,y),
\end{align}
hence it suffices to find an estimation of $I_1(x,y)$.

\medskip
First, assume that $x < 1/2$. From \eqref{eq:p_t_approx} we then have
\begin{align*}
I_1(x,y) &\approx \int_0^\infty \big( 1 \wedge w^{\alpha/2}\big) \big(1 \wedge |x-w|^{-\alpha-1}\big)|w-y|^{-\alpha-1}\,\dw \\
&\approx \int_0^1 w^{\alpha/2}|w-y|^{-\alpha-1}\,\dw + \int_1^\infty |x-w|^{-\alpha-1}|w-y|^{-\alpha-1} \,\dw =: A(y) + B(x,y).
\end{align*}
Using the substitution $w = |y|u$ we obtain that
\begin{align}\label{eq:Axy_podst}
    A(y) = |y|^{-\alpha/2} \int_0^{1/|y|} \frac{u^{\alpha/2}}{(u+1)^{\alpha+1}}\,\du.
\end{align}
Furthermore, for $|y|\leq 1$, 
\begin{align*}
    \int_0^1 \frac{u^{\alpha/2}}{(u+1)^{\alpha+1}}\,\du \leq \int_0^{1/|y|} \frac{u^{\alpha/2}}{(u+1)^{\alpha+1}}\,\du \leq \int_0^\infty \frac{u^{\alpha/2}}{(u+1)^{\alpha+1}}\,\du <\infty,
\end{align*}
hence $\int_0^{1/|y|} \frac{u^{\alpha/2}}{(u+1)^{\alpha+1}}\,\du \approx 1$. Moreover, for $|y| >1$, 
\begin{align*}
    \int_0^{1/|y|} \frac{u^{\alpha/2}}{(u+1)^{\alpha+1}}\,\du \approx \int_0^{1/|y|} u^{\alpha/2}\,\du \approx |y|^{-\alpha/2-1}.
\end{align*}
Therefore, 
\begin{align*}
    \int_0^{1/|y|} \frac{u^{\alpha/2}}{(u+1)^{\alpha+1}}\,\du \approx 1\wedge |y|^{-\alpha/2-1}.
\end{align*}
From \eqref{eq:Axy_podst},
\begin{align}\label{eq:oszac_A}
    A(y) \approx |y|^{-\alpha-1} \wedge |y|^{-\alpha/2}.
\end{align}
For the integral $B(x,y)$, we proceed as follows,
\begin{align*}
    B(x,y) \lesssim  \int_1^\infty |w-y|^{-\alpha-1} \,\dw \leq \int_1^\infty w^{-\alpha-1}\,\dw \approx 1,
\end{align*}
and 
\begin{align*}
    B(x,y) \leq |y|^{-\alpha-1} \int_1^\infty |w-1/2|^{-\alpha-1}\,\dw \approx |y|^{-\alpha-1}.
\end{align*}
Hence,
\begin{align}\label{eq:oszac_B}
    B(x,y) \lesssim 1\wedge |y|^{-\alpha-1}.
\end{align}

Furthermore, note that from \eqref{eq:oszac_A} and \eqref{eq:oszac_B},
\begin{align*}
    I_1(x,y) \approx A(y) + B(x,y) \lesssim \big(|y|^{-\alpha-1}\wedge |y|^{-\alpha/2}\big) + \big(1\wedge |y|^{-\alpha-1}\big) \lesssim |y|^{-\alpha-1}\wedge |y|^{-\alpha/2}.
\end{align*}
Indeed, for $|y|\leq 1$ it is obvious, because $|y|^{-\alpha/2} \geq 1$. On the other hand, for $|y|>1$, $1\wedge |y|^{-\alpha-1} = |y|^{-\alpha-1} = |y|^{-\alpha-1} \wedge |y|^{-\alpha/2}$.
Further, it is obvious that
\begin{align*}
    I_1(x,y) \gtrsim A(y) \approx |y|^{-\alpha-1} \wedge |y|^{-\alpha/2},
\end{align*}
hence
\begin{align}\label{eq:oszac_A+B}
    I_1(x,y) \approx |y|^{-\alpha-1} \wedge |y|^{-\alpha/2}.
\end{align}

Moreover, for $|y|<1$ we have $|x-y| \lesssim 1$, hence from \eqref{eq:oszac_A+B}, we get
\[
I_1(x,y) \approx |y|^{-\alpha/2} \approx \big( 1\wedge |y|\big)^{-\alpha/2} \big( 1\wedge |x-y|^{-\alpha-1}\big).
\]
Similarly, for $|y|\geq 1$, we have $|x-y| = x + |y| \geq x+1 > x + 2x = 3x$ and $|x-y|\geq |y| = |x-y| - |x| > |x-y| - \tfrac13 |x-y| = \tfrac23 |x-y|$ and from \eqref{eq:oszac_A+B},
\begin{align*}
    I_1(x,y) \approx |y|^{-\alpha-1} \approx |x-y|^{-\alpha-1} =  \big( 1\wedge |y|\big)^{-\alpha/2} \big( 1\wedge |x-y|^{-\alpha-1}\big).
\end{align*}
To sum up, we obtain so far that for $x\in (0,1/2)$, 
\begin{align}\label{eq:I_1_2}
    I_1(x,y) \approx \big( 1\wedge |y|\big)^{-\alpha/2} \big( 1\wedge |x-y|^{-\alpha-1}\big).
\end{align}

\medskip
Now, assume that $x\geq 1/2$. From \eqref{eq:p_t_approx} we have
\begin{align*}
    I_1(x,y) &\approx \int_0^\infty \big( 1 \wedge w^{\alpha/2}\big) \big(1 \wedge |x-w|^{-\alpha-1}\big)|w-y|^{-\alpha-1}\,\dw \\
    &\approx \int_0^{1/4} w^{\alpha/2} |x-w|^{-\alpha-1} |w-y|^{-\alpha-1} \dw \\
    &+ \int\limits_{|x-w|< 1/4} |w-y|^{-\alpha-1} \dw \\
    &+ \int\limits_{\,\,\substack{ |x-w| \geq 1/4 \\[3pt] w\geq 1/4 }} |x-w|^{-\alpha-1} |w-y|^{-\alpha-1} \dw =: C(x,y) + D(x,y) + E(x,y). 
\end{align*}
From \eqref{eq:oszac_A},
\begin{align*}
    C(x,y) &\approx |x|^{-\alpha-1} \int_0^{1/4} w^{\alpha/2} |w-y|^{-\alpha-1} \,\dw \approx |x|^{-\alpha-1} A(4y) \\
    &\approx |x|^{-\alpha-1}\big(|4y|^{-\alpha-1} \wedge |4y|^{-\alpha/2}\big) \approx |x|^{-\alpha-1}\big(|y|^{-\alpha-1} \wedge |y|^{-\alpha/2}\big). 
\end{align*}
Furthermore,
\begin{align*}
    D(x,y) &= \int\limits_{|x-w|< 1/4} |w-y|^{-\alpha-1} \dw = \alpha^{-1} (x+|y|)^{-\alpha} \Big[ \Big( 1 - \frac{1/4}{x+|y|}\Big)^{-\alpha} - \Big( 1 + \frac{1/4}{x+|y|}\Big)^{-\alpha}  \Big] \\
    &\approx \alpha^{-1} (x+|y|)^{-\alpha} \frac{1/4}{x+|y|} \approx |x-y|^{-\alpha-1}.
\end{align*}
From the fact that $a\vee b \approx a+b$, $a,b\geq 0$, we also get that
\begin{align*}
    E(x,y) &= \int\limits_{\,\,\substack{ |x-w| \geq 1/4 \\[3pt] w\geq 1/4 }} \big(|x-w|\wedge |w-y|\big)^{-\alpha-1} \big(|x-w| \vee |w-y|\big)^{-\alpha-1} \dw \\
    &\approx \int\limits_{\,\,\substack{ |x-w| \geq 1/4 \\[3pt] w\geq 1/4 }} \big(|x-w|\wedge |w-y|\big)^{-\alpha-1} \big(|x-w| + |w-y|\big)^{-\alpha-1} \dw \\
    &\leq |x-y|^{-\alpha-1} \int\limits_{\,\,\substack{ |x-w| \geq 1/4 \\[3pt] w\geq 1/4 }} \big(|x-w|\wedge |w-y|\big)^{-\alpha-1}\,\dw \\
    &\approx |x-y|^{-\alpha-1}.
\end{align*}
Hence, in case $x\geq 1/2$ we obtain that
\begin{align}\label{eq:I_1_x}
    I_1(x,y) \approx |x|^{-\alpha-1}\big(|y|^{-\alpha-1} \wedge |y|^{-\alpha/2}\big) + |x-y|^{-\alpha-1}.
\end{align}
Let $|y|\geq 1$. It is obvious that $|x|\cdot |y| \approx \big(|x|\wedge |y|\big) \big(|x|+|y|\big) = \big(|x|\wedge |y|\big) |x-y|$ and then from \eqref{eq:I_1_x},
\begin{align*}
    I_1(x,y) &\approx \big(|x|\cdot |y|\big)^{-\alpha-1} + |x-y|^{-\alpha-1} \approx |x-y|^{-\alpha-1} \big[ 1 + \big(|x| \wedge |y|\big)^{-\alpha-1}\big] \\
    &\approx |x-y|^{-\alpha-1} = \big(1\wedge |y|\big)^{-\alpha/2} \big(1\wedge |x-y|^{-\alpha-1} \big).
\end{align*}
Now assume that $|y|<1$. Then $|x-y| = x + |y| \geq \tfrac12|y| + |y| = \tfrac32 |y|$ and $|x| = |x-y| - |y| \geq |x-y| - \tfrac23 |x-y| = \tfrac13 |x-y|$. Therefore, from \eqref{eq:I_1_x},
\begin{align*}
    I_1(x,y) &\approx |x|^{-\alpha-1} |y|^{-\alpha/2} + |x-y|^{-\alpha-1} \approx |x-y|^{-\alpha-1} \big(1 + |y|^{-\alpha/2}\big) \approx |x-y|^{-\alpha-1} |y|^{-\alpha/2} \\
    &= |x-y|^{-\alpha-1} \big(1 \wedge |y|\big)^{-\alpha/2} \approx \big(1\wedge |y|\big)^{-\alpha/2} \big(1\wedge |x-y|^{-\alpha-1}\big).
\end{align*}
Hence, for $x\geq 1/2$,
\begin{align}\label{eq:I_1_1}
    I_1(x,y) \approx \big( 1\wedge |y|\big)^{-\alpha/2} \big( 1\wedge |x-y|^{-\alpha-1}\big).
\end{align}

From \eqref{eq:I_1_2},  \eqref{eq:I_1_1} and \eqref{eq:p_t_approx} it follows that
\begin{align*}
    I_1(x,y) \approx \big( 1\wedge |y|\big)^{-\alpha/2} p_1(x,y), \qquad x>0, ~y<0,
\end{align*}
and then from \eqref{eq:I_t_scaling1} and \eqref{eq:scaling_pt}, for $t>0$, 
\begin{align*}
    I_t(x,y) &= t^{-1/\alpha-1} I_1\big( t^{-1/\alpha} x, t^{-1/\alpha} y\big) \\
    &\approx t^{-1} \Big( 1\wedge \frac{|y|}{t^{1/\alpha}}\Big)^{-\alpha/2} t^{-1/\alpha}p_1\big(t^{-1/\alpha} x, t^{-1/\alpha} y\big) \\
    &= t^{-1} \Big( 1\wedge \frac{|y|}{t^{1/\alpha}}\Big)^{-\alpha/2} p_t(x,y),
\end{align*}
which completes the proof.
\end{proof}

\begin{corollary}\label{cor:oszac_p^D_nu}
    For $\alpha\in (0,2)$, $t>0$, $x>0$ and $y<0$,
    \begin{align*}
        \mathcal{J}(t,x,y) := \int_D  p_t^D(x,z)\nu(z,y)\,\dz \approx t^{-1} \Big(1\wedge \frac{|x|}{t^{1/\alpha}}\Big)^{\alpha/2}\Big(1\wedge \frac{|y|}{t^{1/\alpha}}\Big)^{-\alpha/2} p_t(x,y).
    \end{align*}
\end{corollary}

\begin{proof}
    It follows directly from \eqref{eq:asymptotic}, \eqref{eq:asymp_pstwo} and from Lemma \ref{lem_oszac_calka_z_p^D_po_D}.
\end{proof}

\begin{corollary}\label{cor:p_t_nu_po_D_i_Dc}
    For $\alpha\in (0,2)$, $t>0$ and $x>0$,
    \begin{align*}
        \int_D  \dz \int_{D^c} \dy \, p_t^D(x,z)\nu(z,y) \approx  t^{-1} 
        \left[ \left(\frac{|x|}{t^{1/\alpha}}\right)^{-\alpha} \wedge \left(\frac{|x|}{t^{1/\alpha}}\right)^{\alpha/2}\right].
    \end{align*}
\end{corollary}

\begin{proof}
From the Tonelli's theorem and from Corollary \ref{cor:oszac_p^D_nu},
\begin{align}\label{eq:I_1_1_1}
    I(t,x) &:= \int_D  \dz  \int_{D^c} \dy \, p_t^D(x,z)\nu(z,y) \nonumber\\
    &\approx t^{-1} \Big(1\wedge \frac{|x|}{t^{1/\alpha}}\Big)^{\alpha/2} \int_{D^c} \Big(1\wedge \frac{|y|}{t^{1/\alpha}}\Big)^{-\alpha/2} p_t(x,y)\,\dy.
\end{align}
Hence, it suffices to find an estimation of $A(t,x) := \int_{-\infty}^0 \big(1\wedge \frac{|y|}{t^{1/\alpha}}\big)^{-\alpha/2} p_t(x,y)\,\dy$. Note that by substitution $y = t^{1/\alpha}w$,
\begin{align}\label{eq:A_1_scaling}
    A(t, t^{1/\alpha}x) = \int_{-\infty}^0 \big(1\wedge |w|\big)^{-\alpha/2} p_1(x,w)\,\dw = A(1,x).
\end{align}
Moreover, from \eqref{eq:p_t_approx},
\begin{align*}
    A(1,x) &\approx \int_{0}^\infty \big(1\wedge |y|\big)^{-\alpha/2} \big(1\wedge |x+y|^{-\alpha-1}\big)\,\dy \\
    &= \int_0^1 |y|^{-\alpha/2} \big(1\wedge |x+y|^{-\alpha-1}\big)\,\dy + \int_1^\infty \big(1\wedge |x+y|^{-\alpha-1}\big)\,\dy.
\end{align*}
For $|x|\leq 1$, 
\begin{align*}
    A(1,x) &\approx \int_0^1 |y|^{-\alpha/2}\,\dy + \int_1^\infty |x+y|^{-\alpha-1}\,\dy \approx 1,
\end{align*}
and for $|x|>1$,
\begin{align*}
    A(1,x) &\approx \int_0^1 |y|^{-\alpha/2} |x+y|^{-\alpha-1} \,\dy + \int_1^\infty |x+y|^{-\alpha-1}\,\dy\\
    &= |x|^{-3\alpha/2} \int_0^{1/x} w^{-\alpha/2}(1+w)^{-\alpha-1}\,\dw + |x|^{-\alpha} \int_{1/x}^\infty \frac{\dw}{(w+1)^{\alpha+1}} \\
    &\approx |x|^{-3\alpha/2} \int_0^{1/x} w^{-\alpha/2}\,\dw + |x|^{-\alpha}
    \\
    &\approx |x|^{-\alpha-1} + |x|^{-\alpha} \approx |x|^{-\alpha}.
\end{align*}
Hence, 
\begin{align}
    A(1,x) \approx 1\wedge |x|^{-\alpha},
\end{align}
and from \eqref{eq:A_1_scaling},
\begin{align*}
    A(t,x) = A\big(1, t^{-1/\alpha}x\big) \approx 1\wedge t|x|^{-\alpha}.
\end{align*}
Combining this result with \eqref{eq:I_1_1_1} we get
\begin{align*}
    I(t,x) \approx t^{-1} \Big(1\wedge \frac{|x|^\alpha}{t}\Big)^{1/2} \Big(1\wedge \frac{t}{|x|^{\alpha}}\Big) = t^{-1} \big(t|x|^{-\alpha} \wedge t^{-1/2}|x|^{\alpha/2} \big),
\end{align*}
which is our claim.
\end{proof}

\begin{lemma}\label{lem_oszac_calka_z_p^D}
    For $\alpha\in (0,2)$, $t>0$, $x>0$, $y<0$,
    \begin{align*}
        \int_0^t s^{-1}\Big(1\wedge \frac{|x|^{\alpha/2}}{\sqrt{s}}\Big) &\Big(1\wedge \frac{|y|}{s^{1/\alpha}}\Big)^{-\alpha/2} p_s(x,y) e^{-\nu(y,D)(t-s)}\,\ds \\
        &\approx \Big(1\wedge \frac{|x|^{\alpha/2}}{\sqrt{t}}\Big)\Big(1\wedge \frac{|y|^{\alpha/2}}{\sqrt{t}}\Big)p_t(x,y).
    \end{align*} 
\end{lemma}

\begin{proof}
    Let $I_t(x,y) := \int_0^t s^{-1}\big(1\wedge \frac{|x|^{\alpha/2}}{\sqrt{s}}\big) \big(1\wedge \frac{|y|}{s^{1/\alpha}}\big)^{-\alpha/2} p_s(x,y) e^{-\nu(y,D)(t-s)}\,\ds$. By substitution $s = tu$ and from \eqref{eq:scaling_pt} and \eqref{nu_scaling2}, we get
    \begin{align}\label{eq:I_t_1_sc}
        &I_t(t^{1/\alpha}x, t^{1/\alpha}y) \\
        &= \int_0^t s^{-1}\Big(1\wedge \frac{\sqrt{t}|x|^{\alpha/2}}{\sqrt{s}}\Big) \Big(1\wedge \frac{t^{1/\alpha}|y|}{s^{1/\alpha}}\Big)^{-\alpha/2} p_s(t^{1/\alpha}x,t^{1/\alpha}y) e^{-\nu(t^{1/\alpha}y,D)(t-s)}\,\ds \nonumber \\
        &= t^{-1/\alpha} \int_0^1 u^{-1}\Big(1\wedge \frac{|x|^{\alpha/2}}{\sqrt{u}}\Big) \Big(1\wedge \frac{|y|}{u^{1/\alpha}}\Big)^{-\alpha/2}  p_{u}(x,y) e^{-\nu(y,D)(1-u)}\,\du \nonumber \\
        &= t^{-1/\alpha} I_1(x,y).
    \end{align}
    Hence it suffices to find an estimation of $I_1(x,y)$.

    Note that
    \begin{align}\label{eq:I_1_A_B}
        I_1(x,y) &=
        \int_0^{1/2}\ldots \du + \int_{1/2}^1 \ldots \du  := A(x,y) + B(x,y).
    \end{align}
    For $u\in [1/2, 1)$, we have
    \begin{align*}
        u^{-1}\Big(1\wedge \frac{|x|^{\alpha/2}}{\sqrt{u}}\Big) \Big(1\wedge \frac{|y|}{u^{1/\alpha}}\Big)^{-\alpha/2}  p_{u}(x,y) \approx \big(1\wedge |x|\big)^{\alpha/2} \big(1\wedge |y|\big)^{-\alpha/2} p_{1}(x,y).
    \end{align*}
Furthermore, note that $1-e^{-\eta/2} \approx 1\wedge \eta$, $\eta\geq 0$. Indeed, for $\eta\geq 1$, $1-e^{-\eta/2} \approx 1 = 1\wedge \eta$. For $0\leq \eta<1$, there exists a constant $c\in (0,\eta)$ such that $1-e^{-\eta/2} = \tfrac\eta2\big(1  - \tfrac14 e^{-c} \eta\big)$. Then, $1-e^{-\eta/2} \approx \tfrac\eta2 \approx 1\wedge \eta$, which follows from the inequality $1\geq 1-\tfrac14 e^{-c}\eta \geq 3/4$. From this observation, from \eqref{nu_scaling}, it follows that
    \begin{align*}
        \int_{1/2}^1 e^{-\nu(y,D)(1-u)}\,\du = \frac{1}{\nu(y,D)} \Big[ 1 - e^{-\nu(y,D)/2}\Big] \approx \frac{1}{\nu(y,D)} \big( 1 \wedge \nu(y,D)\big) \approx \big(1 \wedge |y|\big)^\alpha.
    \end{align*}

Therefore,
\begin{align}\label{eq:B_est}
    B(x,y) \approx \big(1\wedge |x|\big)^{\alpha/2} \big(1\wedge |y|\big)^{\alpha/2} p_{1}(x,y).
\end{align}

For $u\in (0,1/2)$, there exists a constant $C>0$ such that 
    \begin{align}\label{eq:oszac_exp}
        e^{-\nu(y,D)(1-u)} \leq e^{-C|y|^{-\alpha}} \lesssim \big(1\wedge |y|\big)^{\alpha+1}.
    \end{align}
From Lemma \ref{lem_oszac_1_1} and \eqref{eq:oszac_exp},
\begin{align*}
    A(x,y) \lesssim \big(1\wedge |x|\big)^{\alpha/2} \big(1\wedge |y|\big)^{\alpha/2} \big(|x-y|^{-1}\wedge |x-y|^{-\alpha-1}\big)\big(1\wedge |y|\big).
\end{align*}
Note that for $|x-y|<1$ obviously we have $|y| < 1$ and
\[
\big(|x-y|^{-1}\wedge |x-y|^{-\alpha-1}\big)\big(1\wedge |y|\big) = \frac{|y|}{|x-y|} < 1 = \big(1 \wedge |x-y|^{-\alpha-1}\big).
\]
Similarly, for $|x-y|\geq 1$, 
\[
\big(|x-y|^{-1}\wedge |x-y|^{-\alpha-1}\big)\big(1\wedge |y|\big) = |x-y|^{-\alpha-1} \big(1\wedge |y|\big) \leq |x-y|^{-\alpha-1} = \big(1 \wedge |x-y|^{-\alpha-1}\big).
\]
Hence,
\begin{align}\label{eq:A_est}
    A(x,y) \lesssim \big(1\wedge |x|\big)^{\alpha/2} \big(1\wedge |y|\big)^{\alpha/2} \big(1 \wedge |x-y|^{-\alpha-1}\big) \approx \big(1\wedge |x|\big)^{\alpha/2} \big(1\wedge |y|\big)^{\alpha/2} p_1(x,y).
\end{align}

From \eqref{eq:I_1_A_B}, \eqref{eq:B_est} and \eqref{eq:A_est} we get
\begin{align}\label{eq:I_t_est}
    I_1(x,y) \approx \big(1\wedge |x|\big)^{\alpha/2} \big(1\wedge |y|\big)^{\alpha/2} p_1(x,y).
\end{align}
and then from \eqref{eq:I_t_1_sc} and \eqref{eq:I_t_est} and \eqref{eq:scaling_pt}, for $t>0$,
\begin{align*}
    I_t(x,y) &= t^{-1/\alpha} I_1\big( t^{-1/\alpha}x, t^{-1/\alpha}y \big) \approx \Big(1\wedge \frac{|x|^{\alpha/2}}{\sqrt{t}}\Big) \Big(1\wedge \frac{|y|^{\alpha/2}}{\sqrt{t}}\Big) p_t(x,y),
\end{align*}
which completes the proof.
\end{proof}

From Lemma \ref{lem_oszac_calka_z_p^D_po_D} and Lemma \ref{lem_oszac_calka_z_p^D} and \eqref{eq:approx_p_D} the following conclusion follows immediately. 

\begin{corollary}\label{K_est}
    For $\alpha\in (0,2)$, $t>0$, $x>0$, and $y<0$,
    \begin{align*}
        \mathcal{K}(t,x,y) := \int_0^t \ds \int_D \dz \, p_s^D(x,z)\nu(z,y) e^{-\nu(y,D)(t-s)} \approx \Big(1\wedge \frac{|x|^{\alpha/2}}{\sqrt{t}}\Big)\Big(1\wedge \frac{|y|^{\alpha/2}}{\sqrt{t}}\Big)p_t(x,y).
    \end{align*}
\end{corollary}

\begin{proof}
From \eqref{eq:asymptotic} and \eqref{eq:asymp_pstwo},
    \begin{align*}
         \mathcal{K}(t,x,y) &\approx \int_0^t \ds \int_D \dz \, \Big( 1 \wedge \frac{|x|^{\alpha/2}}{\sqrt{s}}\Big) \Big( 1 \wedge \frac{|z|^{\alpha/2}}{\sqrt{s}}\Big)p_s(x,z) \nu(z,y) e^{-\nu(y,D)(t-s)} \\
         &= \int_0^t \ds  \, \Big( 1 \wedge \frac{|x|^{\alpha/2}}{\sqrt{s}}\Big) e^{-\nu(y,D)(t-s)} \int_D\Big( 1 \wedge \frac{|z|^{\alpha/2}}{\sqrt{s}}\Big)p_s(x,z)\nu(z,y)\,\dz.
    \end{align*}
From Lemma \ref{lem_oszac_calka_z_p^D_po_D},
\begin{align*}
    \mathcal{K}(t,x,y) \approx \int_0^t  s^{-1}\Big( 1 \wedge \frac{|x|^{\alpha/2}}{\sqrt{s}}\Big)  \Big(1\wedge \frac{|y|}{s^{1/\alpha}}\Big)^{-\alpha/2} p_s(x,y)e^{-\nu(y,D)(t-s)}\,\ds, 
\end{align*}
and further, from Lemma \ref{lem_oszac_calka_z_p^D},
\begin{align*}
    \mathcal{K}(t,x,y) \approx \Big(1\wedge \frac{|x|^{\alpha/2}}{\sqrt{t}}\Big)\Big(1\wedge \frac{|y|^{\alpha/2}}{\sqrt{t}}\Big)p_t(x,y),
\end{align*}
which is our claim.
\end{proof}

\begin{corollary}
    For $\alpha\in (0,2)$, $t>0$, and $x>0$,
    \begin{align*}
        \mathcal{L}(t,x) := \int_0^t \ds \int_D \dz \int_{D^c} \dy \, p_s^D(x,z)\nu(z,y) e^{-\nu(y,D)(t-s)} \approx \left(\frac{|x|^{\alpha/2}}{\sqrt{t}}\right)^{-2} \wedge \frac{|x|^{\alpha/2}}{\sqrt{t}}.
    \end{align*}
\end{corollary}

\begin{proof}
By substituting $s=tu$, $y=t^{1/\alpha}a$, and $z = t^{1/\alpha}b$,we get
\begin{align*}
    \mathcal{L}(t,t^{1/\alpha}x) = \int_0^1 \du \int_D \db \int_{D^c} \da \, t^{1+2/\alpha}p_{tu}^D(t^{1/\alpha}x,t^{1/\alpha}b)\nu(t^{1/\alpha}b,t^{1/\alpha}a) e^{-t\nu(t^{1/\alpha}a,D)(1-u)}.
\end{align*}
From \eqref{eq:scaling} and \eqref{nu_scaling2},
\begin{align}\label{eq:L_1_scaling}
    \mathcal{L}(t,t^{1/\alpha}x) = \int_0^1 \du \int_D \db \int_{D^c} \da \, p_u^D(x,b) \nu(b,a) e^{-\nu(a,D)(1-u)} = \mathcal{L}(1,x).
\end{align}
Hence, it suffices to find an estimation of $\mathcal{L}(1,x)$.

Note that from the Tonelli's theorem and from Corollary \ref{K_est},
\begin{align*}
    \mathcal{L}(1,x) &= \int_{-\infty}^0 \mathcal{K}(1,x,y)\,\dy \approx \big(1\wedge |x|^{\alpha/2}\big) \int_{-\infty}^0 \big(1\wedge |y|^{\alpha/2}\big)\big(1 \wedge |x-y|^{-\alpha-1}\big)\,\dy.
\end{align*}
Moreover,
\begin{align*}
    A&:=\int_{-\infty}^0 \big(1\wedge |y|^{\alpha/2}\big)\big(1 \wedge |x-y|^{-\alpha-1}\big)\,\dy \\
    &\approx \int_{-1}^0 |y|^{\alpha/2}\big(1 \wedge |x-y|^{-\alpha-1}\big)\,\dy + \int_{-\infty}^{-1} \big(1 \wedge |x-y|^{-\alpha-1}\big)\,\dy.
\end{align*}
For $|x|\leq 1$, 
\begin{align*}
    A \approx \int_{-1}^0 |y|^{\alpha/2} \,\dy + \int_{-\infty}^{-1} |x-y|^{-\alpha-1}\,\dy \approx 1 = |x|^{-\alpha} \big(1\wedge |x|\big)^{\alpha},
\end{align*}
and for $|x|>1$,
\begin{align*}
    A &\approx \int_{-1}^0 |y|^{\alpha/2} |x-y|^{-\alpha-1} \,\dy + \int_{-\infty}^{-1} |x-y|^{-\alpha-1}\,\dy \lesssim \int_{-\infty}^{0} |x-y|^{-\alpha-1}\,\dy \approx |x|^{-\alpha}\\
    &= |x|^{-\alpha} \big(1\wedge |x|\big)^{\alpha},
\end{align*}
and
\begin{align*}
    A \gtrsim \int_{-\infty}^{-1} |x-y|^{-\alpha-1}\,\dy = 
    \frac{1}{\alpha}(1+x)^{-\alpha} 
    \geq \frac{2^{-\alpha}}{\alpha} x^{-\alpha}
    \approx 
    |x|^{-\alpha} \big(1\wedge |x|\big)^{\alpha}.
\end{align*}
Hence, $A\approx |x|^{-\alpha} \big(1\wedge |x|\big)^{\alpha}$. 

Therefore, we obtain that 
\begin{align}\label{eq:L_1_est}
    \mathcal{L}(1,x) \approx \big(1\wedge |x|\big)^{\alpha/2} |x|^{-\alpha} \big(1\wedge |x|\big)^{\alpha} = |x|^{-\alpha} \wedge |x|^{\alpha/2}.
\end{align}
Hence, from \eqref{eq:L_1_scaling} and \eqref{eq:L_1_est},
\begin{align*}
    \mathcal{L}(t,x) = \mathcal{L}(1, t^{-1/\alpha}x) \approx t |x|^{-\alpha} \wedge t^{-1/2} |x|^{\alpha/2},
\end{align*}
which is the desired conclusion.
\end{proof}

\subsection{Convergence of integrals}

\begin{lemma}\label{eq:convergence_p_D/t}
For $x,y>0$,
\begin{align*}
    \lim_{t\to 0^+} \frac{p_t^D(x,y)}{t} = \nu(x,y).
\end{align*}
\end{lemma}

\begin{proof}
By the Hunt formula \eqref{eq:Hunt's_formula} and the Ikeda--Watanabe formula \eqref{eq:Ikeda_Watanabe},
\begin{align}
\label{eq:p_D/t_1}
    \frac{p_t^D(x,y)}{t} &=  \frac{p_t(x,y)}{t} - \frac{1}{t} \mE_x^Y\big[ \tau_D < t; ~p_{t-\tau_D}(Y_{\tau_D},y)\big].
\end{align}
From P{\`o}lya \cite{Polya1923} (see also Cygan et al. \cite{MR3646773}), $\lim\limits_{t\to 0^+} p_t(x,y)/t = \nu(x,y)$. Moreover, it is obvious that $Y_{\tau_D}<0,$ hence $|Y_{\tau_D}-y|>y$ and from \eqref{eq:p_t_approx} we get 
\begin{align}\label{eq:p_D/t_2}
\frac{1}{t} \mE_x^Y\left[ \tau_D < t; ~p_{t-\tau_D}(Y_{\tau_D},y)\right] \lesssim 
  \frac{1}{t} \mE_x^Y\left[ \tau_D < t; \frac{t-\tau_D}{|Y_{\tau_D}-y|^{\alpha+1}}\right] \leq
  |y|^{-1-\alpha} \,\mP_x^Y\left( \tau_D < t \right)\to 0,
\end{align}
as $t\to 0^+$.
\end{proof}

\begin{lemma}\label{ciaglosc2}
The function
\[
(0,\infty) \times D \times \overline{D}^c \ni (t,x,y) \mapsto \mathcal{J}(t,x,y) = \int_D p_t^D(x,z)\nu(z,y)\,\dz
\]
is continuous.
\end{lemma}

\begin{proof}
Let $g(t,x,y,z) := p_t^D(x,z)\nu(z,y)$. 
Let $\varepsilon >0$ and define $T = [\varepsilon, \infty)$, $K=[\varepsilon, \infty)$, $L = (-\infty, -\varepsilon]$. 
We will show that for all $(t,x,y)\in T\times K\times L$ and for any sequence $(t_n, x_n,y_n)\subset T\times K\times L$ such that $\lim\limits_{n\to\infty} (t_n, x_n,y_n) = (t,x,y)$, we have $
\lim\limits_{n\to\infty} \mathcal{J}(t_n,x_n,y_n) = \mathcal{J}(t,x,y)$.

Note that $T\times K\times L \ni (t,x,y) \mapsto p_t^D(x,y)$ is continuous. Therefore, the function $(t,x,y) \mapsto g(t,x,y,z)$, $z\in D$, is continuous on $T\times K\times L$, hence $g(t_n, x_n, y_n, z) \to g(t,x,y,z)$ as $n\to\infty$.
Furthermore, from \eqref{eq:approx_p_D} it follows that for $(t,x,y) \in T \times  K \times L$ and $z\in D$ we have
\begin{align*}
        |g(t,x,y,z)| &\lesssim  t^{-1/\alpha} \nu(z,y) \lesssim   \varepsilon^{-1/\alpha} |z+\varepsilon|^{-1-\alpha}.
\end{align*}
Moreover, $\int_D |z+\varepsilon|^{-1-\alpha} \,\dz < \infty$, hence from the dominated convergence theorem we then obtain continuity of the function $\mathcal{J}$ on $T\times K\times L$. 

Since $\varepsilon>0$ was chosen arbitrarily, we conclude that the function $\mathcal{J}$ is continuous on $(0,\infty)\times D\times \overline{D}^c$.
\end{proof}

\begin{lemma}\label{lem:zb_nu}
For $x>0$ and $y<0$, we have
\begin{align*}
    \mathcal{J}(t,x,y)= \int_D p_t^D(x,z)\nu(z,y)\,\dz\to \nu(x,y),
\end{align*}
as $t\to 0^+$.
\end{lemma}

\begin{proof}
Let $\varepsilon>0$ and define $K = [\varepsilon, \infty)$, $L = (-\infty,-\varepsilon]$. Let $U,V\subset \R$, $U\neq V$, be the open sets such that $x\in K \subset U \subset V \subset D$. Let $\{\alpha_1, \alpha_2\}$ be a partition of unity for the sets $U^c$ and $V$, i.e. for $i=1,2$, the functions $\alpha_i: D\to [0,1]$ satisfy the following conditions: $\alpha_i\in C^\infty(D)$, $\supp{\alpha_1} \subseteq V$, $\supp{\alpha_2}\subseteq U^c \cap D$, $\alpha_1 = 1$ on $U$, $\alpha_2 = 1$ on $V^c\cap D$ and $\alpha_1(x) + \alpha_2(x) = 1$ for $x\in D$. 

For $z\in D$ and $y\in L$, let $\varphi_1(z,y) := \alpha_1(z)\nu(z,y)$, $\varphi_2(z,y) := \alpha_2(z)\nu(z,y)$. Then $\supp{\varphi_1}\subseteq V$, $\supp{\varphi_2}\subseteq U^c\cap D$ and
\begin{align*}
    \mathcal{J}(t,x,y) = \int_D p_t^D(x,z)\varphi_1(z,y)\,\dz + \int_D p_t^D(x,z)\varphi_2(z,y)\,\dz =: A(t,x,y) + B(t,x,y).   
\end{align*}

From the construction it follows that $\varphi_1(\cdot,y)\in C_0(D)$, hence from Lemma \ref{theorem_p^D_feller},
\begin{align*}
    A(t,x,y) = \int_D p_t^D(x,z)\varphi_1(z,y)\,\dz \to \varphi_1(x,y) = \nu(x,y),
\end{align*}
as $t\to 0^+$.

We will show that $B(t,x,y) \to 0$, as $t\to 0^+$. From \eqref{eq:approx_p_D} we have
\begin{equation*}
    \begin{split}
       B(t,x,y) &= \int_{\supp{\varphi_2(\cdot,y)}}  p_t^D(x,z)\varphi_2(z,y) \,\dz \\
       &\leq \int_{U^c\cap D}  p_t^D(x,z)\nu(z,y) \,\dz \\
        &\lesssim  t\int_{U^c\cap D} |x-z|^{-\alpha-1} |z-y|^{-\alpha-1} \,\dz.
    \end{split}
\end{equation*}
Let $\rho := \dist(K, U^c\cap D)>0$ and $\eta := \dist(L, U^c\cap D)$. Then, 
\begin{align*}
    B(t,x,y) \lesssim  t \rho^{-\alpha-1} \eta^{-\alpha-1} |U^c\cap D| \to 0,
\end{align*}
as $t\to 0^+$.
\end{proof}

\begin{lemma}\label{lem_zb3}
For $x>0$, $y<0$, and $\mu \in \{0,1\}$, we have
\[
\lim_{t\to 0^+} \frac{1}{t} \int_0^t \ds \int_D \dz \, p_s^D(x,z)\nu(z,y) e^{-\mu \nu(y,D)(t-s)} = \nu(x,y).
\]
\end{lemma}

\begin{proof}
Let $\varepsilon>0$ and $K := [\varepsilon, \infty)$, $L:= (-\infty,-\varepsilon]$. Without loss of generality, we may assume that $t\leq 1$. Let $x\in K$ and $y\in L$. From Lemma \ref{ciaglosc2} and Lemma \ref{lem:zb_nu}, by putting $\mathcal{J}(0,x,y) := \nu(x,y)$, we obtain a continuous function $t\mapsto \mathcal{J}(t,x,y)$ on the interval $[0,1]$. Hence, by the Mean Value Theorem it follows that there exists $c = c(t) \in [0,t]$ such that
\begin{align*}
    \frac{1}{t} \int_0^t \ds \int_D \dz \, p_s^D(x,z)\nu(z,y) e^{-\mu \nu(y,D)(t-s)} &= \frac{1}{t} \int_0^t \mathcal{J}(s,x,y)e^{-\mu \nu(y,D)(t-s)} \,\ds \\
    &= \mathcal{J}(c(t),x,y)e^{-\mu \nu(y,D)(t-c(t))}.
\end{align*}
Therefore, from Lemma \ref{lem:zb_nu},
\begin{align*}
    \lim_{t\to 0^+} \frac{1}{t} \int_0^t \ds \int_D \dz \, p_s^D(x,z)\nu(z,y) e^{-\mu \nu(y,D)(t-s)} &= \lim_{t\to 0^+} \mathcal{J}(c(t),x,y)e^{-\mu \nu(y,D)(t-c(t))} \\
    &= \nu(x,y),
\end{align*}
which completes the proof.
\end{proof}

\begin{lemma}\label{lem_zb1}
For $x>0$, we have
\[
\lim_{t\to 0^+} \frac{1}{t} \mP_x^Y(\tau_D < t) = \lim_{t\to 0^+} \frac{1}{t} \int_0^t \ds \int_D \dz  \, p_s^D(x,z)\nu(z,D^c) = \nu(x,D^c).
\]
\end{lemma}

\begin{proof}
The first equality follows from \eqref{eq:Ikeda_Watanabe}. We will prove the latter equality. 
Without loss of generality, we may assume that $t\leq 1$. Note that from the Tonelli's theorem,
\begin{align*}
    A(t,x) := \frac{1}{t} \int_0^t \ds \int_D \dz  \, p_s^D(x,z)\nu(z,D^c) = \int_{D^c}  \left[\frac{1}{t} \int_0^t \mathcal{J}(s,x,y) \,\ds \right] \dy.
\end{align*}
By a similar argument as in the proof of Lemma \ref{lem_zb3}, we conclude that there exists $c = c(t) \in [0,t]$, such that 
\begin{align*}
    \frac{1}{t} \int_0^t \mathcal{J}(s,x,y) \,\ds =  \mathcal{J}(c(t),x,y).
\end{align*}
From Corollary \ref{cor:oszac_p^D_nu} and \eqref{eq:p_t_approx},   
\begin{align*}
    \mathcal{J}(c(t),x,y) &\lesssim (c(t))^{-1} \Big(1\wedge \frac{|y|}{(c(t))^{1/\alpha}}\Big)^{-\alpha/2} p_{c(t)}(x,y) \\
    &\leq (c(t))^{-1} \big(1\wedge |y|\big)^{-\alpha/2} p_{c(t)}(x,y) \\
    &\lesssim \big(1\wedge |y|\big)^{-\alpha/2} |x-y|^{-\alpha-1}.
\end{align*}
Moreover, 
\begin{align*}
    \int_{D^c} \big(1\wedge |y|\big)^{-\alpha/2} |x-y|^{-\alpha-1}\,\dy  &= \int_D \big(1\wedge y\big)^{-\alpha/2} (x+y)^{-\alpha-1}\,\dy \\
    &= \int_0^1 \frac{\dy}{y^{\alpha/2}(x+y)^{\alpha+1}} + \int_1^\infty \frac{\dy}{(x+y)^{\alpha+1}}<\infty.
\end{align*}
Hence, we use the dominated convergence theorem to obtain that
\begin{align*}
    \lim_{t\to 0^+} A(t,x) =  \int_{D^c} \lim_{t\to 0^+} \mathcal{J}(c(t),x,y)\,\dy = \nu(x,D^c), 
\end{align*}
where the latter convergence follows from Lemma \ref{lem:zb_nu}.
\end{proof}

\begin{lemma}\label{lem_zbieznosc_K_1}
Let $\alpha\in (0,2)$. Assume that $f:\R_*\to [0,\infty)$ is a function such that for some $x\neq 0$ we have $\widehat{\nu}f(x) <\infty$. Then 
\begin{align*}
    \lim_{t\to 0^+}  \frac{1}{t} K_{t,1} f(x) = \widehat{\nu}f(x).
\end{align*}
\end{lemma}

\begin{proof}
Assume that $x>0$. Note that from the Tonelli's theorem,
\begin{align*}
    \frac{1}{t}K_{t,1}f(x) &= \frac{1}{t} \int_0^t\ds \int_D \dz \int_{D^c} \dy \, p_s^D(x,z)\nu(z,y) e^{-\nu(y,D)(t-s)}f(y) \nonumber\\
    &= \int_{D^c}  \Big[ \frac{1}{t}  \int_0^t\ds \int_D \dz \, p_s^D(x,z)\nu(z,y) e^{-\nu(y,D)(t-s)}\Big] f(y) \,\dy \\
    &= \int_{D^c} t^{-1} \mathcal{K}(t,x,y) f(y)\,\dy.
\end{align*}
From Corollary \ref{K_est} and \eqref{eq:p_t_approx}, $t^{-1}\mathcal{K}(t,x,y) \lesssim t^{-1}p_t(x,y)\lesssim \nu(x,y)$.
Moreover,
\begin{align*}
    \int_{D^c} \nu(x,y) f(y)\,\dy = \widehat{\nu}f(x) <\infty.
\end{align*}
Therefore, we can use the dominated convergence theorem to obtain that from Lemma \ref{lem_zb3} with $\mu=1$ we get
\begin{align*}
    \lim_{t\to 0^+} \frac{1}{t}K_{t,1}f(x) = \int_{D^c} \lim_{t\to 0^+} t^{-1} \mathcal{K}(t,x,y) f(y)\,\dy = \int_{D^c} \nu(x,y)f(y)\,\dy = \widehat{\nu}f(x).
\end{align*}

\medskip
The case $x<0$ we prove in the same way, but we propose it for the convenience of the reader. Assume that $x<0$. From the Tonelli's theorem,
\begin{align*}
    \frac{1}{t} K_{t,1}f(x) &=  \frac{1}{t} \int_0^t\dr \int_D \dy  \int_D \dz \, e^{-\nu(x,D)r} \nu(x,y)p_{t-r}^D(y,z)f(z) \\
    &= \int_D \left[ \frac{1}{t} \int_0^t \dr \int_D \dy \, p_{t-r}^D(z,y)\nu(y,x)e^{-\nu(x,D)r} \right] f(z)\,\dz \\
    &= \int_D \left[ \frac{1}{t} \int_0^t \ds \int_D \dy \, p_{s}^D(z,y)\nu(y,x)e^{-\nu(x,D)(t-s)} \right] f(z)\,\dz \\
    &= \int_D t^{-1}\mathcal{K}(t,z,x)f(z)\,\dz.
\end{align*}
From Corollary \ref{K_est} and \eqref{eq:p_t_approx}, $t^{-1}\mathcal{K}(t,z,x) \lesssim t^{-1} p_t(z,x) \lesssim \nu(z,x)$.
Moreover, 
\begin{align*}
    \int_D \nu(z,x)f(z)\,\dz = \widehat{\nu}f(x) < \infty.
\end{align*} 
Therefore, we can use the dominated convergence theorem to obtain that from Lemma \ref{lem_zb3} with $\mu=1$,
\[
    \lim_{t\to 0^+} \frac{1}{t}K_{t,1}f(x) = \int_D \lim_{t\to 0^+} t^{-1}\mathcal{K}(t,z,x)f(z)\,\dz = \int_{D} \nu(z,x)f(z)\,\dz = \widehat{\nu}f(x),
\]
which is our conclusion.
\end{proof}

\begin{corollary}\label{cor:K_t_1_h_beta}
    For $\alpha\in (0,2)$ and $\beta\in (-1,\alpha)$, we have 
\begin{align}\label{eq:K_t_1_h_beta}
    \lim_{t\to 0^+} \frac{1}{t}K_{t,1}h_\beta(x) = \widehat{\nu}h_\beta(x), \qquad x\neq0,
\end{align}
\end{corollary}

\begin{proof}
    Note that for $x>0$ we have
     \[
    \widehat{\nu}h_\beta(x) \approx \int_{D^c} \frac{|y|^\beta}{|x-y|^{\alpha+1}} \,\dy = x^{\beta-\alpha} \int_{0}^\infty \frac{w^\beta}{(1+w)^{\alpha+1}}\,\dw = x^{\beta-\alpha} \mathfrak{B}(\beta+1, \alpha-\beta) <\infty.
    \]
    Similarly, for $x<0$, $\widehat{\nu}h_\beta(x)<\infty$. Therefore, we use Lemma \ref{lem_zbieznosc_K_1} to get \eqref{eq:K_t_1_h_beta}.    
\end{proof}

\subsection{The main theorem}

Now we introduce the main result of this section.

\begin{proposition}\label{gen_inequality}
    Assume that $x\neq 0$, $\alpha\in (0,2)$, $\beta(\alpha-\beta-1)\geq 0$ and let
    \[
    \mathcal{C}(\alpha,\beta,x) := 
    \begin{cases}
         \alpha^{-1} - \mathfrak{B}(\beta+1, \alpha-\beta)-\gamma(\alpha,\beta), &x>0, \\
         \alpha^{-1} - \mathfrak{B}(\beta+1, \alpha-\beta), &x<0,
    \end{cases}
    \]
    where
    \[
    \gamma(\alpha,\beta) := \int_0^1 \frac{(t^\beta-1)(1-t^{\alpha-\beta-1})}{(1-t)^{\alpha+1}}\,\dt \leq 0.
    \]
    Then $\mathcal{C}(\alpha,\beta,x) \geq 0$ and  
    \[
    \lim_{t\to 0^+} \frac{h_\beta(x) - K_th_\beta(x)}{t} = \mathcal{A}_{1,\alpha} \mathcal{C}(\alpha,\beta,x) h_\beta(x) |x|^{-\alpha}.
    \]   
    Moreover, $\mathcal{C}(\alpha,\beta,x) = 0$ if and only if $\beta=0$ or $\beta=\alpha-1$.
\end{proposition}

To prove this theorem, we first have to find the analogous result for the $\alpha$-stable killed L\'evy process $Y$. Below, we prove it for a wider class of functions.

\begin{lemma}
\label{stable_generator_estimate}
    Let $\alpha\in (0,2)$. If $f\in C^2(D)$ satisfies
    \[
    \int_D \frac{|f(y)|}{(1+y)^{\alpha+1}}\,\dy <\infty,
    \]    
    then for $x>0$, we have
    \begin{align*}
        \lim_{t\to 0^+} \frac{f(x) - P^D_tf(x)}{t}
        = \mathrm{p.v.} \int_D (f(x) - f(y))\nu(x,y)\,\dy + f(x)\nu(x,D^c).
    \end{align*}
\end{lemma}

\begin{proof}
    Let $t>0$. Note that 
    \begin{align}\label{eq:G1}
        \frac{f(x) - P^D_tf(x)}{t} &= \frac{1}{t}\int_D (f(x)-f(y))p_t^D(x,y)\,\dy + f(x)\frac{1-p_t^D(x,D)}{t}.
    \end{align}
From \eqref{eq:Ikeda_Watanabe} we have that
\begin{align*}
\frac{1-p_t^D(x,D)}{t} &= \frac{1}{t}\mP^Y_x(\tau_D<t) = \frac{1}{t}\int_0^t \ds \int_D \dy \, p_s^D(x,y)\nu(y,D^c).
\end{align*}
Hence, from Lemma \ref{lem_zb1},
\begin{align}\label{eq:G2}
    \lim_{t\to 0^+} \frac{1-p_t^D(x,D)}{t} = \nu(x,D^c).
\end{align}
Combining \eqref{eq:G1} and \eqref{eq:G2} we obtain the following equality
\begin{align*}
    \lim_{t\to 0^+} \frac{f(x) - P^D_tf(x)}{t} = \lim_{t\to 0^+} \int_D (f(x)-f(y))\frac{p_t^D(x,y)}{t}\,\dy + f(x) \nu(x,D^c),
\end{align*}
with the assumption that the latter limit exists---we will show it in the further part of the proof. We will find this limit in a few steps.

Let $\varepsilon \in (0,x/2]$ be an arbitrary constant. Then we write that
\begin{align}
\label{eq:I+II}
    &\int_D (f(x)-f(y))\frac{p_t^D(x,y)}{t} \,\dy \nonumber \\
    &= \int_{D\cap \{|x-y|<\varepsilon\}} (f(x)-f(y))\frac{p_t^D(x,y)}{t} \,\dy + \int_{D\cap \{|x-y|\geq \varepsilon\}} (f(x)-f(y))\frac{p_t^D(x,y)}{t} \,\dy\nonumber \\
    &=: I_t + II_t.
\end{align}
Note that from \eqref{eq:approx_p_D} it follows that
\begin{align*}
    |f(x) - f(y)| \frac{p_t^D(x,y)}{t} \lesssim |f(x) - f(y)| \cdot |x-y|^{-\alpha-1},
\end{align*}
and
\begin{align}\label{eq:nier1_1}
    &\int_{D\cap \{|x-y|\geq \varepsilon\}} |f(x) - f(y)| \cdot |x-y|^{-\alpha-1} \,\dy  \nonumber \\
    &\leq |f(x)| \int_{D\cap \{|x-y|\geq \varepsilon\}}|x-y|^{-\alpha-1} \,\dy  + \int_{D\cap \{|x-y|\geq \varepsilon\}} \frac{|f(y)|}{|x-y|^{\alpha+1}} \,\dy.
\end{align}

Moreover, for $|x-y|\geq \varepsilon$ there exists a constant $\widehat{C} = \widehat{C}(x,\varepsilon) >0$, such that $\widehat{C}(y+1) \leq |x-y|$. 
To prove it we consider two cases. Let $y\leq x-\varepsilon$. Then for $C_1 := \frac{\varepsilon}{x+1}$ we have $C_1 < \frac{\varepsilon}{x+1-\varepsilon} \leq \frac{\varepsilon}{y+1} \leq \frac{x-y}{y+1}$, which gives us the inequality $C_1(y+1) < |x-y|$. In case $y\geq x+\varepsilon$  we set $C_2 := \frac{\varepsilon}{1+3x/2}$ and then we observe that $C_2 < \frac{\varepsilon}{1+x+\varepsilon} = \frac{1-\frac{x}{x+\varepsilon}}{1+\frac{1}{x+\varepsilon}} \leq \frac{1-\frac{x}{y}}{1+\frac{1}{y}} = \frac{y-x}{y+1}$. This gives us $C_2(y+1) < |x-y|$. Thus, the claim follows by taking $\widehat{C} := C_1\wedge C_2$.

Furthermore, note that from the above observations we get that 
\begin{align*}
    \int_{D\cap \{|x-y|\geq \varepsilon\}} \frac{|f(y)|}{|x-y|^{\alpha+1}} \,\dy \leq \widehat{C}^{-1-\alpha} \int_{D\cap \{|x-y|\geq \varepsilon\}} \frac{|f(y)|}{(1+y)^{\alpha+1}}\,\dy <\infty.
\end{align*}
Hence, the integral of the left-hand side of \eqref{eq:nier1_1} is finite. Therefore, from the dominated convergence theorem and by Lemma \ref{eq:convergence_p_D/t},
\begin{align}
    \label{eq:liminfII}
    \lim_{t\to 0^+} II_t = \int_{D\cap \{|x-y|\geq \varepsilon\}} (f(x)-f(y))\nu(x,y) \,\dy.
\end{align}

For now, we consider the integral $I_t$. From the Hunt formula \eqref{eq:Hunt's_formula} we get
\begin{align}
\label{eq:A-B}
    I_t &= \int_{D\cap \{|x-y|<\varepsilon\}} (f(x)-f(y))\frac{p_t(x,y)}{t} \,\dy - \int_{D\cap \{|x-y|<\varepsilon\}} (f(x)-f(y)) \mathcal{H}(t,x,y) \,\dy \nonumber\\
    &=: A_t - B_t,
\end{align}
where $\mathcal{H}(t,x,y) := t^{-1} \mE_x^Y\big[ \tau_D < t; ~ p_{t-\tau_D}(Y_{\tau_D},y)\big]$. 
From \eqref{eq:p_D/t_2} it follows that
\begin{align*}
    |f(x) - f(y)|\mathcal{H}(t,x,y) &\lesssim |f(x) - f(y)| \cdot |y|^{-\alpha-1}.
\end{align*}
Moreover,
\begin{align}\label{eq:nier1_2}
    &\int_{D\cap\{|x-y| < \varepsilon\}} \frac{|f(x) - f(y)|}{|y|^{\alpha+1}}\,\dy \nonumber \\
    &\leq |f(x)| \int_{D\cap\{|x-y| < \varepsilon\}} \frac{\dy}{|y|^{\alpha+1}}+ \int_{D\cap\{|x-y| < \varepsilon\}} \frac{|f(y)|}{|y|^{\alpha+1}}\,\dy.
\end{align}

Note that for $|x-y|<\varepsilon$ there exists a constant $\widetilde{C} = \widetilde{C}(x)>0$ such that $1+ y \leq  \widetilde{C} y$. Indeed, let $\widetilde{C} := \frac{2}{x} + 4$. Then $\frac{3x}{2} < \frac{\widetilde{C}x}{2} - 1$. Recall that $\varepsilon \leq \frac{x}{2}$. Then, from the assumption, it is obvious that $\frac{x}{2} < y < \frac{3x}{2}$ and then $\widetilde{C}^{-1} (1+y) < \widetilde{C}^{-1} \big(1+\frac{3x}{2}\big) < \frac{x}{2} < y$, which is our claim.

Furthermore, from the above result we obtain that
\begin{align*}
    \int_{D\cap\{|x-y| < \varepsilon\}} \frac{|f(y)|}{|y|^{\alpha+1}}\,\dy \approx \int_{D\cap\{|x-y| < \varepsilon\}} \frac{|f(y)|}{(1+y)^{\alpha+1}}\,\dy <\infty.
\end{align*}
Hence, the integral of the left-hand side of \eqref{eq:nier1_2} is finite. Hence, from the dominated convergence theorem and from \eqref{eq:p_D/t_2},
\begin{align}
\label{eq:B_t_H}
    \lim_{t\to 0^+} B_t = \int_{D\cap \{|x-y|<\varepsilon\}} (f(x) - f(y)) \lim_{t\to 0^+} \mathcal{H}(t,x,y)\,\dy = 0.
\end{align}

Now we consider the integral $A_t$. Let $t>0$ be small enough that $t^{1/\alpha} < \varepsilon$. Note that by the symmetry of $p_t$,
\begin{align*}
    \int_{D\cap \{|x-y|<\varepsilon\}} (y-x)p_t(x,y)\,\dy &= \int_x^{x+\varepsilon} (y-x)p_t(x,y)\,\dy + \int_{x-\varepsilon}^x (y-x)p_t(x,y)\,\dy \\
    &= \int_x^{x+\varepsilon} (y-x)p_t(x,y)\,\dy + \int_x^{x+\varepsilon} (x-w)p_t(x,2x-w)\,\dw \\
    &= \int_x^{x+\varepsilon} (y-x)p_t(x,y)\,\dy - \int_x^{x+\varepsilon} (w-x)p_t(x,w)\,\dw = 0.
\end{align*}
Hence,
\begin{align*}
    |A_t| &= \left| \int_{D\cap\{|x-y|<\varepsilon\}} (f(x) - f(y))\frac{p_t(x,y)}{t}\,\dy \right| \\
    &=\left| \int_{D\cap\{|x-y|<\varepsilon\}} \big(f(y) - f(x)- (y-x) f'(x)\big)\frac{p_t(x,y)}{t}\,\dy \right| \\
    &\leq \int_{D\cap\{|x-y|<\varepsilon\}} \big|f(y) - f(x)- (y-x) f'(x)\big|\frac{p_t(x,y)}{t}\,\dy.
\end{align*}
From the Taylor's theorem,
\begin{align*}
    f(y) - f(x)- (y-x) f'(x) = \frac{f''(c)}{2}(y-x)^2,
\end{align*}
where $c$ lies strictly between $x$ and $y$. From the fact that $\varepsilon \leq x/2$, it follows that $x-\varepsilon \geq x/2$. Hence, for $|x-y|<\varepsilon$,
\begin{align*}
    \big| f(y) - f(x)- (y-x) f'(x) \big| &\leq \sup_{c\in (x-\varepsilon, x+\varepsilon)} |f''(c)| |y-x|^2 \\
    &\leq \sup_{c\in (x/2, 3x/2)} |f''(c)| |y-x|^2 =: C(x) |y-x|^2.
\end{align*}
Hence,
\begin{align*}
    |A_t| \leq C(x)  \int_{D\cap\{|x-y|<\varepsilon\}} |y-x|^2 \frac{p_t(x,y)}{t}\,\dy.
\end{align*}
From \eqref{eq:p_t_approx}, there exists a constant $C_0>0$ such that
\begin{align*}
    |A_t| &\leq C(x) C_0 \int_{D\cap\{|x-y|<\varepsilon\}} |y-x|^2 \frac{1}{t} \Big( t^{-1/\alpha} \wedge \frac{t}{|x-y|^{\alpha+1}}\Big)\,\dy \\
    &= C(x) C_0 t^{-1/\alpha-1} \int_{D\cap\{|x-y|<t^{1/\alpha}\}} |y-x|^2\,\dy + C(x) C_0 \int_{D\cap\{t^{1/\alpha} \leq |x-y|<\varepsilon\}} |y-x|^{1-\alpha}\,\dy \\
    &\leq \tfrac23 C(x) C_0 t^{-1/\alpha-1}  t^{3/\alpha} + C(x)C_0  \int_{D\cap \{|x-y|<\varepsilon\}} |y-x|^{1-\alpha}\,\dy \\
    &= C(x) C_0 \big[ \tfrac23 t^{2/\alpha-1} + \tfrac{2}{2-\alpha}\varepsilon^{2-\alpha}\big].
\end{align*}
Using the above observations, we conclude that
\begin{align}
\label{eq:supA}
    \limsup_{t\to 0^+} |A_t| \leq \tfrac{2}{2-\alpha} C(x) C_0\,\varepsilon^{2-\alpha} =: \widehat{C}(x,\alpha) \varepsilon^{2-\alpha}.
\end{align}

Combining \eqref{eq:A-B}, \eqref{eq:B_t_H} and \eqref{eq:supA} we get
\begin{align}\label{eq:liminfI}
    \liminf_{t\to 0^+} I_t \geq \liminf_{t\to 0^+} A_t - \limsup_{t\to 0^+}  B_t = \liminf_{t\to 0^+} A_t \geq -\widehat{C}(x,\alpha) \varepsilon^{2-\alpha},
\end{align}
and
\begin{align}\label{eq:limsupI}
    \limsup_{t\to 0^+} I_t \leq \limsup_{t\to 0^+} A_t - \liminf_{t\to 0^+} B_t = \limsup_{t\to 0^+} A_t \leq \widehat{C}(x,\alpha) \varepsilon^{2-\alpha}.
\end{align}
Hence, from \eqref{eq:I+II}, \eqref{eq:liminfII} and \eqref{eq:liminfI},
\begin{align}\label{eq:liminf_eps}
    &\liminf_{t\to 0^+} \int_D (f(x)-f(y))\frac{p_t^D(x,y)}{t} \,\dy \nonumber \\
    &\geq -\widehat{C}(x,\alpha) \varepsilon^{2-\alpha} + \int_{D\cap \{|x-y|\geq \varepsilon\}} (f(x)-f(y))\nu(x,y) \,\dy.
\end{align}

We will show that there exists the limit 
\begin{align*}
\mathrm{p.v.} \int_D  (f(x) - f(y))\nu(x,y)\,\dy := \lim_{\varepsilon\to 0^+} \int_{D\cap \{|x-y|\geq \varepsilon\}} (f(x) - f(y))\nu(x,y)\,\dy.
\end{align*}
The existence follows from the dominated convergence theorem. Indeed, by the same argument as in the previous part of the proof we know that
\begin{align*}
    &\int_{D\cap \{|x-y|\geq \varepsilon\}} (f(x) - f(y))\nu(x,y)\,\dy \\
    &= \int_{D}  \ind_{\{|x-y|\geq \varepsilon\}} \big(f(x) - f(y) - \ind_{B(x,x/2)}(y) (y-x) f'(x)\big)\nu(x,y)\,\dy.
\end{align*}
Furthermore, from the Taylor's theorem we get
\begin{align*}
    &\ind_{\{|x-y|\geq \varepsilon\}} \big| f(x) - f(y) - \ind_{B(x,x/2)}(y) (y-x) f'(x) \big| \\
    &\leq \ind_{B(x,x/2)^c}(y) \big( |f(x)| + |f(y)| \big)  + \ind_{B(x,x/2)}(y) \sup_{y\in B(x,x/2)} |f''(y)| |y-x|^2 \\
    &=: \ind_{B(x,x/2)^c}(y) \big( |f(x)| + |f(y)| \big)  + \ind_{B(x,x/2)}(y) C(x) |y-x|^2,
\end{align*}
and then
\begin{align*}
    \int_{D\cap B(x,x/2)^c} \frac{|f(x)|+|f(y)|}{|x-y|^{\alpha+1}}\,\dy + C(x) \int_{B(x,x/2)}  |y-x|^{1-\alpha}\,\dy <\infty.
\end{align*}
Note that the finiteness of the above expression follows from the previous part of the proof. Then from the dominated convergence theorem we obtain that
\begin{align*}
    &\lim_{\varepsilon\to 0^+} \int_{D\cap \{|x-y|\geq \varepsilon\}} (f(x) - f(y))\nu(x,y)\,\dy \\
    &= \int_{D} \big(f(x) - f(y) - \ind_{B(x,x/2)}(y) (y-x) f'(x)\big)\nu(x,y)\,\dy,
\end{align*}
and in particular the above limit exists.
Hence, by taking $\varepsilon\to 0^+$ in \eqref{eq:liminf_eps}, we get
\begin{align*}
    \liminf_{t\to 0^+} \int_D (f(x)-f(y))\frac{p_t^D(x,y)}{t} \,\dy \geq \mathrm{p.v.} \int_{D} (f(x) - f(y))\nu(x,y)\,\dy.
\end{align*}
Similarly, from \eqref{eq:I+II}, \eqref{eq:liminfII} and \eqref{eq:limsupI}, we obtain that
\begin{align*}
    \limsup_{t\to 0^+} \int_D (f(x)-f(y))\frac{p_t^D(x,y)}{t} \,\dy \leq \mathrm{p.v.} \int_{D} (f(x)-f(y))\nu(x,y) \,\dy.
\end{align*}
Therefore,
\begin{align*}
    \lim_{t\to 0^+} \int_D (f(x)-f(y))\frac{p_t^D(x,y)}{t} \,\dy = \mathrm{p.v.} \int_{D} (f(x) - f(y))\nu(x,y)\,\dy,
\end{align*}
which ends the proof.
\end{proof}
We note in passing that if the pointwise formula in Lemma~\ref{stable_generator_estimate} gives a $C_0(D)$ function, then $f$ is in the domain of the Dirichlet fractional Laplacian on $D$. This is true more generally, see \cite[Theorem 2.3]{MR3897925}.

In what follows, however, we focus on the functions $h_\beta$.

\begin{corollary}\label{cor:gener_h_beta}
For $x>0$, $\alpha\in (0,2)$, and $\beta\in (-1,\alpha)$, we have
\begin{align}\label{eq:h_beta_gener}
        \lim_{t\to 0^+} \frac{h_\beta(x) - P^D_th_\beta(x)}{t}
        = \mathrm{p.v.} \int_D (h_\beta(x) - h_\beta(y))\nu(x,y)\,\dy + h_\beta(x)\nu(x,D^c).
    \end{align}    
\end{corollary}

\begin{lemma}\label{h_sum_from_2}
    For $\alpha\in (0,2)$ and $\beta(\alpha-\beta-1)\geq0$,
    \[
    \lim_{t\to 0^+} \frac{1}{t} \sum_{n=2}^\infty K_{t,n}h_\beta(x) = 0, \qquad x\neq 0.
    \]
\end{lemma}

\begin{proof}
    First, assume that $x>0$. Then from Corollary \ref{perturbation_formula} it follows that
    \begin{align}\label{eq:perturb_2_times}
    K_t = \P_t + \int_0^t \P_r\widehat{\nu}\P_{t-r}\,\dr + \int_0^t \int_r^{t} \P_r\widehat{\nu}  \P_{s-r}\widehat{\nu} K_{t-s}\,\ds\,\dr,
    \end{align}
    and from Corollary \ref{excesive_function_beta} we have that
    \begin{align*}
    &\frac{1}{t} \sum_{n=2}^\infty K_{t,n}h_\beta(x) \\
    &= \frac1t\int_0^t \int_r^{t} \P_r\widehat{\nu}  \P_{s-r}\widehat{\nu} K_{t-s}h_\beta(x)\,\ds\,\dr \\
    &= \frac1t \int_0^t \dr \int_D \dy \int_{D^c} \dz \int_r^t \ds \int_D \dv \, p_r^D(x,y)\nu(y,z) e^{-\nu(z,D)(s-r)} \nu(z,v) K_{t-s}h_\beta(v) \\
    &\leq \frac1t \int_0^t \dr \int_D \dy \int_{D^c} \dz \int_r^t \ds \int_D \dv \, p_r^D(x,y)\nu(y,z) e^{-\nu(z,D)(s-r)} \nu(z,v)h_\beta(v).
    \end{align*}
    Direct calculations show that for $z<0$,
    \[
    \int_0^\infty \nu(z,v)h_\beta(v)\,\dv \approx  \int_0^\infty \frac{v^\beta}{|z-v|^{\alpha+1}}\,\dv = |z|^{\beta-\alpha} \mathfrak{B}(\beta+1, \alpha-\beta) \approx |z|^{\beta-\alpha}.
    \]
    Therefore, from \eqref{nu_scaling},
    \begin{align*}
        \frac{1}{t} \sum_{n=2}^\infty K_{t,n}h_\beta(x) &\lesssim \frac1t \int_0^t \dr \int_D \dy \int_{D^c} \dz \int_r^t \ds \, p_r^D(x,y)\nu(y,z) e^{-\nu(z,D)(s-r)} |z|^{\beta-\alpha} \\
        &\approx \frac1t \int_0^t \dr \int_D \dy \int_{D^c} \dz \, p_r^D(x,y)\nu(y,z) |z|^{\beta} \big[1 -  e^{-\nu(z,D)(t-r)} \big] \\
        &\lesssim \frac1t \int_0^t \dr \int_D \dy \int_{D^c} \dz \, p_r^D(x,y)\nu(y,z) |z|^{\beta} \big( 1 \wedge t|z|^{-\alpha} \big).      
    \end{align*}
    Using the substitution $r = tu$ we get
    \begin{align*}
        \frac{1}{t} \sum_{n=2}^\infty K_{t,n}h_\beta(x) &\lesssim \int_0^1 \du \int_D \dy \int_{D^c} \dz \, p_{tu}^D(x,y)\nu(y,z) |z|^\beta \big(1\wedge t|z|^{-\alpha}\big),
    \end{align*}
    and by \eqref{eq:approx_p_D},
    \begin{align*}
        &\frac{1}{t} \sum_{n=2}^\infty K_{t,n}h_\beta(x) \\
        &\lesssim \int_0^1 \du \int_D \dy \int_{D^c} \dz \, 
        \Big( 1\wedge \frac{y^{\alpha/2}}{\sqrt{tu}}\Big)\Big((tu)^{-1/\alpha}\wedge \frac{tu}{|x-y|^{\alpha+1}}\Big)\nu(y,z) |z|^\beta \big(1\wedge t|z|^{-\alpha}\big) \\
        &\lesssim \int_0^1 \du \int_{D \cap \{ |x-y|< (tu)^{1/\alpha} \}} \dy \int_{D^c} \dz \, (tu)^{-1/\alpha}\frac{|z|^\beta}{|y-z|^{\alpha+1}}\big(1\wedge t|z|^{-\alpha}\big) \\
        &+ \int_0^1 \du \int_{D \cap \{ |x-y|\geq  (tu)^{1/\alpha} \}} \dy \int_{D^c} \dz \, 
        \Big( 1\wedge \frac{y^{\alpha/2}}{\sqrt{tu}}\Big)\frac{tu}{|x-y|^{\alpha+1}} \frac{|z|^\beta}{|y-z|^{\alpha+1}} \big(1\wedge t|z|^{-\alpha}\big)\\
        &=: I_t + II_t.
    \end{align*}
    If $t$ is sufficiently small compared with $x$ in $I_t$, then $|z-x| \lesssim |z-y|$ and
    \begin{align*}
    I_t &\lesssim \int_0^1 \du \int_{D \cap \{ |x-y|< (tu)^{1/\alpha} \}} \dy \int_{D^c} \dz \, (tu)^{-1/\alpha}\frac{|z|^\beta}{|x-z|^{\alpha+1}}\big(1\wedge t|z|^{-\alpha}\big) \\
    &\lesssim \int_0^1 \du \int_{D^c} \dz \, \frac{|z|^\beta}{|x-z|^{\alpha+1}}\big(1\wedge t|z|^{-\alpha}\big) \\
    &=\int_{D^c}  \frac{|z|^\beta}{|x-z|^{\alpha+1}}\big(1\wedge t|z|^{-\alpha}\big) \,\dz,
    \end{align*}
    and the latter integral converges to $0$ as $t\to 0^+$ from the dominated convergence theorem. Indeed, $\big(1\wedge t|z|^{-\alpha}\big) \leq 1$ and 
    \begin{align*}
    \int_{-\infty}^0 \frac{|z|^\beta}{|x-z|^{\alpha+1}}\,\dz &= x^{\beta -\alpha} \int_0^\infty \frac{w^\beta}{(w+1)^{\alpha+1}}\,\dw = x^{\beta-\alpha} \,\mathfrak{B}(\beta+1, \alpha-\beta) <\infty.
    \end{align*}
    It remains to show that $II_t\to 0$ as $t\to 0^+$. We will show that in a few steps. 

    The integrand in $II_t$ is non-negative, so from the Tonelli's theorem, we can change the order of integrals. Hence, for now, we consider only integral over $u$. We have
    \begin{align*}
        \int_0^1 \ind_{\{ |x-y| \geq (tu)^{1/\alpha}\}}\Big( 1\wedge \frac{y^{\alpha/2}}{\sqrt{tu}}\Big) u\,\du = \int_0^{1\wedge (|x-y|^\alpha/t)} \Big( 1\wedge \frac{y^{\alpha/2}}{\sqrt{tu}}\Big) u\,\du.
    \end{align*}
    Here we consider two cases. If $y\geq x/2$ then $y \geq |x-y|$ and hence $\frac{y^\alpha}{t} \geq 1\wedge \frac{|x-y|^\alpha}{t} \geq u$. This means that $\frac{y^{\alpha/2}}{\sqrt{tu}} \geq 1$, so
    \begin{align*}
        \int_0^{1\wedge (|x-y|^\alpha/t)} \Big( 1\wedge \frac{y^{\alpha/2}}{\sqrt{tu}}\Big) u\,\du = \int_0^{1\wedge (|x-y|^\alpha/t)} u\,\du = \frac{1}{2} \Big( 1 \wedge \frac{|x-y|^\alpha}{t}\Big)^2.
    \end{align*}
    If $y\in (0,x/2)$ then similarly we have that $\frac{y^\alpha}{t} < \frac{|x-y|^\alpha}{t}$. Hence, for sufficiently small $t$,
    \begin{align*}
        \int_0^{1\wedge (|x-y|^\alpha/t)} \Big( 1\wedge \frac{y^{\alpha/2}}{\sqrt{tu}}\Big) u\,\du &\leq \int_0^{1\wedge (y^\alpha/t)} u\,\du + \int_{1\wedge (y^\alpha/t)}^{1\wedge (|x-y|^\alpha/t)} \frac{y^{\alpha/2}}{\sqrt{t}} \sqrt{u}\,\du \\
        &=\frac{1}{2} \Big( 1\wedge \frac{y^\alpha}{t}\Big)^2 + \frac{2}{3} \frac{y^{\alpha/2}}{\sqrt{t}} \Big[ \Big(1\wedge \frac{|x-y|^\alpha}{t}\Big)^{3/2} - \Big( 1\wedge \frac{y^\alpha}{t}\Big)^{3/2}\Big] \\
        &\lesssim \Big( 1\wedge \frac{y^\alpha}{t}\Big)^2 + \Big(\frac{y^\alpha}{t}\Big)^{1/2} \\
        &\lesssim \frac{y^{\alpha/2}}{\sqrt{t}}.
    \end{align*}
    Now we consider again the integral $II_t$. Using previous observations, we get
    \begin{align*}
        II_t &= \int_0^1 \du \int_{D} \dy \int_{D^c} \dz \, 
        \ind_{\{ |x-y|\geq  (tu)^{1/\alpha} \}}
        \Big( 1\wedge \frac{y^{\alpha/2}}{\sqrt{tu}}\Big)\frac{tu}{|x-y|^{\alpha+1}} \frac{|z|^\beta}{|y-z|^{\alpha+1}} \big(1\wedge t|z|^{-\alpha}\big) \\
        &\lesssim \int_0^{x/2} \dy \int_{D^c} \dz \, \frac{y^{\alpha/2}}{\sqrt{t}} \frac{t}{|x-y|^{\alpha+1}} \frac{|z|^\beta}{|y-z|^{\alpha+1}} \big(1\wedge t|z|^{-\alpha}\big)  \\
        &+ \int_{x/2}^\infty \dy \int_{D^c}\dz \, \Big( 1 \wedge \frac{|x-y|^\alpha}{t}\Big)^2\frac{t}{|x-y|^{\alpha+1}} \frac{|z|^\beta}{|y-z|^{\alpha+1}} \big(1\wedge t|z|^{-\alpha}\big) \\
        &=: A_t + B_t.
    \end{align*}
    For $y\in (0,x/2)$, it is obvious that $|x-y| \approx x$ and then
    \begin{align*}
        A_t &\leq \sqrt{t} \int_0^{x/2} \dy \int_{D^c} \dz \, \frac{y^{\alpha/2}}{|x-y|^{\alpha+1}} \frac{|z|^\beta}{|y-z|^{\alpha+1}} \\
        &\approx  x^{-\alpha-1} \sqrt{t} \int_0^{x/2} \dy \int_{D^c} \dz \, y^{\alpha/2}\frac{|z|^\beta}{|y-z|^{\alpha+1}} \\
        &\approx x^{-\alpha-1} \sqrt{t} \int_0^{x/2} y^{\beta-\alpha/2}\,\dy \\
        &\lesssim  x^{\beta-3\alpha/2}\sqrt{t} \to 0,
    \end{align*}
    as $t\to 0^+$. 
    
    Since, for $y>x/2$ and $z<0$, we have $|y-z| > |x/2-z| = \tfrac12 |x-2z|$, then by substitution $w=2z$ we get
    \begin{align*}
        B_t &=  \int_{x/2}^\infty \dy \int_{D^c}\dz \, \Big( 1 \wedge \frac{|x-y|^\alpha}{t}\Big)^2\frac{t}{|x-y|^{\alpha+1}} \frac{|z|^\beta}{|y-z|^{\alpha+1}} \big(1\wedge t|z|^{-\alpha}\big) \\
        &\lesssim  \int_{|x-y| \leq t^{1/\alpha}} \dy \int_{D^c}\dw \, \frac{|x-y|^{\alpha-1}}{t}\frac{|w|^\beta}{|x-w|^{\alpha+1}} \big(1\wedge t|w|^{-\alpha}\big) \\
        &+ \int_{\{y\geq x/2\}\cap\{|x-y|> t^{1/\alpha}\}} \dy \int_{D^c} \dw \frac{t}{|x-y|^{\alpha+1}} \frac{|w|^\beta}{|x-w|^{\alpha+1}} \big(1\wedge t|w|^{-\alpha}\big) \\
        &=: C_t + D_t.
    \end{align*}
    Note that from Tonelli's theorem we get
    \begin{align*}
        C_t &\lesssim \int_{D^c} \frac{|w|^\beta}{|x-w|^{\alpha+1}} \big(1\wedge t|w|^{-\alpha}\big)\,\dw \to 0,
    \end{align*}
    as $t\to 0^+$, which follows in the same way as in the case of $I_t$. For the integral $D_t$, again from the Tonelli's theorem,  we have 
    \begin{align*}
        D_t &= \int_{\{y\geq x/2\}\cap\{|x-y|> t^{1/\alpha}\}} \dy \int_{D^c} \dw \,\frac{t}{|x-y|^{\alpha+1}} \frac{|w|^\beta}{|x-w|^{\alpha+1}} \big(1\wedge t|w|^{-\alpha}\big)  \\
        &\leq \int_{D^c} \frac{t|w|^\beta}{|x-w|^{\alpha+1}} \big(1\wedge t|w|^{-\alpha}\big) \int_{\{|x-y|> t^{1/\alpha}\}}   \frac{\dy}{|x-y|^{\alpha+1}} \,\dw
        \\
        &= \int_{D^c} \frac{t|w|^\beta}{|x-w|^{\alpha+1}} \big(1\wedge t|w|^{-\alpha}\big) \int_{\{|u|> t^{1/\alpha}\}}   \frac{\du}{|u|^{\alpha+1}} \,\dw
        \\
        &=\frac{2}{\alpha} \int_{D^c}  \frac{|w|^\beta}{|x-w|^{\alpha+1}} \big(1\wedge t|w|^{-\alpha}\big)\,\dw \to 0,
    \end{align*}
    as $t\to 0^+$. The latter convergence follows from the dominated convergence theorem.

\medskip
    Now let $x<0$. From Lemma \ref{proposition_recurrence} we have the following equality:
    \begin{align*}
        \frac{1}{t}\sum_{n=2}^\infty K_{t,n}h_\beta(x) &= \frac{1}{t}\sum_{n=1}^\infty \int_0^t \P_r \widehat{\nu} K_{t-r,n}h_\beta(x)\,\dr \\
        &= \frac{1}{t} \int_0^t \,\dr \int_D \dz \, e^{-\nu(x,D)r} \nu(x,z) \sum_{n=1}^\infty K_{t-r,n}h_\beta(z) \\
        &= \int_0^1 \,\dr \int_D \dz \, e^{-\nu(x,D)tr} \nu(x,z) \sum_{n=1}^\infty K_{t(1-r),n}h_\beta(z) \\
        &\leq \int_0^1 \,\dr \int_D \dz \, \nu(x,z) \sum_{n=1}^\infty K_{t(1-r),n}h_\beta(z).
    \end{align*}
    From the first part of the proof, we know that $\lim\limits_{t\to 0^+}\tfrac{1}{t}\sum_{n=2}^\infty K_{t,n}h_\beta(z) = 0$, where $z>0$. Hence, of course,  $\lim\limits_{t\to 0^+} \sum_{n=2}^\infty K_{t,n}h_\beta(z) = 0$. Moreover, from Corollary \ref{cor:K_t_1_h_beta} it follows that $\lim\limits_{t\to 0^+} K_{t,1}h_\beta(z) = 0$. Combining this results, for $z>0$, we get $\lim\limits_{t\to 0^+} \sum_{n=1}^\infty K_{t,n}h_\beta(z) = 0$. From Corollary \ref{excesive_function_beta} it follows that
    $\sum_{n=1}^\infty K_{t(1-r),n}h_\beta(z) \leq K_{t(1-r)}h_\beta(z) \leq h_\beta(z)$ and we observe that 
    \begin{align*}
        \int_0^1 \,\dr \int_D \dz \, \nu(x,z) h_\beta(z) \approx \int_0^\infty \frac{z^\beta}{|x-z|^{\alpha+1}}\,\dz = x^{\beta-\alpha} \mathfrak{B}(\beta+1, \alpha-\beta) < \infty.
    \end{align*}
    Therefore, by the dominated convergence theorem,
    \[
    \int_0^1 \,\dr \int_D \dz \, \nu(x,z) \sum_{n=1}^\infty K_{t(1-r),n}h_\beta(z)\to 0,
    \]
    as $t\to 0^+$. 
\end{proof}

\begin{proof}[Proof of Proposition \ref{gen_inequality}]
Let $\mathcal{G}(t,x) := (h_\beta(x) - K_th_\beta(x))/t$ and assume that $x>0$. Then from Corollary \ref{cor:gener_h_beta},
\begin{align*}
    \lim_{t\to 0^+} \mathcal{G}(t,x) &= \lim_{t\to 0^+} \frac{h_\beta(x) - \P_th_\beta(x)}{t} - \lim_{t\to 0^+} \frac{1}{t}\sum_{n=1}^\infty K_{t,n}h_\beta(x) \\
    &=\mathrm{p.v.} \int_D (h_\beta(x) - h_\beta(y))\nu(x,y)\,\dy + h_\beta(x)\nu(x,D^c) - \lim_{t\to 0^+} \frac{1}{t}\sum_{n=1}^\infty K_{t,n}h_\beta(x).
\end{align*}
From Corollary \ref{cor:K_t_1_h_beta} and from Lemma \ref{h_sum_from_2} we have
\begin{align*}
    \lim_{t\to 0^+} \frac{1}{t} \sum_{n=1}^\infty K_{t,n}h_\beta(x) = \int_{D^c} h_\beta(y)\nu(x,y)\,\dy.
\end{align*}
Therefore,
\begin{align}\label{eq:gen_0}
    \lim_{t\to 0^+} \mathcal{G}(t,x) &= \mathrm{p.v.} \int_D (h_\beta(x) - h_\beta(y))\nu(x,y)\,\dy + \int_{D^c} (h_\beta(x) - h_\beta(y))\nu(x,y)\,\dy \\
    &= \mathrm{p.v.} \int_{\R} (h_\beta(x) - h_\beta(y))\nu(x,y)\,\dy.\nonumber
\end{align}
Now we will calculate the exact values of the above integrals. Using the substitution $y = |x|z$ we obtain,
\begin{align*}
    \mathrm{p.v.} \int_{D} (h_\beta(x)-h_\beta(y))\nu(x,y)\,\dy = \mathcal{A}_{1,\alpha} |x|^{\beta - \alpha}\,\mathrm{p.v.} \int_0^\infty \frac{1-z^\beta}{|1-z|^{\alpha+1}}\,\dz.
\end{align*}
From Section 5 of \cite{MR2006232} we have 
\begin{align*}
     \mathrm{p.v.} \int_0^\infty \frac{1-z^\beta}{|1-z|^{\alpha+1}}\,\dz = - \gamma(\alpha,\beta) \geq 0.
\end{align*}
Hence,
\begin{align}\label{eq:gen_1}
    \mathrm{p.v.} \int_{D} (h_\beta(x)-h_\beta(y))\nu(x,y)\,\dy = -\mathcal{A}_{1,\alpha} |x|^{\beta - \alpha} \gamma(\alpha,\beta). 
\end{align}
Moreover, from \eqref{nuDc_scaling},
\begin{align}\label{eq:gen_2}
    \int_{D^c} (h_\beta(x) - h_\beta(y))\nu(x,y)\,\dy &= h_\beta(x)\nu(x,D^c) - \mathcal{A}_{1,\alpha} \int_{-\infty}^0 \frac{|y|^\beta}{|x-y|^{\alpha+1}}\,\dy  \nonumber\\
    &= \mathcal{A}_{1,\alpha} |x|^{\beta -\alpha} \big[ \alpha^{-1} - \mathfrak{B}(\beta+1, \alpha-\beta)\big].
\end{align}
From \eqref{eq:gen_0}, \eqref{eq:gen_1} and \eqref{eq:gen_2} we get
\[
\lim_{t\to 0^+} \mathcal{G}(t,x) = \mathcal{A}_{1,\alpha} |x|^{\beta-\alpha} \big[\alpha^{-1} - \mathfrak{B}(\beta+1, \alpha-\beta) -\gamma(\alpha,\beta) \big].
\]

\medskip
Now consider the case $x<0$. Then,
\[
\mathcal{G}(t,x) = h_\beta(x)\frac{1 - e^{-\nu(x,D)t}}{t} - \frac{1}{t} K_{t,1}h_\beta(x) - \frac{1}{t}\sum_{n=2}^\infty K_{t,n}h_\beta(x).
\]
From Corollary \ref{cor:K_t_1_h_beta} and Lemma \ref{h_sum_from_2} we obtain that
\begin{align}\label{x_neg_1}
    \lim_{t\to 0^+} \mathcal{G}(t,x) = \int_D (h_\beta(x) - h_\beta(y))\nu(x,y)\,\dy.
\end{align}
Moreover, from \eqref{nu_scaling},
\begin{align}\label{}
    \int_{D} (h_\beta(x) - h_\beta(y))\nu(x,y)\,\dy &= h_\beta(x)\nu(x,D) - \mathcal{A}_{1,\alpha} \int_0^\infty \frac{y^\beta}{|x-y|^{\alpha+1}} \,\dy \nonumber\\
    &= \mathcal{A}_{1,\alpha} |x|^{\beta -\alpha} \big[ \alpha^{-1} - \mathfrak{B}(\beta+1, \alpha-\beta)\big].
\end{align}
Hence,
\[
\lim_{t\to 0^+} \mathcal{G}(t,x) = \mathcal{A}_{1,\alpha} |x|^{\beta -\alpha} \big[ \alpha^{-1} - \mathfrak{B}(\beta+1, \alpha-\beta)\big].
\]

\medskip
Now, it remains to show that $\mathcal{C}(\alpha,\beta,x)>0$ for $\beta(\alpha-\beta-1)>0$. From \cite[(5.3)]{MR2006232} it follows that $\gamma(\alpha,\beta)<0$. Hence, it suffices to show that for $\beta(\alpha-\beta-1)>0$, $\mathcal{M}(\alpha,\beta) := \alpha^{-1} - \mathfrak{B}(\beta+1, \alpha-\beta)>0$.
From the properties of the Beta function we obtain that
\begin{align*}
    \mathcal{M}(\alpha,\beta) = \frac{1}{\Gamma(\alpha+1)} \big[ \Gamma(\alpha) - \Gamma(\beta+1)\Gamma(\alpha-\beta)\big].
\end{align*}
We will consider two cases. 

Assume that $1<\alpha<2$ and $0<\beta<\alpha-1$ and define a function $(0,\alpha-1)\ni \beta\mapsto F(\beta) := \Gamma(\alpha) - \Gamma(\beta+1)\Gamma(\alpha-\beta)$. Then 
\[
F'(\beta) = \Gamma(\beta+1)\Gamma(\alpha-\beta) \big[ \psi(\alpha-\beta) - \psi(\beta+1)\big],
\]
where $\psi$ denotes the digamma function. From the fact that the digamma function is increasing on $(0,\infty)$, we obtain that for $\beta \in (0, (\alpha-1)/2)$, $F'(\beta) > 0$ and for $\beta \in ( (\alpha-1)/2,\alpha-1)$, $F'(\beta) < 0$. Hence, for $\beta = (\alpha-1)/2$ the function $F$ has the global maximum on the interval $(0,\alpha-1)$.
Moreover, for $\beta\in (0,\alpha-1)$, $F(\beta) > F(0) = F(\alpha-1) = 0$. Thus, $\mathcal{M}(\alpha,\beta)>0$.

Now, assume that $0<\alpha<1$ and $\alpha-1<\beta<0$. Using the same notation as in the previous case we get that for $\beta \in (\alpha-1, (\alpha-1)/2)$, $F'(\beta)>0$ and for $\beta\in ((\alpha-1)/2,0)$, $F'(\beta)<0$. And then we conclude that again for $\beta = (\alpha-1)/2$ the function $F$ has the global maximum on the interval $(\alpha-1,0)$ and moreover, $F(\beta)>0$ for $\beta\in (\alpha-1,0)$. Hence, $\mathcal{M}(\alpha,\beta)>0$.
\end{proof}

\begin{corollary}\label{c.lpg}
If $f\in C^2(\R_*)$ and $|f(x)|\lesssim \left(1+|x|^{\alpha-1}\right)$ on $\R_*$, then
\[
 \lim_{t\to 0^+} \frac{f(x) - K_t f(x)}{t} = \mathrm{p.v.} \int_{\R} (f(x) - f(y))\nu(x,y)\, \dy, \quad x>0,
\]
and
\[
 \lim_{t\to 0^+} \frac{f(x) - K_t f(x)}{t} = \int_D (f(x) - f(y)) \nu(x,y)\, \dy, \quad x<0.
\]
\end{corollary}
\begin{proof}
In view of Lemma~\ref{h_sum_from_2} and the proof of Proposition~\ref{gen_inequality}, the result holds for $C^2$ functions $f$ majorized by the functions $h_0$ or $h_{\alpha-1}$, in fact, by their sum. 
\end{proof}

\section{Function spaces and the Dirichlet form}\label{chap_dirichlet}

In this section, we prove the Hardy inequality for $K$ in case $\alpha\neq 1$ and investigate the Dirichlet form $\E$ corresponding to the process $X$. We also prove various characterizations of the domain of the form $\E$.

\subsection{Dirichlet form}

From Section \ref{sec:semigroup} we know that  $(K_t)_{t\geq 0}$ is symmetric and strongly continuous contraction semigroup on $L^2(\R)$. Following  \cite[p. 3]{MR2849840} or \cite[p. 23]{MR2778606}, for $t>0$,
\[
\mathcal{E}^{(t)}(u,v) := \frac{1}{t}\langle u-K_tu,v\rangle_{L^2(\R)}, \qquad u,v\in L^2(\R),
\]
is a symmetric form on $L^2(\R)$. From general theory (see, e.g.,  \cite{MR2849840}, \cite{MR2778606}) it turns out that for $u\in L^2(\R)$, $\mathcal{E}^{(t)}(u,u)$ is non-negative. Moreover, if $t>0$ decreases, then $\mathcal{E}^{(t)}(u,u)$ is increasing (it follows from the spectral representation of $(K_t)_{t\geq 0}$). Therefore, the limit $\lim\limits_{t\to 0^+} \mathcal{E}^{(t)}(u,u) = \sup\limits_{t>0} \mathcal{E}^{(t)}(u,u)$ exists, and we may then set
\begin{align*}
    \F := \{u\in L^2(\R): ~ \lim_{t\to 0^+} \mathcal{E}^{(t)}(u,u)<\infty\},
\end{align*}
\begin{align}\label{eq:E_definition}
    \E(u,v) := \lim_{t\to 0^+} \mathcal{E}^{(t)}(u,v), \qquad u,v\in \F.
\end{align}
Then $\mathcal{E}$ becomes a closed symmetric form on $L^2(\R)$ corresponding to the semigroup $(K_t)_{t\geq 0}$ (see \cite[Chapter 1, p. 3]{MR2849840}). Moreover, from Lemma \ref{K_subprobability} it follows that for each $t\geq 0$ the operator $K_t$ is \emph{Markovian}, i.e. for $u\in L^2(\R)$, $0\leq u\leq 1$ we have $0\leq K_tu\leq 1$. Hence, from \cite[Theorem 1.4.1]{MR2778606}, the form $\mathcal{E}$ is \emph{Markovian} (for the definition of Markovian property of symmetric forms see, e.g.,  \cite[p. 4]{MR2778606}). Therefore, the form $(\mathcal{E}, \F)$ corresponding to the semigroup $(K_t)_{t\geq 0}$ is in fact a \emph{Dirichlet form}. 

In what follows we will use the following abbreviation: $\E[u] := \E(u,u)$, $u\in \F$. We want to extend this definition to the whole space $L^2(\R)$, by setting $\E[v] = \infty$ if $v\notin\F$. This convention will be useful in the formulation of the Hardy inequality for the form $\E$.

From the general properties of the Dirichlet forms we can prove the following (see proof of the Theorem 1.4.2 in \cite{MR2778606}): if $u_1, u_2\in \F$, $w\in L^2(\R)$ satisfy 
\[
|w(x) - w(y)| \leq |u_1(x) - u_1(y)| + |u_2(x)-u_2(y)|
\]
and
\[
|w(x)| \leq |u_1(x)| + |u_2(x)|
\]
almost everywhere, then $w\in\F$ and
\begin{align}
    \label{eq:E_triangle_general}
    \sqrt{\E[w]} \leq \sqrt{\E[u_1]} + \sqrt{\E[u_2]}.
\end{align}
By taking $u_1, u_2\in\F$ and $w:= u_1+u_2 \in L^2(\R)$, from \eqref{eq:E_triangle_general}, we obtain the following triangle inequality for $\E^{1/2}$:
\begin{align}
    \label{eq:E_triangle}
    \sqrt{\E[u_1+u_2]} \leq \sqrt{\E[u_1]} + \sqrt{\E[u_2]}, \qquad u_1, u_2\in\F.
\end{align}
Hence, from the definition of the form $\E$ and from the property \eqref{eq:E_triangle}, it follows that in fact $\sqrt{\E[\cdot]}$ is a seminorm on $\F$. From \eqref{eq:E_triangle} it follows also that $\F$ is in fact a linear subspace of $L^2(\R)$.

On the space $\F$ we define an inner product $\langle \cdot,\cdot\rangle_\F$ and a norm $\norm{\cdot}_\F$ by the following expressions:
\[
\langle u,v\rangle_\F := \langle u,v\rangle_{L^2(\R)} + \E(u,v), \qquad u,v\in \F,
\]
\begin{align}\label{eq:F_norm}
\norm{u}_\F^2 := \norm{u}_{L^2(\R)}^2 + \E[u], \qquad u\in \F.
\end{align}
The fact that $\langle\cdot,\cdot\rangle_{\F}$ is an inner product is an elementary exercise, and we will not present its details. Moreover, let us observe that the norm $\norm{\cdot}_\F$ is norm induced by the inner product $\langle\cdot,\cdot\rangle_\F$.
Then from the fact that $\E$ is a Dirichlet form, $(\F, \norm{\cdot}_\F)$ becomes a Hilbert space.

\subsection{Proof of Theorem~\ref{nier_Hardyego}}\label{sec:dnH}

We will use the method of Bogdan, Dyda, Kim \cite{MR3460023}.

    Let $h_\beta(x) = |x|^\beta$ for $\beta(\alpha-\beta-1)>0$ and $v := u/h_\beta$, with the convention that for $h_\beta(x) = 0$ or $\infty$, we have $v(x) = 0$. 
    Of course $vh_\beta\in L^2(\R)$, because $|vh_\beta| \leq |u|$. Moreover, from Corollary \ref{excesive_function_beta} we know that $vK_th_\beta \in L^2(\R)$. Therefore, for $t>0$ we have
    \begin{align*}
        \mathcal{E}^{(t)}[vh_\beta] = \frac{1}{t} \langle v(h_\beta-K_th_\beta), vh_\beta \rangle_{L^2(\R)} + \frac{1}{t} \langle vK_th_\beta - K_t(vh_\beta), vh_\beta\rangle_{L^2(\R)} =: A_t + B_t.
    \end{align*}
    
Note that
\[
    B_t = \frac{1}{t} \int_\R \dx\int_\R K_t(x,\dy) \,v(x)h_\beta(x)h_\beta(y)(v(x) - v(y))
\]
Moreover, using Lemma \ref{symmetry_of_K}, we obtain
\begin{align*}
    B_t & = \frac{1}{t} \langle vK_th_\beta, vh_\beta\rangle_{L^2(\R)} 
    - \frac{1}{t} \langle  K_t(vh_\beta), vh_\beta\rangle_{L^2(\R)}\\
    &=\frac{1}{t} \int_\R \dx\int_\R K_t(x,\dy) \, v^2(x)h_\beta(x)h_\beta(y) - \frac{1}{t} \int_\R \dx\int_\R K_t(x,\dy) \, v(x)h_\beta(x)v(y)h_\beta(y) \\
    &= \frac{1}{t} \int_\R \dx\int_\R K_t(x,\dy) \, v^2(y)h_\beta(y)h_\beta(x)- \frac{1}{t} \int_\R \dx\int_\R K_t(x,\dy) \, v(x)h_\beta(x)v(y)h_\beta(y) \\
    &= \frac{1}{t} \int_\R \dx\int_\R K_t(x,\dy) \, v(y)h_\beta(y)h_\beta(x)(v(y) - v(x)).
\end{align*}
From the above equalities, we get
\begin{align*}
    B_t =\frac{1}{2t} \int_\R \dx\int_\R K_t(x,\dy) \, h_\beta(x)h_\beta(y) (v(x) - v(y))^2 \geq 0,
\end{align*}
hence
\begin{align*}
    \mathcal{E}^{(t)}[vh_\beta] \geq A_t = \int_\R v^2(x)h_\beta(x) \frac{h_\beta(x) - K_th_\beta(x)}{t}\,\dx.
\end{align*}
From Corollary \ref{excesive_function_beta}, the integrand is non-negative and then from Fatou's lemma  and Proposition \ref{gen_inequality} we have
\begin{align*}
    \mathcal{E}[vh_\beta] &\geq \int_\R v^2(x)h_\beta(x) \liminf_{t\to 0^+} \frac{h_\beta(x) - K_th_\beta(x)}{t}\,\dx \\
    &= \mathcal{A}_{1,\alpha} \int_\R v^2(x)h_\beta^2(x)  \mathcal{C}(\alpha,\beta,x) |x|^{-\alpha}\,\dx.
\end{align*}
Thus,
\begin{align}\label{eq:Hardy_with_beta}
    \mathcal{E}[u] &\geq (\mathcal{C}_{\alpha,\beta} + \mathcal{D}_{\alpha,\beta}) \int_D u^2(x) |x|^{-\alpha}\,\dx + \mathcal{C}_{\alpha,\beta} \int_{D^c} u^2(x) |x|^{-\alpha}\,\dx,
\end{align}
where
\[
\mathcal{C}_{\alpha,\beta} = \mathcal{A}_{1,\alpha} \big[\alpha^{-1} - \mathfrak{B}(\beta+1, \alpha-\beta)\big], \qquad 
\mathcal{D}_{\alpha,\beta} = -\mathcal{A}_{1,\alpha}\gamma(\alpha,\beta).
\]
The fact that $\mathcal{C}_{\alpha,\beta}>0$ and $\mathcal{D}_{\alpha,\beta}> 0$ follows also from Proposition \ref{gen_inequality}.

\medskip
Now, we will show that the constants $\mathcal{C}_{\alpha,\beta}$ and $\mathcal{D}_{\alpha,\beta}$ have the biggest values for $\beta = (\alpha-1)/2$. Recall that 
\[
\mathcal{D}_{\alpha,\beta} = -\mathcal{A}_{1,\alpha} \int_0^1 \frac{(t^\beta-1)(1-t^{\alpha-\beta-1})}{(1-t)^{\alpha+1}}\,\dt,
\]
and define the function 
\[
u(\beta,t) = \frac{(t^\beta-1)(1-t^{\alpha-\beta-1})}{(1-t)^{\alpha+1}}, 
\]
where $t\in (0,1)$ and $\beta(\alpha-\beta-1)>0$. We will calculate the derivative
\[
\frac{\partial}{\partial\beta} \mathcal{D}_{\alpha,\beta} = -\mathcal{A}_{1,\alpha} \frac{\partial}{\partial\beta} \int_0^1 u(\beta,t)\,\dt.
\]
We want to use \cite[Theorem 11.5]{Schilling2} to change the order of derivative and integral.
Obviously, the function $(0,1)\ni t\mapsto u(\beta,t)$ is integrable for each considered $\beta$ and the function $\beta\mapsto u(\beta,t)$ is differentiable for each $t\in (0,1)$. Moreover, note that for $t\in (0,1)$,
\begin{align}\label{eq:pochodna_po_beta}
    \left|\frac{\partial}{\partial\beta} u(\beta,t)\right| = \frac{|\ln{t}|}{(1-t)^{\alpha+1}} \big|t^\beta - t^{\alpha-\beta-1}\big|.
\end{align}

\medskip
We will show that, for $t\in (0,1)$ and $\beta(\alpha-\beta-1)>0$,
\begin{align}\label{eq:bez_beta}
    \big|t^\beta - t^{\alpha-\beta-1}\big| \leq \big|1-t^{\alpha-1}\big|.
\end{align}
We consider two cases. First, assume that $1<\alpha<2$ and $0<\beta<\alpha-1$. Then for $\beta\in (0,\tfrac{\alpha-1}{2})$, $t^{\alpha-1} \leq t^{\alpha-1-2\beta}$ and
\[
\big|1-t^{\alpha-1}\big| = 1-t^{\alpha-1} \geq 1-t^{\alpha-1-2\beta} \geq t^\beta \big(1-t^{\alpha-1-2\beta}\big) = t^\beta - t^{\alpha-\beta-1} = \big| t^\beta - t^{\alpha-\beta-1}\big|.
\]
Similarly, for $\beta\in [\tfrac{\alpha-1}{2},\alpha-1)$, $t^{\alpha-1} \leq t^{2\beta - \alpha+1}$ and
\[
\big| 1-t^{\alpha-1}\big| = 1-t^{\alpha-1} \geq 1-t^{2\beta - \alpha+1} \geq t^{\alpha-\beta-1} \big( 1-t^{2\beta-\alpha+1}\big) = t^{\alpha-\beta-1} - t^{\beta} = \big| t^\beta - t^{\alpha-\beta-1}\big|.
\]

Now assume that $0<\alpha<1$ and $\alpha-1<\beta<0$. Note that \eqref{eq:bez_beta} is equivalent to the inequality
\begin{align}\label{eq:nier_a<1}
\big| t^{-\beta} - t^{(2-\alpha) + \beta-1}\big| \leq \big| 1- t^{(2-\alpha) - 1} \big|.
\end{align}
Let $\gamma := 2-\alpha \in (1,2)$ and $\delta := -\beta \in (0, \gamma-1)$ and from the first case we get the following inequality
\[
\big| t^\delta - t^{\gamma-\delta-1}\big| \leq \big| 1 - t^{\gamma-1}\big|,
\]
which is equivalent to the inequality \eqref{eq:nier_a<1}. Thus, we have proved \eqref{eq:bez_beta} in the second case.

\medskip
From \eqref{eq:pochodna_po_beta} and \eqref{eq:bez_beta} we have
\begin{align}
    \left|\frac{\partial}{\partial\beta} u(\beta,t)\right| \leq \frac{|\ln{t}|}{(1-t)^{\alpha+1}} |1-t^{\alpha-1}|.
\end{align}
Moreover,
\begin{align}\label{eq:Hardy_skonczonosc}
    I := \int_0^1 \frac{|\ln{t}|}{(1-t)^{\alpha+1}} |1-t^{\alpha-1}|\,\dt <\infty.
\end{align}
We will show that again in two cases.

Assume that $\alpha\in (0,1)$. Note that for $t\in (1/2,1)$ we have $|\ln{t}|\approx 1-t$ and $t^{\alpha-1} - 1 \leq t^{-1}-1 = (1-t)/t$. Then
\begin{align*}
    I &= \int_0^{1/2} \frac{|\ln{t}|}{(1-t)^{\alpha+1}} (t^{\alpha-1}-1)\,\dt + \int_{1/2}^1 \frac{|\ln{t}|}{(1-t)^{\alpha+1}} (t^{\alpha-1}-1)\,\dt \\
    &\lesssim \int_0^{1/2} |\ln{t}| \, t^{\alpha-1}\,\dt + \int_{1/2}^1 \frac{\dt}{t (1-t)^{\alpha-1}} \\
    &\lesssim \int_0^{1/2} |\ln{t}| \, t^{\alpha-1}\,\dt + \int_{1/2}^1 \frac{\dt}{(1-t)^{\alpha-1}} <\infty.
\end{align*}
Now let $\alpha\in (1,2)$. Then for $t\in (0,1)$, $1-t^{\alpha-1} \leq 1-t$. Hence,
\begin{align*}
    I &= \int_0^{1/2} \frac{|\ln{t}|}{(1-t)^{\alpha+1}} (1-t^{\alpha-1})\,\dt  + \int_{1/2}^1 \frac{|\ln{t}|}{(1-t)^{\alpha+1}} (1-t^{\alpha-1})\,\dt  \\
    &\lesssim \int_0^{1/2} |\ln{t}|\,\dt  + \int_{1/2}^1 \frac{|\ln{t}|}{(1-t)^{\alpha}} \,\dt \\
    &\approx \int_0^{1/2} |\ln{t}|\,\dt  + \int_{1/2}^1 \frac{\dt}{(1-t)^{\alpha-1}} < \infty.
\end{align*}
Thus we have proved \eqref{eq:Hardy_skonczonosc}.

Therefore, from \cite[Theorem 11.5]{Schilling2} we obtain that
\[
\frac{\partial}{\partial\beta} \mathcal{D}_{\alpha,\beta} = \mathcal{A}_{1,\alpha} \int_0^1 \frac{\ln{t}}{(1-t)^{\alpha+1}} \big[t^{\alpha-\beta-1} - t^\beta\big]\,\dt.
\]
Note that for $\beta < (\alpha-1)/2$, $\frac{\mathrm{d}}{\mathrm{d}\beta} \mathcal{D}_{\alpha,\beta} > 0$ and for $\beta > (\alpha-1)/2$, $\frac{\mathrm{d}}{\mathrm{d}\beta} \mathcal{D}_{\alpha,\beta} <0$. Hence, $\mathcal{D}_{\alpha,\beta}$ has the biggest value for $\beta = (\alpha-1)/2$. Moreover, from the proof of Proposition \ref{gen_inequality}, it follows that the constant $\mathcal{C}_{\alpha,\beta}$ has the biggest value also for $\beta = (\alpha-1)/2$.

Thus, by taking the maximum over $\beta \in (0,\alpha-1)$ in \eqref{eq:Hardy_with_beta} we get the desired inequality:
\[
    \mathcal{E}[u] \geq (\mathcal{C}_{\alpha, (\alpha-1)/2} + \mathcal{D}_{\alpha,(\alpha-1)/2}) \int_D u^2(x) |x|^{-\alpha}\,\dx + \mathcal{C}_{\alpha,(\alpha-1)/2} \int_{D^c} u^2(x) |x|^{-\alpha}\,\dx. \qedhere
\]
The proof is complete. \hfill \qed

\medskip
In case $\alpha = 1$ the above proof is also valid, but we  obtain only trivial inequality (with constant $0$), which cannot be improved because of \cite[(F2) with $D=\mathbb R^*$]{MR2085428}.

From Theorem \ref{nier_Hardyego} we have obvious corollary.

\begin{corollary}\label{cor_Hardy}
    For $u\in L^2(\R)$ and $\alpha\in (0,1)\cup (1,2)$,
    \begin{align*}
        \int_\R u^2(x)|x|^{-\alpha}\,\dx \lesssim \E[u].
    \end{align*}
\end{corollary}

\subsection{The characterization of a domain of the Dirichlet form}

Recall that $D = (0,\infty)$ and the form $\E_D$ on $L^2(\R)$ with its natural domain $\D(\E_D)$ are defined as follows:
\begin{align*}
    \E_D(u,v) := \frac{1}{2} \iint_{\R\times\R\setminus D^c\times D^c} (u(x)-u(y))(v(x)-v(y))\nu(x,y)\,\dx\,\dy, \qquad u,v\in L^2(\R),
\end{align*}
and 
\begin{align*}
    \D(\E_D) := \{u\in L^2(\R):~ \E_D(u,u)<\infty\}.
\end{align*}
Similarly as in the case of $\mathcal{E},$ we denote $\E_D[u] := \E_D(u,u)$. Note that the set $\R\times\R\setminus D^c\times D^c$ in the definition of $\E_D$ is equal to the sum $(D\times D) \cup (D\times D^c) \cup (D^c\times D)$.

Let $u,v\in\D(\E_D)$. From the inequality $(a-b)^2 \leq 2a^2 + 2b^2$, $a,b\in\R$, we get
    \[
        \E_D[u+v] = \E_D(u+v,u+v) \leq 2\left( \E_D[u] + \E_D[v]\right) <\infty,
    \]    
which yields that $\D(\E_D)$ is a linear subspace of $L^2(\R).$

We want to define a norm on the space $\D(\E_D)$. For this purpose, we prove the following lemma.

\begin{lemma}\label{E_D_seminorm}
    A function $p: \D(\E_D)\to [0,\infty)$ defined by $p(u) := \sqrt{\E_D[u]}$ is a seminorm.
\end{lemma}

\begin{proof}
    The absolute homogeneity follows immediately from the definition of the form $\E_D$. We will prove the triangle inequality for $p$, i.e. the inequality $p(u+v) \leq p(u) + p(v)$, $u,v\in\D(\E_D)$. We will use the analogous methods as in the proof of the Minkowski inequality (see Stein and Shakarchi \cite[Theorem 1.2]{Stein_AF}).

     We may assume that $p(u+v)>0$, because if $p(u+v)=0$, then the triangle inequality for $p$ is obvious.  
    From the inequality $|x+y|\leq |x| + |y|$, $x,y\in\R$, and from the H\"older inequality, we get
    \begin{align*}
        2 p^2(u+v) &= \iint_{\R\times\R \setminus D^c\times D^c} (u(x) + v(x) - u(y)-v(y))^2\nu(x,y)\,\dx\,\dy \\
        &\leq \iint_{\R\times\R \setminus D^c\times D^c} |u(x) - u(y)| |u(x) + v(x) - u(y)-v(y)|\nu(x,y)\,\dx\,\dy
        \\
        &+
        \iint_{\R\times\R \setminus D^c\times D^c} |v(x) - v(y)| |u(x) + v(x) - u(y)-v(y)|\nu(x,y)\,\dx\,\dy \\
        &\leq 2p(u) p(u+v) + 2 p(v)p(u+v).
    \end{align*}
    Dividing both sides of the above inequality by $2p(u+v)$ we get a desired triangle inequality for $p$.
\end{proof}

In what follows, we will also consider the inner product $\langle\cdot,\cdot\rangle_{\E_D}$ and the norm $\norm{\cdot}_{\E_D}$ corresponding to the form $\E_D$ defined by the following expressions:
\[
\langle u,v\rangle_{\E_D} := \langle u,v\rangle_{L^2(\R)} + \E_D(u,v), \qquad u,v\in \D(\E_D),
\]
\[
\norm{u}_{\E_D}^2 := \norm{u}_{L^2(\R)}^2 + \E_D[u], \qquad u\in \D(\E_D).
\]

In this section we propose the connection between the Dirichlet form $(\E,\F)$ corresponding to the semigroup $(K_t)_{t\geq 0}$ and the form $(\E_D, \F^*)$, where
\begin{align*}
    \F^* := \Big\{u\in \D(\E_D): ~ \int_\R u^2(x)|x|^{-\alpha}\,\dx <\infty\Big\}.
\end{align*}
The first connection is established in the following proposition, which describes the explicit form of the Dirichlet form $\E$ on the smooth functions with compact support.

\begin{proposition}\label{forms_equal}
    Let $u,v\in C_c^\infty(\R_*)$. Then $u,v\in\D(\E_D)\cap \F$ and 
    \begin{align}\label{eq:form_0}
    \mathcal{E}(u,v) = \mathcal{E}_D(u,v).
    \end{align}
\end{proposition}

\begin{proof}
We will show that $C_c^\infty(\R_*)\subset \D(\E_D)$. Let $u\in C_c^\infty(\R_*)$ and observe that it is obvious that there exist $f\in C_c^\infty(D)$ and $g\in C_c^\infty\big(\overline{D}^c\big)$ such that $u = f+g$. From the inequality $(a+b)^2\leq 2a^2+2b^2$, $a,b\in\R$, it suffices to show that $f,g\in \D(\E_D)$.

\medskip
It is obvious that $f,g\in L^2(\R)$. Note that from Tonelli's theorem,
\begin{align}\label{eq:skoncz_E_D}
\E_D[f] &= \frac{1}{2} \iint_{D\times D} (f(x)-f(y))^2\nu(x,y)\,\dx\,\dy + \int_D f^2(x)\nu(x,D^c)\,\dx.
\end{align}
From \eqref{nuDc_scaling} it follows that 
\begin{align*}
    &\int_D f^2(x)\nu(x,D^c)\,\dx \lesssim \norm{f}_\infty^2 \int_{\supp(f)} |x|^{-\alpha}\,\dx<\infty,
\end{align*}
hence it suffices to show the finiteness of the first integral in \eqref{eq:skoncz_E_D}. From the fact that $f\in C_c^\infty(D)$ and by the symmetry it is obvious that 
\begin{align*}
    &\iint_{D\times D} (f(x)-f(y))^2\nu(x,y)\,\dx\,\dy \\
    &\leq \iint_{(\supp{f}\times D)\cup (D\times \supp{f})} (f(x)-f(y))^2\nu(x,y)\,\dx\,\dy \\
    &\leq 2 \int_{\supp{f}} \dx \int_D \dy \, (f(x)-f(y))^2\nu(x,y) \\
    &\approx \int_{\supp{f}} \dx \int_{D\cap \{|y-x|< 1\}} \dy \, (f(x)-f(y))^2\nu(x,y) \\
    &+ \int_{\supp{f}} \dx \int_{D\cap \{|y-x|\geq  1\}} \dy \, (f(x)-f(y))^2\nu(x,y) =: A+B.
\end{align*}
In case of the integral $A$ we note that from Lagrange mean value theorem $|f(x) - f(y)| = |f'(c)| |x-y|$, where $c$ lies between $x$ and $y$. Moreover, for $|y-x|<1$, $|f'(c)| \leq M$ for some constant $M>0$. Then,
\begin{align*}
    A \lesssim  \int_{\supp{f}} \dx \int_{\{|y-x|< 1\}} \dy \, |x-y|^2\nu(x,y) \approx \int_{\supp{f}} \dx \int_{\{|w|< 1\}}  |w|^{1-\alpha}\,\dw <\infty,
\end{align*}
and
\begin{align*}
    B \lesssim 2\norm{f}^2_\infty  \int_{\supp{f}} \dx \int_{\{|y-x|\geq  1\}} \dy \, |x-y|^{-\alpha-1} \lesssim \int_{\supp{f}} \dx \int_{\{|w|\geq  1\}} \dy \, |w|^{-\alpha-1}<\infty.
\end{align*}
Thus we have proved that $\E_D[f]<\infty$. 

Similarly,
\[
\E_D[g] = \int_{D^c} g^2(x)\nu(x,D)\,\dx \lesssim \norm{g}_\infty^2\int_{\supp(g)} |x|^{-\alpha}\,\dx <\infty.
\]
Hence, $f,g\in\D(\E_D)$.

\medskip
Let $u,v\in C_c^\infty(\R_*)$ and recall that 
    \begin{align}\label{E_D_0}
    \mathcal{E}(u,v) &= \lim_{t\to 0^+} \frac{1}{t}\int_\R \big[ u(x) - K_tu(x)\big] v(x)\,\dx \nonumber \\
    &= \lim_{t\to 0^+} \frac{1}{t}\int_D \big[ u(x) - K_tu(x)\big] v(x)\,\dx + \lim_{t\to 0^+} \frac{1}{t}\int_{D^c} \big[ u(x) - K_tu(x)\big] v(x)\,\dx \nonumber \\
    &=: I+II,
    \end{align}
    provided the limits exist.
    Therefore, without loss of generality, in what follows, we may assume that $t\in (0,1)$.

    \medskip
    Let $x>0$. From Corollary \ref{perturbation_formula},
    \begin{align}\label{E_D_form_1}
        &\frac{1}{t}\big[ u(x) - K_tu(x) \big] \nonumber \\
        &= \frac{1}{t}\big[u(x) - P_t^Du(x)\big] - \frac{1}{t}\int_0^t \dr \int_D \dy \int_{D^c} \dz \, p_r^D(x,y)\nu(y,z)K_{t-r}u(z).
    \end{align}
    First we show that 
    \begin{align}\label{eq:C_D_1}
        \mathcal{C}^D(u,v) &:=\lim_{t\to 0^+} \frac{1}{t} \int_D \big[u(x)-P_t^Du(x)\big]v(x)\,\dx \nonumber \\
        &= \frac{1}{2}\int_D \dx \int_D \dy \, (u(x)-u(y))(v(x)-v(y))\nu(x,y) + \int_D u(x)v(x)\nu(x,D^c)\,\dx.
    \end{align}
    The exact form of the form $\mathcal{C}^D$ of the killed stable process is commonly known (see, e.g.,  Fukushima et al. \cite[Theorem 4.5.2, p. 185]{MR2778606}), but we propose its proof for completion of the argument. 

    Note that $u(x) = u(x)p_t^D(x,D) + u(x)\big(1-p_t^D(x,D)\big)$. Thus,
    \begin{align*}
        &\frac{1}{t} \int_D \big[u(x)-P_t^Du(x)\big]v(x)\,\dx \\
        &= \int_D \dx \int_D \dy \, (u(x)-u(y))v(x)\frac{p_t^D(x,y)}{t} + \int_D u(x)v(x) \frac{1-p_t^D(x,D)}{t}\,\dx \\
        &=: A_t + B_t.
    \end{align*}
From \eqref{eq:approx_p_D} it follows that $p_t^D(x,y)/t \lesssim \nu(x,y)$. Moreover,
from the inequality $ab \leq \tfrac12 (a^2+b^2)$, $a,b\in\R$, we have
\begin{align*}
    &\int_D \dx \int_D \dy \, |u(x)-u(y)| |v(x)| \frac{p_t^D(x,y)}{t} \\
    &\leq \norm{v}_\infty t^{-1} \int_{D\cap \,\supp(v)} \dx \int_D \dy \, (|u(x)|+|u(y)|)p_t^D(x,y) \\
    &\leq \norm{u}_\infty\norm{v}_\infty t^{-1} \left[\int_{D\cap\,\supp(v)} p_t^D(x,D)\,\dx + \int_{D\cap\,\supp(v)} \dx \int_{D\cap\supp(u)} p_t^D(x,y)\,\dy\right] \\
    &\leq \norm{u}_\infty\norm{v}_\infty t^{-1} \big[|\supp(v)| + t^{-1/\alpha} |\supp(v)| \cdot |\supp(u)| \big] <\infty.
\end{align*}

From the Fubini's theorem and from the symmetry of $p^D$,
\begin{align*}
    A_t &= \int_D \dx \int_D \dy \, (u(x)-u(y))v(x)\frac{p_t^D(x,y)}{t} = - \int_D \dx \int_D \dy \, (u(x)-u(y))v(y)\frac{p_t^D(x,y)}{t}.
\end{align*}
Hence,
\begin{align*}
    A_t = \frac{1}{2} \int_D \dx \int_D \dy \, (u(x)-u(y))(v(x)-v(y))\frac{p_t^D(x,y)}{t}.
\end{align*}
Note that from the inequalities $p_t^D(x,y)/t \lesssim \nu(x,y)$ and $ab\leq \tfrac12 (a^2+b^2)$, $a,b\in\R$, and from the fact that
\begin{align*}
    &\int_D \dx \int_D \dy \, |u(x)-u(y)| |v(x)-v(y)| \nu(x,y)  \\
    &\leq \frac{1}{2} \int_D \dx \int_D \dy \, (u(x)-u(y))^2 \nu(x,y) + \frac{1}{2} \int_D \dx \int_D \dy \, (v(x)-v(y))^2\nu(x,y) \\
    &\leq \E_D[u] + \E_D[v]<\infty,
\end{align*}
it follows that we can use the dominated convergence theorem to obtain that 
\begin{align*}
    \lim_{t\to 0^+} A_t &= \frac{1}{2} \int_D \dx \int_D \dy \, (u(x)-u(y))(v(x)-v(y)) \lim_{t\to 0^+} \frac{p_t^D(x,y)}{t} \\
    &= \frac{1}{2} \int_D \dx \int_D \dy \, (u(x)-u(y))(v(x)-v(y))\nu(x,y).
\end{align*}
The latter convergence follows from Lemma \ref{eq:convergence_p_D/t}. 

\medskip
From \eqref{eq:Ikeda_Watanabe}, \eqref{eq:p_survival} and from Corollary \ref{cor:p_t_nu_po_D_i_Dc} it follows that 
\begin{align}\label{eq_B_t_1}
    \frac{1-p_t^D(x,D)}{t} = \frac{1}{t}\int_0^t \ds \int_D \dy \int_{D^c} \dz\, p_s^D(x,y)\nu(y,z) \lesssim |x|^{-\alpha},
\end{align}
and
\begin{align*}
    \int_D |u(x)v(x)| |x|^{-\alpha}\,\dx \leq \norm{u}_\infty \norm{v}_\infty \int_{D\,\cap \,\supp(u)} |x|^{-\alpha}\,\dx <\infty.
\end{align*}
Therefore, again from the dominated convergence theorem, \eqref{eq_B_t_1} and Lemma \ref{lem_zb1} it follows that 
\begin{align*}
    \lim_{t\to 0^+} B_t = \int_D u(x)v(x) \lim_{t\to 0^+} \frac{1-p_t^D(x,D)}{t}\,\dx = \int_D u(x)v(x) \nu(x,D^c)\,\dx.
\end{align*}
Thus, we have proved \eqref{eq:C_D_1}.

\medskip
Further, from the fact that $|K_tu(z)| \leq \norm{u}_\infty$, $t>0$, $z\neq 0$, and from Fubini's theorem,
    \begin{align*}
        C(t,x):\!&= \frac{1}{t}\int_0^t \dr \int_D \dy \int_{D^c} \dz \, p_r^D(x,y)\nu(y,z)K_{t-r}u(z) \\
        &=\int_0^1 \dr \int_{D^c} \dz \Big[\int_D  p_{tr}^D(x,y)\nu(y,z)\,\dy \Big] K_{t(1-r)}u(z) \\
        &=\int_0^1 \dr \int_{D^c} \mathcal{J}(tr,x,z) K_{t(1-r)}u(z)\,\dz.
    \end{align*}
    From Corollary \ref{cor:oszac_p^D_nu} and from \eqref{eq:p_t_approx},
    \begin{align*}
    \mathcal{J}(tr,x,z)|K_{t(1-r)}u(z) |&\lesssim \norm{u}_\infty (tr)^{-1}\Big(1\wedge \frac{|z|}{(tr)^{1/\alpha}}\Big)^{-\alpha/2} \frac{tr}{|x-z|^{\alpha+1}} \\
    &\lesssim \norm{u}_\infty \Big(1\wedge \frac{|z|}{r^{1/\alpha}}\Big)^{-\alpha/2} |x-z|^{-\alpha-1},
    \end{align*}
    and
    \begin{align*}
        &\int_0^1 \dr \int_{D^c}\Big(1\wedge \frac{|z|}{r^{1/\alpha}}\Big)^{-\alpha/2} |x-z|^{-\alpha-1}\,\dz \\
        &= \int_0^1 \dr \int_{-\infty}^{-r^{1/\alpha}} |x-z|^{-\alpha-1}\,\dz + \int_0^1 \dr \int_{-r^{1/\alpha}}^0 \sqrt{r} |z|^{-\alpha/2}|x-z|^{-\alpha-1}\,\dz \\
        &\leq \int_{-\infty}^{0} |x-z|^{-\alpha-1}\,\dz + \int_0^1 \sqrt{r}\dr \int_{-\infty}^0  |z|^{-\alpha/2}|x-z|^{-\alpha-1}\,\dz <\infty.
    \end{align*}
    Therefore, from the dominated convergence theorem, Lemma \ref{lem:zb_nu} and Lemma \ref{K_t_pointwise_convergence} it follows that
    \begin{align}\label{E_D_2}
        \lim_{t\to 0^+} C(t,x)&=\int_0^1 \dr \int_{D^c} \lim_{t\to0^+} \mathcal{J}(tr,x,z) K_{t(1-r)}u(z)\,\dz =\int_{D^c} \nu(x,z)u(z)\,\dz.
    \end{align}
    From Corollary \ref{cor:p_t_nu_po_D_i_Dc} we have
\[
|C(t,x)| \leq \frac{1}{t}\int_0^t \dr \int_D \dy \int_{D^c} \dz \, p_r^D(x,y)\nu(y,z)|K_{t-r}u(z)| \lesssim \norm{u}_\infty |x|^{-\alpha},
\]
    and
\[
\int_D |v(x)| |x|^{-\alpha}\,\dx \leq \norm{v}_\infty \int_{D\,\cap \,\supp{(v)}} |x|^{-\alpha}\,\dx <\infty.
\]
Hence, from the dominated convergence theorem and from \eqref{E_D_2},
\begin{align}\label{E_D_3}
&\lim_{t\to 0^+} \int_D v(x) C(t,x) \,\dx = \int_D v(x)  \lim_{t\to 0^+} C(t,x)\, \dx   = \int_D \dx \int_{D^c} \dy \, u(y)v(x) \nu(x,y).
\end{align}
Combining \eqref{E_D_form_1}, \eqref{eq:C_D_1} and \eqref{E_D_3} we get
\begin{align}
    \label{E_D_4}
    I &= \frac{1}{2}\int_D \dx \int_D \dy \, (u(x)-u(y))(v(x)-v(y))\nu(x,y) \nonumber \\
    &+ \int_D \dx \int_{D^c} \dy \, (u(x)-u(y))v(x)\nu(x,y).
\end{align}

\medskip
Now assume that $x<0$. From Corollary \ref{perturbation_formula} we have
    \begin{align}\label{E_D_5}
        \frac{1}{t}\big[ u(x) - K_tu(x) \big] &= u(x)\frac{1 - e^{-\nu(x,D)t}}{t} - \frac{1}{t} \int_0^t \dr \int_D \dy \, e^{-\nu(x,D)r}\nu(x,y)K_{t-r}u(y).
    \end{align}
    Moreover, from the inequality $1-e^{-x}\leq x$, $x\geq 0$, and from the fact that
    \[
    \int_{D^c} |u(x)v(x)|\nu(x,D)\,\dx  \lesssim \norm{u}_\infty \norm{v}_\infty \int_{D^c \,\cap \,\supp{(u)}} |x|^{-\alpha}\,\dx <\infty,
    \]
    we can use the dominated convergence theorem to obtain that
    \begin{align}\label{E_D_6}
        \lim_{t\to 0^+} \int_{D^c} u(x)v(x)\frac{1 - e^{-\nu(x,D)t}}{t}\,\dx &= \int_{D^c} u(x)v(x)\nu(x,D)\,\dx \nonumber \\
        &= \int_{D^c}\,\dx \int_D \dy \, u(x)v(x) \nu(x,y).
    \end{align}
    Furthermore, 
    \[
    e^{-\nu(x,D)tr}\int_D \nu(x,y)|K_{t(1-r)}u(y)|\,\dy \leq \norm{u}_\infty \nu(x,D),
    \]
    and 
    \[
    \int_{D^c} \dx \int_0^1 \dr \, |v(x)| \nu(x,D) \lesssim \norm{v}_\infty \int_{D^c\,\cap \, \supp{(v)}} |x|^{-\alpha} \,\dx <\infty.
    \]
    Then, from the dominated convergence theorem and from Lemma \ref{K_t_pointwise_convergence},
    \begin{align}\label{E_D_7}
        &\lim_{t\to 0^+} \int_{D^c} v(x) \left[ \frac{1}{t} \int_0^t \dr \int_D \dy \, e^{-\nu(x,D)r}\nu(x,y)K_{t-r}u(y)\right] \,\dx \nonumber \\
        &= \lim_{t\to 0^+} \int_{D^c} \dx \int_0^1 \dr \, v(x) e^{-\nu(x,D)tr} \int_D \nu(x,y)K_{t(1-r)}u(y)\,\dy \nonumber \\
        &=\int_{D^c} \dx \int_0^1 \dr \, v(x) \lim_{t\to 0^+} \int_D \nu(x,y)K_{t(1-r)}u(y)\,\dy \nonumber \\
        &= \int_{D^c} \dx \int_D \dy \, u(y)v(x)\nu(x,y).
    \end{align}
Combining \eqref{E_D_5}, \eqref{E_D_6} and \eqref{E_D_7} we get
\begin{align}\label{E_D_8}
    II &= \int_{D^c}\,\dx \int_D \dy \, (u(x)-u(y))v(x) \nu(x,y).
\end{align}

From \eqref{E_D_0}, \eqref{E_D_4} and \eqref{E_D_8},
\begin{align}\label{E_D_9}
    \mathcal{E}(u,v) &= \frac{1}{2}\int_D \dx \int_D \dy \, (u(x)-u(y))(v(x)-v(y))\nu(x,y) \nonumber \\
    &+ \int_D \dx \int_{D^c} \dy \, (u(x)-u(y))v(x)\nu(x,y) \nonumber \\
    &+ \int_{D^c}\,\dx \int_D \dy \, (u(x)-u(y))v(x) \nu(x,y).
\end{align}
    Note that from the inequality $ab\leq \tfrac12(a^2+b^2)$, $a,b\in\R$, it follows that
    \begin{align*}
        &\int_D \dx \int_{D^c} \dy \, |u(x)-u(y)| |v(x)| \nu(x,y) 
        \\
        &\leq \frac{1}{2} \int_D\dx\int_{D^c}\dy \, (u(x)-u(y))^2\nu(x,y) + \frac{1}{2} \int_D v^2(x)\nu(x,D^c)\,\dx \\
        &\lesssim \mathcal{E}_D[u] + \frac{1}{2} \norm{v}_\infty^2 \int_{D\,\cap\,\supp(v)} |x|^{-\alpha}\,\dx <\infty.
    \end{align*}
    Thus, from the Fubini's theorem and the symmetry of $\nu$ we get
    \begin{align*}
        \int_D \dx \int_{D^c} \dy \, (u(x) - u(y)) v(x)\nu(x,y) &= \int_{D^c} \dy \int_D \dx \, (u(x)-u(y))v(x)\nu(x,y) \\
        &= - \int_{D^c} \dx \int_D \dy \, (u(x)-u(y))v(y)\nu(x,y).
    \end{align*}
    Hence,
    \begin{align}\label{eq:form_2}
        &\int_D \dx \int_{D^c} \dy \, (u(x) - u(y)) v(x)\nu(x,y) \nonumber \\
        &= \frac{1}{2} \int_D \dx \int_{D^c} \dy \, (u(x) - u(y)) v(x)\nu(x,y) - \frac{1}{2} \int_{D^c} \dx \int_D \dy \, (u(x)-u(y))v(y)\nu(x,y).
    \end{align}
    Similarly, we prove that
    \begin{align}\label{eq:form_3}
        &\int_{D^c} \dx \int_D \dy \, (u(x)-u(y))v(x)\nu(x,y) \nonumber \\
        &= \frac{1}{2} \int_{D^c} \dx \int_D \dy \, (u(x)-u(y))v(x)\nu(x,y) - \frac{1}{2} \int_D \dx \int_{D^c} \dy \, (u(x)-u(y))v(y)\nu(x,y).
    \end{align}
    Combining \eqref{E_D_9}, \eqref{eq:form_2} and \eqref{eq:form_3} we get \eqref{eq:form_0}, which completes the proof.  
\end{proof}

The above result gives us explicit form of the Dirichlet form $\E$ only for the smooth functions with compact support. To extend this result for a larger set of functions, we will study equivalent definitions of the domain $\F$.

\begin{lemma}\label{lem_ED_less_E}
    For $u\in L^2(\R)$, $\E_D[u] \leq \E[u]$.
\end{lemma}

\begin{proof}
From Lemma \ref{K_subprobability}, we have \begin{align}\label{eq:char1}
        \mathcal{E}^{(t)}[u] &= \frac{1}{t} \int_\R \big[u(x)K_t(x,\R) - K_tu(x) + u(x)\big(1-K_t(x,\R)\big)\big]u(x)\,\dx \nonumber\\
        &= \frac{1}{t}\int_\R \dx \int_\R K_t(x,\dy) \, (u(x)-u(y))u(x) + \frac{1}{t}\int_\R u^2(x) \big(1-K_t(x,\R)\big)\,\dx \nonumber \\
        &\geq \frac{1}{t}\int_\R \dx \int_\R K_t(x,\dy) \, (u(x)-u(y))u(x).
    \end{align}
Moreover, $\int_\R u^2(x)K_t\mathbf{1}(x)\,\dx <\infty$ and then from Lemma \ref{symmetry_of_K},
\begin{align*}
    \int_\R \dx \int_\R K_t(x,\dy) \, (u(x)-u(y))u(x) &= \int_\R u^2(x)K_t\mathbf{1}(x)\,\dx - \int_\R u(x)K_tu(x)\,\dx \\
    &= \int_\R K_tu^2(x)\,\dx - \int_\R u(x)K_tu(x)\,\dx \\
    &=\int_\R \dx \int_\R K_t(x,\dy) \, (u(y)-u(x))u(y).
\end{align*}
Hence,
\begin{align}\label{eq:char2}
    \int_\R \dx \int_\R K_t(x,\dy) \, (u(x)-u(y))u(x) = \frac{1}{2} \int_\R \dx \int_\R K_t(x,\dy) \, (u(x)-u(y))^2.
\end{align}
Recall that $K_{t,1}(x,D) = 0$ for $x>0$ and $K_{t,1}(x,D^c)=0$ for $x<0$. Combining \eqref{eq:char1} and \eqref{eq:char2} we obtain that
\begin{align*}
\mathcal{E}^{(t)}[u] &\geq \frac{1}{2t} \int_\R \dx \int_\R K_t(x,\dy) \, (u(x)-u(y))^2 \\
&\geq \frac{1}{2t} \int_D \dx \int_D \dy \, p_t^D(x,y) (u(x)-u(y))^2 + \frac{1}{2t} \int_D \dx \int_{D^c} K_{t,1}(x,\dy) (u(x)-u(y))^2 \\
&+ \frac{1}{2t} \int_{D^c} \dx \int_D K_{t,1}(x,\dy) (u(x)-u(y))^2 \\
&= \frac{1}{2t} \int_D \dx \int_D \dy \, p_t^D(x,y) (u(x)-u(y))^2 \\
&+ \frac{1}{2} \int_D \dx \int_{D^c} \dy \, (u(x)-u(y))^2 \, \Big[\frac{1}{t}\int_0^t \dr \int_D \da  \, p_r^D(x,a) \nu(a,y)  e^{-\nu(y,D)(t-r)}\Big]  \\
&+ \frac{1}{2} \int_{D^c} \dx \int_D \dy \, (u(x)-u(y))^2 \Big[ \frac{1}{t} \int_0^t \dr \int_D \db  \, e^{-\nu(x,D)r} \nu(x,b)p_{t-r}^D(b,y)\Big].
\end{align*}
From the Fatou's lemma, Lemma \ref{eq:convergence_p_D/t} and Lemma \ref{lem_zb3},
\begin{align*}
    \E[u] &\geq \frac{1}{2} \int_D \dx \int_D \dy \, (u(x)-u(y))^2\nu(x,y) + \frac{1}{2} \int_D \dx \int_{D^c} \dy \, (u(x)-u(y))^2\nu(x,y) \nonumber\\
    &+ \frac{1}{2} \int_{D^c} \dx \int_D \dy \, (u(x)-u(y))^2\nu(y,x) \nonumber \\
    &=\E_D[u],
\end{align*}
which is our claim.
\end{proof}

For $u\in L^2(\R)$ and $A\subset \R$, we set
\begin{align*}
    u_A(x) := u(x)\ind_A(x) = \begin{cases} u(x), &x\in A, \\ 0, &x\notin A.\end{cases}
\end{align*}
Of course $u_A\in L^2(\R)$ and $u = u_D + u_{\overline{D}^c}$ a.e. 
With such definition, we have the following lemma.

\begin{lemma}\label{funkcjerozdzielone}
Let $u\in\F^*$. Then each of the functions $u_D$ and $u_{\overline{D}^c}$ belongs to $\D(\E_D)$.
\end{lemma}

\begin{proof}
From \eqref{nuDc_scaling} we have
\begin{align*}
    \E_D[u_D] &= \frac{1}{2}\int_D \int_D (u(x)-u(y))^2\nu(x,y)\,\dx\,\dy + \int_D u^2(x)\nu(x,D^c)\,\dx \\
    &\lesssim \E_D[u] + \int_D u^2(x) |x|^{-\alpha}\,\dx <\infty.
\end{align*}
Then similarly,
\begin{align*}
    \E_D[u_{\overline{D}^c}] &= \int_{D^c} u^2(x)\nu(x,D)\,\dx \approx \int_{D^c} u^2(x)|x|^{-\alpha}\,\dx <\infty. \qedhere
\end{align*}
\end{proof}

\begin{lemma}\label{lem_equiv_smooth}
Let $u\in L^2(\R)$. Assume that there exists a sequence $(u_n)\subset C_c^\infty(\R_*)$ such that $\norm{u-u_n}_\F\to 0$ as $n\to\infty$, then $\norm{u-u_n}_{\E_D}\to 0$ as $n\to\infty$. Conversely, if there exists $(u_n)\subset C_c^\infty(\R_*)$ such that $\norm{u-u_n}_{\E_D}\to 0$ as $n\to\infty$, then $\norm{u-u_n}_{\F}\to 0$ as $n\to\infty$. Hence, 
\[
\overline{C_c^\infty(\R_*)}^{\norm{\cdot}_\F} =  \overline{C_c^\infty(\R_*)}^{\norm{\cdot}_{\E_D}}.
\]
\end{lemma}

\begin{proof}
Assume that there exists $(u_n)\subset C_c^\infty(\R_*)$ such that $\norm{u-u_n}_{\F} \to 0$. Then $(u_n)$ is a Cauchy sequence with respect to the norm $\norm{\cdot}_{\F}$, i.e. $\norm{u_n-u_m}_{\F}^2 = \norm{u_n-u_m}_{L^2(\R)}^2 + \E[u_n-u_m]\to 0$ as $n,m\to\infty$. From Proposition \ref{forms_equal} we get that $\norm{u_n-u_m}_{L^2(\R)}^2 + \E_D[u_n-u_m]\to 0$ as $n,m\to\infty$. Thus $\norm{u_n-u_m}_{\E_D} \to 0$, $n,m\to\infty$. From the fact that $\E_D$ is closed and symmetric form (for the proof of the closedness see Voight \cite[Lemma 2.19]{Voight2017}), it follows that there exists $\widehat{u}\in \D(\E_D)$ such that $\norm{\widehat{u}-u_n}_{\E_D}\to 0$ as $n\to\infty$ (see \cite[p. 4]{MR2778606}). We have obtained so far that in particular $\norm{u-u_n}_{L^2(\R)} \to 0$ and $\norm{\widehat{u}-u_n}_{L^2(\R)}\to 0$, $n\to\infty$. Hence, from the uniqueness of the limit in $L^2(\R)$ it follows that $\widehat{u} = u$ a.e. Therefore, $\norm{u - u_n}_{\E_D} = \norm{\widehat{u} - u_n}_{\E_D} \to 0$ as $n\to\infty$.

We know that $\E$ is a Dirichlet form, hence it is closed. Therefore, the second implication follows in the same way.
\end{proof}

\begin{lemma}\label{GestoscDc}
If $u\in\D(\E_D)$ and $u=0$ a.e. on $D$ then there exists a sequence $(u_n)\subset C_c^\infty(\R_*)$ such that $\norm{u-u_n}_{\E_D}\to 0$ as $n\to\infty$.
\end{lemma}

\begin{proof}
Note that
\begin{align}\label{eq:Dc1}
\norm{u}_{\E_D}^2 &=  \int_{D^c} |u(x)|^2 \big( 1 + \nu(x,D)\big)\,\dx = \int_{\overline{D}^c} \Big[ u(x) \sqrt{1+\nu(x,D)}\Big]^2\,\dx.    
\end{align}
Let $f(x) := u(x)\sqrt{1+\nu(x,D)}$ for $x\in \overline{D}^c$. Then from \eqref{eq:Dc1} it follows that $f\in L^2(\overline{D}^c)$. From Lemma 3.1 in \cite[p. 222]{MR2129625} it follows that there exists a sequence $(f_n)\subset C_c^\infty\big(\overline{D}^c\big)$ such that $\norm{f-f_n}_{L^2(\overline{D}^c)}\to 0$ as $n\to\infty$. Therefore, 
\begin{align}\label{eq:Dc2}
    \norm{f-f_n}_{L^2(\overline{D}^c)}^2 &= \int_{D^c} |f(x)-f_n(x)|^2\,\dx = \int_{D^c} \Big[ u(x)\sqrt{1+\nu(x,D)} - f_n(x)\Big]^2\,\dx \nonumber \\
    &= \int_{D^c} \Big[ u(x) - u_n(x)\Big]^2 \big(1+\nu(x,D)\big)\,\dx\to 0,
\end{align}
where $u_n(x) := \frac{f_n(x)}{\sqrt{1+\nu(x,D)}}$ for $x\in\overline{D}^c$. It is an easy  exercise to show that $u_n \in C_c^\infty(\overline{D}^c)$. We extend $u_n$ to $\R_*$ by letting $u_n=0$ on $D$ and from \eqref{eq:Dc1} and \eqref{eq:Dc2} we get
\[
\norm{u-u_n}_{\E_D}^2 = \int_{D^c} \big[u(x) - u_n(x)\big]^2 \big(1+\nu(x,D)\big)\,\dx \to 0,
\]
as $n\to\infty$. 
\end{proof}

\begin{proposition}\label{prop:equality_of_domains}
    For $\alpha\in (0,1)\cup (1,2)$, we have the following equalities between the domains:
    \[
    \F = \F^*  = \overline{C_c^\infty(\R_*)}^{\norm{\cdot}_{\E_D}} = \overline{C_c^\infty(\R_*)}^{\norm{\cdot}_\F}.
    \]
\end{proposition}

\begin{proof}
    Note that the inclusion $\overline{C_c^\infty(\R_*)}^{\norm{\cdot}_\F} \subset \F$ is obvious by the definition. Therefore, from Lemma \ref{lem_equiv_smooth} it suffices to show the following inclusions: $\F \subset \F^* \subset \overline{C_c^\infty(\R_*)}^{\norm{\cdot}_{\E_D}}$.  

    Assume that $u\in\F$, i.e. $u\in L^2(\R)$ and $\E[u]<\infty$. Then from Lemma \ref{lem_ED_less_E}, $\E_D[u] <\infty$. Moreover, from Corollary \ref{cor_Hardy}, $\int_\R u^2(x)|x|^{-\alpha}\,\dx \lesssim \E[u] <\infty$. Hence, $u\in \F^*$.

    \medskip
    Assume that $u\in\F^*$. From Lemma \ref{funkcjerozdzielone} it follows that $u_D$ and $u_{\overline{D}^c}$ belongs to $\D(\E_D)$. Note that 
    \[
    \E_D[u_D] = \frac{1}{2} \iint_{\R\times\R} (u_D(x)-u_D(y))^2\nu(x,y)\,\dx\,\dy,
    \]
    and by Fiscella et al. \cite[Theorem 6]{MR3310082} it follows that there exists a sequence $(f_n)\subset C_c^\infty(\R_*)$ such that $\norm{u_D-f_n}_{L^2(\R)} + \E_D^{1/2}[u_D-f_n]\to 0$ as $n\to\infty$. Therefore, $\norm{u_D-f_n}_{\E_D}\to 0$ as $n\to\infty$. 
    Moreover, from Lemma \ref{GestoscDc} for $u_{\overline{D}^c}$ there exists a sequence $(g_n)\subset C_c^\infty(\R_*)$ such that $\norm{u_{\overline{D}^c} -g_n}_{\E_D}\to 0$ as $n\to\infty$. 
    For $n\in\mathbb{N}$, let $u_n := f_n + g_n$ and note that $u_n\in C_c^\infty(\R_*)$. Furthermore,
    \begin{align*}
        \norm{u-u_n}_{\E_D} = \norm{(u_D - f_n) + (u_{\overline{D}^c} - g_n)}_{\E_D} \leq \norm{u_D - f_n}_{\E_D} + \norm{u_{\overline{D}^c} - g_n}_{\E_D} \to 0,
    \end{align*}
    as $n\to\infty$. Thus $u\in \overline{C_c^\infty(\R_*)}^{\norm{\cdot}_{\E_D}}$. \qedhere  
\end{proof}

For an analogous result to the one given in Proposition \ref{prop:equality_of_domains}, see \cite[Proposition 6.3]{KB_TK_LM_2023}.

In terms of Dirichlet forms, we may then say that the form $(\E,\F)$ is in fact regular. Indeed, by the general theory of Dirichlet forms (see \cite[p. 6]{MR2778606}) it follows that the symmetric form $\E$ is called \emph{regular} if there exists a subset $\mathcal{C}$ of $\F\cap C_c(\R_*)$ such that $\mathcal{C}$ is dense in $\F$ with norm $\norm{\cdot}_\F$ and dense in $C_c(\R_*)$ with uniform norm. The set $\mathcal{C}$ is then called a \emph{core} of the form $\E$. From Proposition \ref{prop:equality_of_domains} it follows immediately that $\mathcal{C} = C_c^\infty(\R_*)$ is a core for $(\E,\F)$. According to this, we get the following corollary.

\begin{corollary}\label{corollary_regular}
    For $\alpha\in (0,1)\cup (1,2)$, the form $(\E,\F)$ is regular.
\end{corollary}

\medskip
The next proposition describes the explicit form of the form $\E$ on the larger class of functions than in the Proposition \ref{forms_equal}.

\begin{proposition}\label{FormEquality_generalver}
For $\alpha\in (0,2)$ and $u\in \overline{C_c^\infty(\R_*)}^{\norm{\cdot}_\F}$,
    \[
    \E[u] = \E_D[u].
    \]
\end{proposition}

\begin{proof}
Let $u\in \overline{C_c^\infty(\R_*)}^{\norm{\cdot}_\F}$. Then there exists a sequence $(u_n) \subset C_c^\infty(\R_*)$ such that $\norm{u-u_n}_\F\to 0$, $n\to\infty$. From \eqref{eq:E_triangle}, \eqref{eq:F_norm} and from Proposition \ref{forms_equal} it follows that 
\begin{align*}
    \sqrt{\E[u]} &= \sqrt{\E[(u-u_n)+u_n]} \leq \sqrt{\E[u-u_n]} + \sqrt{\E[u_n]} \leq \norm{u-u_n}_\F + \sqrt{\E_D[u_n]}.
\end{align*}
By taking $n\to\infty$, we get
\begin{align}\label{eq:nier_E_ED}
    \sqrt{\E[u]} \leq \lim_{n\to\infty} \sqrt{\E_D[u_n]}.
\end{align}
We will show that $\E_D[u_n]\to\E_D[u]$, $n\to\infty$. Indeed, from Lemma \ref{lem_equiv_smooth} it follows that $\norm{u-u_n}_{\E_D}\to 0$ as $n\to\infty$ and then from Lemma \ref{E_D_seminorm}, by the inverse triangle inequality for $\sqrt{\E_D[\cdot]}$, we have
\begin{align}\label{eq:E_D_zbieznosc}
    \big| \sqrt{\E_D[u]} - \sqrt{\E_D[u_n]} \big| \leq \sqrt{\E_D[u-u_n]} \leq \norm{u-u_n}_{\E_D}\to 0,
\end{align}
as $n\to\infty$. Thus, from \eqref{eq:nier_E_ED} and \eqref{eq:E_D_zbieznosc}, we get the inequality $\E[u] \leq \E_D[u]$, which, together with Lemma \ref{lem_ED_less_E}, ends the proof.
\end{proof}

\section{Boundary problems}\label{chap_Neumann}

This section is devoted to finding the direct solutions of two nonlocal boundary problems. For this purpose, we use the \emph{Dynkin characteristic operator}, which is defined in Section \ref{DynkChO}. 
The first problem involves the \emph{$\lambda$-potentials} (or the \emph{resolvents}), see. Corollary \ref{cor:resolvent} below. It is a well-known formula for the Feller generators (see, e.g.,  Dynkin \cite[Theorem 1.1, p. 24]{Dynkin1965}) and we prove that it holds in our case also for the Dynkin characteristic operator.
The second problem is central for the paper---we prove that under certain assumptions on the function $f$, the solution of the Neumann boundary problem 
\begin{align*}\
        \begin{cases}
            \hfill (-\Delta)^{\alpha/2} u &= ~f, \quad \mathrm{in } ~D, \\
            \hfill \mathcal{N}_{\alpha/2} u&= ~f, \quad \,\mathrm{ in } ~\overline{D}^c.
        \end{cases}
\end{align*}
is given by the Green operator $G$ of $K$.

\subsection{Dynkin characteristic operator}\label{DynkChO}

The \emph{Dynkin characteristic} operator for the process $X = (X_t)_{t\geq 0}$ is defined by the following expression
\begin{align}\label{eq:Dynkin_operator}
    \mathcal{D}f(x) := \lim_{r\to 0^+} \frac{\mE_x f(X_{\tau_{B(x,r)}}) - f(x)}{\mE_x\tau_{B(x,r)}}, \quad x\neq 0.
\end{align}
Here $x\in\R_*$ and a function $f$ are such that the limit exists.

Assume that $x>0$ and a function $f$ are such that the limit \eqref{eq:Dynkin_operator} exists. Then, from the construction of the process $X$, from Lemma \ref{lem:L1} and from the Ikeda--Watanabe formula \eqref{eq:Ikeda_Watanabe} it follows that $\mE_xf(X_{\tau_{B(x,r)}}) = \mE^Y_xf(Y_{\tau_{B(x,r)}})$ and $\mE_x\tau_{B(x,r)} = \mE^Y_x\tau_{B(x,r)}$. Hence,
\begin{align*}
    \mathcal{D}f(x) &=  \lim_{r\to 0^+} \frac{\mE^Y_x f(Y_{\tau_{B(x,r)}}) - f(x)}{\mE^Y_x\tau_{B(x,r)}}, \qquad x>0.
\end{align*}
Moreover, from Kwa\'{s}nicki \cite[Lemma 3.3]{MR3613319} it follows that 
\begin{align}\label{eq:Dynkin_singular}
    \mathcal{D}f(x) &= -(-\Delta)^{\alpha/2} f(x),
\qquad x>0,
\end{align}
i.e. the operator $\mathcal{D}$, for $x>0$, is the fractional Laplacian on $\R$.

Now, assume that $x<0$ and $f$ is such that $\widehat{\nu}|f|(x)<\infty.$
Then, from the construction of the process $X$, it is obvious that 
\[
\mE_x f(X_{\tau_{B(x,r)}}) = \int_D f(y) k(x,\dy) = \int_D f(y) \frac{\nu(x,y)}{\nu(x,D)}\,\dy,
\]
and
\[
\mE_x\tau_{B(x,r)} = \frac{1}{\nu(x,D)},
\]
which follows from the fact that for $x<0$, $\tau_{B(x,r)}$ is the random variable from the exponential distribution with mean $1/\nu(x,D)$. Hence,
\begin{align}\label{Dynkin_normal_derivative}
\mathcal{D}f(x) = -\mathcal{N}_{\alpha/2}f(x) = -\int_D (f(x)-f(y))\nu(x,y)\,\dy, \qquad x<0,
\end{align}
i.e. the operator $\mathcal{D}$, for $x<0$, is the \emph{nonlocal normal derivative} (see, e.g.,  Dipierro et al. \cite{MR3651008}). Further, note that from \eqref{Dynkin_normal_derivative}, we also have the following equality:
\begin{align}
    \label{Dynkin_normal_derivative_rozwiniety}
    \mathcal{D}f(x) = \widehat{\nu}f(x) - f(x)\nu(x,D).
\end{align}

\subsection{The \texorpdfstring{$\lambda$}%
     {TEXT}-potentials}

Here assume that $f\in C_b(\R_*)$ and define the \emph{$\lambda$-potential}, $\lambda>0$,  of the semigroup $K=(K_t)_{t\geq 0}$ by
\[
U_\lambda f(x) := \mE_x \int_0^\infty e^{-\lambda t} f(X_t)\,\dt = \int_0^\infty e^{-\lambda t} K_tf(x)\,\dt, \qquad x\neq 0.
\]
Then of course $U_\lambda |f|<\infty$ for all $\lambda>0$. Indeed, for $x\neq 0$, from Lemma \ref{K_subprobability},
\begin{align*}
    U_\lambda |f|(x) &= \int_0^\infty e^{-\lambda t} K_t|f|(x)\,\dt \leq \norm{f}_\infty \int_0^\infty e^{-\lambda t} K_t\mathbf{1}(x)\,\dt \\
    &\leq \norm{f}_\infty \int_0^\infty e^{-\lambda t}\,\dt = \lambda^{-1} \norm{f}_\infty <\infty.
\end{align*}

\begin{proposition}\label{Boundary_problem_1}
    Let $\alpha\in (0,2)$, $\lambda>0$ and $f\in C_b(\R_*)$. Then, $u = U_\lambda f$ is the solution of the equation $(\mathcal{D} -\lambda \mathrm{I}) u = -f$ in $\R_*$.
\end{proposition}

\begin{proof}
For $\lambda >0$ and $f\in C_b(\R_*)$, $u = U_\lambda f$ is well-defined.
From \cite[Theorem 5.1]{Dynkin1965} it follows that
\begin{align*}
    \mE_x \big( e^{-\lambda \tau_{B(x,r)}} u(X_{\tau_{B(x,r)}})\big) = u(x) - \mE_x \int_0^{\tau_{B(x,r)}} e^{-\lambda t} f(X_t)\,\dt.
\end{align*}
Recall that we may use cited theorem, because from Chapter 3.3 in Dynkin \cite{Dynkin1965} it follows that each right-continuous process is strongly measurable.

Hence, for $x\neq 0$,
\begin{align}\label{eq:Dynkin_I-II}
    \mathcal{D}u(x) &= \lim_{r\to 0^+} \frac{\mE_x\big[ (1-e^{-\lambda \tau_{B(x,r)}}) u(X_{\tau_{B(x,r)}})\big]}{\mE_x \tau_{B(x,r)}} - \lim_{r\to 0^+} \frac{1}{\mE_x\tau_{B(x,r)}}\mE_x \int_0^{\tau_{B(x,r)}} e^{-\lambda t} f(X_t)\,\dt \nonumber \\
    &=: I - II.
\end{align}

\medskip
Assume that $x>0$. Since $\tau_{B(x,r)} \leq \zeta^{(1)}$, we have $X_t = Y_t$ for $t<\tau_{B(x,r)}$. Therefore, from Fubini's theorem,
\begin{align*}
    \mE_x\int_0^{\tau_{B(x,r)}} e^{-\lambda t} f(X_t)\,\dt &= \mE^Y_x \int_0^{\tau_{B(x,r)}} e^{-\lambda t} f(Y_t)\,\dt \\
    &= \mE^Y_x \int_0^{\infty} e^{-\lambda t} f(Y_t) \ind_{[t,\infty)}(\tau_{B(x,r)}) \,\dt \\
    &= \int_0^{\infty} \mE^Y_x\big[e^{-\lambda t} f(Y_t) \ind_{[t,\infty)}(\tau_{B(x,r)})\big] \,\dt \\
    &=\int_0^{\infty} \mE^Y_0\big[e^{-\lambda t} f(x+Y_t) \ind_{[t,\infty)}(\tau_{B(0,r)})\big] \,\dt.
\end{align*}
From the scaling property of $(Y_t, \tau_{B(x,r)})$ with respect to $\mP_0^Y$ (see \eqref{eq:pair_scling}) and again from the Fubini's theorem,
\begin{align*}
    \mE^Y_x \int_0^{\tau_{B(x,r)}} e^{-\lambda t} f(Y_t)\,\dt &= \int_0^{\infty} \mE^Y_0\big[e^{-\lambda t} f(x+rY_{r^{-\alpha }t}) \ind_{[t,\infty)}(r^\alpha \tau_{B(0,1)})\big] \,\dt \\
    &= \mE^Y_0 \int_0^{r^\alpha \tau_{B(0,1)}} e^{-\lambda t} f(x+rY_{r^{-\alpha} t}) \,\dt.
\end{align*}
Using the substitution $s=r^{-\alpha}t$ we get
\begin{align*}
    \mE^Y_x \int_0^{\tau_{B(x,r)}} e^{-\lambda t} f(Y_t)\,\dt &= r^\alpha \,\mE^Y_0 \int_0^{\tau_{B(0,1)}} e^{-\lambda sr^\alpha } f(x+rY_s) \,\ds.
\end{align*}
Moreover, from the shift invariance of the process $Y$, $\mE_x^Y \tau_{B(x,r)} = \mE_0^Y \tau_{B(0,r)} = r^\alpha \mE_0^Y \tau_{B(0,1)}$. Hence, from the dominated convergence theorem,
\begin{align}\label{eq:Dynkin_res_II}
    II &= \lim_{r\to 0^+} \frac{1}{\mE^Y_x\tau_{B(x,r)}}\mE^Y_x \int_0^{\tau_{B(x,r)}} e^{-\lambda t} f(Y_t)\,\dt \nonumber \\
    &=\lim_{r\to 0^+} \frac{1}{\mE_0^Y \tau_{B(0,1)}} \mE^Y_0 \int_0^{\tau_{B(0,1)}} e^{-\lambda sr^\alpha} f(x+rY_s) \,\ds = f(x).
\end{align}
In case of the limit $I$ we proceed similarly: from \eqref{eq:scaling_Y_tau},
\begin{align*}
    I &= \lim_{r\to 0^+} \frac{\mE^Y_x\big[ (1-e^{-\lambda \tau_{B(x,r)}}) u(Y_{\tau_{B(x,r)}})\big]}{\mE^Y_x \tau_{B(x,r)}} \\
    &=\lim_{r\to 0^+} \frac{\mE^Y_0\big[ (1-e^{-\lambda r^\alpha \tau_{B(0,1)}}) u(x+rY_{ \tau_{B(0,1)}})\big]}{r^\alpha \,\mE^Y_0 \tau_{B(0,1)}} \\
    &= \frac{1}{\mE_0^Y \tau_{B(0,1)}} \lim_{r\to 0^+} \mE_0^Y \left[ \frac{1-e^{-\lambda r^\alpha \tau_{B(0,1)}}}{r^\alpha} u(x+r Y_{ \tau_{B(0,1)}})\right].
\end{align*}
Note that $u(x) = \int_0^\infty e^{-\lambda t} K_tf(x)\,\dt$ is bounded, which follows from the fact that $f$ is bounded. Moreover, from Proposition \ref{K_bounded_continuity}, it follows that the function $\R_*\ni x\mapsto u(x)$ is continuous. Hence, from the inequality $1-e^{-z}\leq z$ for $z>0$, it follows that we can use the dominated convergence theorem to obtain that 
\begin{align}\label{eq:Dynkin_res_I}
    I = \frac{1}{\mE_0^Y \tau_{B(0,1)}}  \mE_0^Y \left[ \lim_{r\to 0^+} \frac{1-e^{-\lambda r^\alpha \tau_{B(0,1)}}}{r^\alpha} u(x+r Y_{ \tau_{B(0,1)}})\right] = \lambda u(x).
\end{align}
Combining \eqref{eq:Dynkin_I-II}, \eqref{eq:Dynkin_res_II} and \eqref{eq:Dynkin_res_I} we obtain the equality $\mathcal{D}u(x) = \lambda u(x) - f(x)$, $x>0$.    

\medskip
Now assume that $x<0$. From the construction of the process $X$ and from Lemma \ref{lem:L1} it follows that
\begin{align}\label{eq:Dynkin_x_1}
    I &= \frac{\mE_x\big[ (1-e^{-\lambda R_1}) u(X_{R_1})\big]}{\mE_x R_1} \nonumber \\
    &= \nu(x,D) \,\mE_x\big[ (1-e^{-\lambda R_1}) u(X_{R_1})\big] \nonumber \\
    &= \nu(x,D) \int_0^\infty \ds \int_D \dy \, e^{-\nu(x,D)s} \nu(x,y) (1-e^{-\lambda s})u(y) \nonumber \\
    &= \int_0^\infty \nu(x,D)e^{-\nu(x,D)s}(1-e^{-\lambda s})\,\ds \int_D \nu(x,y)u(y)\,\dy \nonumber \\
    &= \frac{\lambda}{\nu(x,D)+\lambda} \widehat{\nu}u(x).
\end{align}
Similarly, 
\begin{align}\label{eq:Dynkin_x_2}
    II &= f(x) \frac{1}{\mE_xR_1} \, \mE_x \int_0^{R_1} e^{-\lambda t}\,\dt  \nonumber \\
    &= f(x) \nu(x,D) \lambda^{-1} \mE_x \big[ 1 - e^{-\lambda R_1}\big] \nonumber \\
    &= f(x) \nu(x,D) \lambda^{-1} \int_0^\infty \nu(x,D)e^{-\nu(x,D)s} \big(1-e^{-\lambda s}\big)\,\ds \nonumber \\
    &= f(x) \frac{\nu(x,D)}{\nu(x,D)+\lambda}.
\end{align}
Combining \eqref{eq:Dynkin_I-II}, \eqref{eq:Dynkin_x_1} and \eqref{eq:Dynkin_x_2} we obtain the equality
\begin{align}
\label{eq:x_neq_rozwinieta}
   \mathcal{D}u(x) = \frac{\lambda}{\nu(x,D)+\lambda} \widehat{\nu}u(x) - f(x) \frac{\nu(x,D)}{\nu(x,D)+\lambda}, \qquad x<0. 
\end{align}
From \eqref{Dynkin_normal_derivative_rozwiniety} the equation \eqref{eq:x_neq_rozwinieta} takes the form
\begin{align}
   \mathcal{D}u(x) &= \frac{\lambda}{\nu(x,D)+\lambda} \big[ \mathcal{D}u(x) + u(x)\nu(x,D) \big] - f(x) \frac{\nu(x,D)}{\nu(x,D)+\lambda},
\end{align}
which is equivalent to the equation $ \mathcal{D}u(x) = \lambda u(x) -f(x)$, $x<0$.
\end{proof}

From Proposition \ref{Boundary_problem_1}, \eqref{eq:Dynkin_singular} and \eqref{Dynkin_normal_derivative} we have the following corollary.

\begin{corollary}\label{cor:resolvent}
    Let $\alpha \in (0,2)$, $\lambda >0$ and $f\in C_b(\R_*)$. Then, $u = U_\lambda f$ is the solution of the following Neumann problem for the fractional Laplacian
    \begin{align*}
        \begin{cases}
            \big[(-\Delta)^{\alpha/2} + \lambda \mathrm{I}\big] u &= f, \quad \mathrm{in }~D, \\
            \phantom{\,\,\,\,\,\,\,\,\,}\big[\mathcal{N}_{\alpha/2} + \lambda \mathrm{I}\big] u &= f, \quad \mathrm{in }~\overline{D}^c.
        \end{cases}
    \end{align*}
\end{corollary}

\subsection{The \texorpdfstring{$0$}%
     {}-potential}

In this section, we assume that $f\in C_c(\R_*)$, i.e. $f$ is a bounded continuous function with compact support. We define the \emph{$0$-potential}  or the \emph{Green operator} of the semigroup $K = (K_t)_{t\geq 0}$ by
\begin{align}
    Gf(x) := \mE_x \int_0^\infty f(X_t)\,\dt = \int_0^\infty K_tf(x)\,\dt, \qquad x\neq 0.
\end{align}

\begin{lemma}\label{green_operator_finitness}
    If $\alpha\in (0,2)$ and $\beta(\alpha-\beta-1)>0$, then $Gh_{\beta-\alpha} \lesssim  h_\beta$ on $\mathbb R^*$.
\end{lemma}

\begin{proof}
    Let $x\neq 0$. Recall that $h_\beta(x) = |x|^\beta$. From Proposition \ref{gen_inequality},
    \begin{align}\label{eq:0potential_1}
        \lim_{t\to 0^+} \frac{h_\beta(y) - K_th_\beta(y)}{t} = \mathcal{A}_{1,\alpha} \mathcal{C}(\alpha,\beta,y) |y|^{\beta-\alpha} \geq \mathcal{C}_1(\alpha,\beta) |y|^{\beta-\alpha}, \qquad y\neq 0,
    \end{align}
where $\mathcal{C}_1(\alpha,\beta) := \mathcal{A}_{1,\alpha} \big[\alpha^{-1} - \mathfrak{B}(\beta+1, \alpha-\beta)\big]$. For  $s,t>0$, from Corollary \ref{excesive_function_beta}, $h_\beta - K_th_\beta \geq 0$ and $K_s\big[h_\beta - K_th_\beta\big](x) \geq 0$. 
Hence, from \eqref{eq:0potential_1} and from Fatou's lemma,
\begin{align}\label{eq:0potential_3}
    G|x|^{\beta-\alpha} &= \int_0^\infty \int_\R K_s(x,\dy)|y|^{\beta-\alpha}\,\ds \nonumber \\
    &\leq \big(\mathcal{C}_1(\alpha,\beta)\big)^{-1} \int_0^\infty  \int_\R K_s(x,\dy) \lim_{t\to 0^+} \frac{h_\beta(y) - K_th_\beta(y)}{t}\,\ds \nonumber \\
    &\leq \big(\mathcal{C}_1(\alpha,\beta)\big)^{-1} \liminf_{t\to 0^+}\int_0^\infty  K_s\left[  \frac{h_\beta - K_th_\beta}{t}\right](x)\,\ds \nonumber \\
    &\approx \liminf_{t\to 0^+} \frac{1}{t} \int_0^\infty \big[ K_sh_\beta(x) - K_{s+t}h_\beta(x)\big]\,\ds \nonumber \\
    &=\liminf_{t\to 0^+} \frac{1}{t} \sum_{i=1}^\infty \int_{(i-1)t}^{it} \big[ K_sh_\beta(x) - K_{s+t}h_\beta(x)\big]\,\ds.
\end{align}
Again, from Corollary \ref{excesive_function_beta}, for $s,t>0$, $K_{t+s}h_\beta(x) = K_s\big(K_th_\beta\big)(x) \leq K_sh_\beta(x)$, hence the function $s\mapsto K_sh_\beta(x)$ is non-increasing. Therefore, from \eqref{eq:0potential_3},
\begin{align}\label{eq:0potential_4}
    G|x|^{\beta-\alpha} &\lesssim \liminf_{t\to 0^+} \frac{1}{t} \sum_{i=1}^\infty \int_{(i-1)t}^{it} \big[ K_{(i-1)t}h_\beta(x) - K_{(i+1)t}h_\beta(x)\big]\,\ds \nonumber \\
    &= \liminf_{t\to 0^+}  \sum_{i=1}^\infty \big[ K_{(i-1)t}h_\beta(x) - K_{(i+1)t}h_\beta(x)\big].
\end{align}
Note that for every $n\in\mathbb{N}$,
\begin{align*}
    h_\beta(x) + K_th_\beta(x) &= \sum_{i=1}^n \big[ K_{(i-1)t}h_\beta(x) - K_{(i+1)t}h_\beta(x)\big] + K_{nt}h_\beta(x) + K_{(n+1)t}h_\beta(x),
\end{align*}
so
\begin{align}\label{eq:0potential_5}
    h_\beta(x) + K_th_\beta(x) \geq \sum_{i=1}^\infty \big[ K_{(i-1)t}h_\beta(x) - K_{(i+1)t}h_\beta(x)\big].
\end{align}
From \eqref{eq:0potential_4}, \eqref{eq:0potential_5} and from Proposition \ref{gen_inequality},
\[
    G|x|^{\beta-\alpha} \lesssim \liminf_{t\to 0^+} \big[ h_\beta(x) + K_th_\beta(x)\big] = 2h_\beta(x) <\infty. \qedhere
\]
\end{proof}

\begin{corollary}\label{G_cor_f}
    Let $x\neq 0$, $\alpha\in (0,2)$ and $\beta(\alpha-\beta-1)>0$. If $f$ is a function such that $|f(x)|\lesssim  |x|^{\beta-\alpha}$, then $G|f|(x)<\infty$. In particular,
    \begin{itemize}
        \item $G(1+|x|)^{-1-\kappa}<\infty$ for every $\kappa>0$,        
        \item if $f$ is a bounded and compactly supported function on $\R_*$, then $G|f|(x)<\infty$.
    \end{itemize}
\end{corollary}

\begin{proof}
We will prove that $G(1+|x|)^{-1-\kappa}<\infty$ for every $\kappa>0$.
We observe that for every $\alpha\in (0,1)\cup (1,2)$ and $\kappa>0$ there exists $\beta$ such that $\beta(\alpha-\beta-1)>0$ and $\beta\geq \alpha-1-\kappa.$ For such $\beta$, we have
\[
  (1+|x|)^{-1-\kappa} \leq (1+|x|)^{\beta-\alpha} \leq |x|^{\beta-\alpha}, \qquad x\neq 0,
\]
and the result follows from Lemma \ref{green_operator_finitness}.
The other statements of the corollary are obvious.
\end{proof}
 
\begin{proposition}\label{Boundary_problem_2}
    Let $\alpha\in (0,1)\cup (1,2)$, $f\in C_c(\R_*)$. Then, $u := G f$ solves   $\mathcal{D} u = -f$ in $\R_*$.
\end{proposition}

\begin{proof}
    From Corollary \ref{G_cor_f}, $u = Gf$ is well-defined. From \cite[Theorem 5.1]{Dynkin1965} it follows that
    \begin{align}\label{Dynikin_2}
    \mE_x u(X_{\tau_{B(x,r)}}) = u(x) - \mE_x \int_0^{\tau_{B(x,r)}} f(X_t)\,\dt.
    \end{align}
    Hence, for $x> 0$,
    \begin{align}
        \mathcal{D} u(x) &= -\lim_{r\to 0^+} \frac{1}{\mE_x\tau_{B(x,r)}} \mE_x \int_0^{\tau_{B(x,r)}} f(X_t)\,\dt \nonumber \\
        &= -\lim_{r\to 0^+} \frac{1}{\mE_x\tau_{B(x,r)}} \left[ \mE_x \int_0^{\tau_{B(x,r)}} \big( f(X_t) - f(x)\big) \,\dt + f(x)\mE_x \tau_{B(x,r)}\right] \nonumber\\
        &= -f(x) - \lim_{r\to 0^+} \frac{1}{\mE_x\tau_{B(x,r)}} \mE_x \int_0^{\tau_{B(x,r)}} \big( f(X_t) - f(x)\big) \,\dt \nonumber\\
        &= -f(x) - \lim_{r\to 0^+} \frac{1}{\mE^Y_x\tau_{B(x,r)}} \mE^Y_x \int_0^{\tau_{B(x,r)}} \big( f(Y_t) - f(x)\big) \,\dt.
    \end{align}
    Note that from Fubini's theorem,
    \begin{align*}
        \mathcal{D} u(x) &= -f(x) - \lim_{r\to 0^+} \frac{1}{\mE^Y_x\tau_{B(x,r)}} \int_0^{\infty} \mE^Y_x \big[\big( f(Y_t) - f(x)\big) \ind_{[t,\infty)}(\tau_{B(x,r)})\big] \,\dt \\
        &= -f(x) - \lim_{r\to 0^+} \frac{1}{ \mE^Y_0\tau_{B(0,r)}} \int_0^{\infty} \mE^Y_0 \big[\big( f(x+Y_t) - f(x)\big) \ind_{[t,\infty)}(\tau_{B(0,r)})\big] \,\dt.
    \end{align*}
    From the scaling property of $(Y_t, \tau_{B(x,r)})$ with respect to $\mP_0^Y$ (see \eqref{eq:pair_scling}), Fubini's theorem and by the substitution $s=r^{-\alpha}t$ we get
    \begin{align*}
        \mathcal{D}u(x) &= -f(x) - \lim_{r\to 0^+} \frac{1}{ r^\alpha \mE^Y_0\tau_{B(0,1)}} \int_0^{\infty} \mE^Y_0 \big[\big( f(x+rY_{r^{-\alpha}t}) - f(x)\big) \ind_{[t,\infty)}(r^\alpha \tau_{B(0,1)})\big] \,\dt \\
        &= -f(x) - \lim_{r\to 0^+} \frac{1}{ r^\alpha \mE^Y_0\tau_{B(0,1)}} \mE^Y_0 \int_0^{r^\alpha \tau_{B(0,1)}} \big( f(x+rY_{r^{-\alpha}t}) - f(x)\big) \,\dt \\
        &= -f(x) - \lim_{r\to 0^+} \frac{1}{\mE^Y_0\tau_{B(0,1)}} \mE^Y_0 \int_0^{\tau_{B(0,1)}} \big( f(x+rY_{s}) - f(x)\big) \,\ds.
    \end{align*}
Hence, from the dominated convergence theorem,
\begin{align*}
\mathcal{D}u(x) = -f(x) - \frac{1}{\mE^Y_0\tau_{B(0,1)}} \mE^Y_0 \int_0^{\tau_{B(0,1)}} \lim_{r\to 0^+} \big( f(x+rY_{s}) - f(x)\big) \,\ds = -f(x).
\end{align*}

Now, consider $x<0$. From \eqref{Dynikin_2} and from the construction of the process $X$, it is obvious that 
\[
    \mathcal{D}u(x) = -\lim_{r\to 0^+} \frac{1}{\mE_x \tau_{B(x,r)}} \mE_x \int_0^{\tau_{B(x,r)}} f(X_t)\,\dt = -f(x). \qedhere
\]
\end{proof}

By Lemma \ref{green_operator_finitness} and the proof of Corollary \ref{cor:K_t_1_h_beta}, for $f\in C_c(\R_*)$, $\widehat{\nu} |Gf|<\infty$. 

\begin{proof}[Proof of Theorem~\ref{Neumann_problem}]
The result follows by Proposition \ref{Boundary_problem_2}, \eqref{eq:Dynkin_singular}, and \eqref{Dynkin_normal_derivative}.
\end{proof}


\end{document}